\documentclass[11pt,a4paper]{article}
\usepackage{hyperref} 
\usepackage{amsmath, amsthm, amssymb, mathrsfs, mathtools}
\usepackage{enumerate, enumitem}
\usepackage{verbatim}
\usepackage{esint}
\usepackage[T1]{fontenc}
\usepackage{comment}

\usepackage{color}
\usepackage{setspace}
\usepackage{graphics}
\usepackage{graphicx}

\numberwithin{equation}{section}

\newtheorem{theorem}{Theorem}[section]
\newtheorem{lemma}[theorem]{Lemma}
\newtheorem{corollary}[theorem]{Corollary}
\newtheorem{proposition}[theorem]{Proposition}

\newtheorem{maintheorem}{Theorem}

\theoremstyle{definition}
\newtheorem{definition}[theorem]{Definition}
\newtheorem{remark}[theorem]{Remark}
\newtheorem{problem}[theorem]{Question}
\newtheorem{question}[theorem]{Question}

\newenvironment{claim}[1]{\par\noindent\textbf{Claim:}\space#1}{}
\newenvironment{claimproof}[1]{\par\noindent\textbf{Proof:}\space#1}{}

\numberwithin{equation}{section}

\newcommand{\bR}{{\mathbb{R}}}
\newcommand{\bH}{{\mathbb{H}}}
\newcommand{\bZ}{{\mathbb{Z}}}

\newcommand{\sF}{\mathcal{F}}
\newcommand{\sN}{\mathcal{N}}

\newcommand{\sH}{\mathcal{H}}
\newcommand{\sD}{\mathcal{D}}

\newcommand{\sP}{\mathcal{P}}
\newcommand{\sC}{\mathcal{C}}

\newcommand{\sB}{\mathcal{B}}

\newcommand{\scD}{\mathscr{D}}
\newcommand{\scT}{\mathscr{T}}
\newcommand{\scH}{\mathscr{H}}
\newcommand{\scF}{\mathscr{F}}
\newcommand{\scQ}{\mathscr{Q}}
\newcommand{\scW}{\mathscr{W}}
\newcommand{\scZ}{\mathscr{Z}}
\newcommand{\scC}{\mathscr{C}}

\newcommand{\scB}{\mathscr{B}}
\newcommand{\scG}{\mathscr{G}}
\newcommand{\scR}{\mathscr{R}}
\newcommand{\scL}{\mathscr{L}}
\newcommand{\scV}{\mathscr{V}}
\newcommand{\scU}{\mathscr{U}}

\newcommand{\diam}{\mathop\mathrm{diam}}

\newcommand{\Net}{\mathop\mathrm{Net}}

\newcommand{\dist}{\mathop\mathrm{dist}}
\newcommand{\ang}{\mathop\mathrm{Angle}}

\newcommand{\face}{\mathop\mathrm{Faces}}

\newcommand{\Layer}{\mathop\mathrm{Layer}}
\newcommand{\Up}{\mathop\mathrm{Up}}

\newcommand{\BWGL}{\mathop\mathrm{BWGL}}

\newcommand{\Graph}{\mathop\mathrm{Graph}}
\newcommand{\Divider}{\mathop\mathrm{Divider}}
\newcommand{\Cover}{\mathop\mathrm{Cover}}
\newcommand{\Stop}{\mathop\mathrm{Stop}}

\newcommand{\Bot}{\mathop\mathrm{Bot}}
\newcommand{\Hp}{\bH^{d+1}}
\newcommand{\R}{\mathbb{R}}
\newcommand{\Z}{\mathbb{Z}}
\newcommand{\N}{\mathbb{N}}
\newcommand{\D}{\mathbb{D}}
\newcommand{\C}{\mathbb{C}}

\title{Lipschitz decompositions of domains with bilaterally flat boundaries}
\date{}
\author{Jared Krandel,\vspace{1ex} \\ Department of Mathematics, University of California Davis, \\ One Shields Ave., Davis CA 94720, USA, \vspace{1ex} \\ \text{jkrandel@ucdavis.edu}}

\begin{document}

\maketitle

\begin{abstract}
We study classes of domains in $\R^{d+1},\ d \geq 2$ with sufficiently flat boundaries that admit a decomposition or covering of bounded overlap by Lipschitz graph domains with controlled total surface area. This study is motivated by the following result proved by Peter Jones as a piece of his proof of the Analyst's Traveling Salesman Theorem in the complex plane: Any simply connected domain $\Omega\subseteq\C$ with finite boundary length $\sH^1(\partial\Omega) < \infty$ can be decomposed into Lipschitz graph domains with total boundary length at most $M\sH^1(\partial\Omega)$ for some $M > 0$ independent of $\Omega$. In this paper, we prove an analogous Lipschitz decomposition result in higher dimensions for domains with Reifenberg flat boundaries satisfying a uniform beta-squared sum bound. We use similar techniques to show that domains with general Reifenberg flat or uniformly rectifiable boundaries admit similar Lipschitz decompositions while allowing the constituent domains to have bounded overlaps rather than be disjoint.
\end{abstract}

\textbf{Mathematics Subject Classification:} 28A75, 28A78, 28A12

\textbf{Funding:} Jared Krandel was partially supported by the National Science Foundation under Grant No. DMS-1763973

\tableofcontents
\section{Introduction}

\subsection{Overview}
In many areas of analysis, general domains that are somehow ``close'' or well-approximated by a \textit{Lipschitz domain} tend to have many desirable properties.
\begin{definition}[Lipschitz domains]
    We say that $\Omega\subseteq\R^{d+1}$ is a \textit{Lipschitz domain} if for each $p\in\partial\Omega$, there exists $r > 0$ such that $B(p,r)\cap\partial\Omega$ is a Lipschitz graph.
\end{definition}
For instance, the idea of finding good Lipschitz domains inside of more general domains has an important place in the study of harmonic measure in the plane and beyond \cite{DJ90}, \cite{Dah77}, \cite{Bad10}, \cite{Azz18}. Lipschitz domains have similarly been used to give characterizations of rectifiability and uniform rectifiability and generally to study the relationship between harmonic measure on a domain and the geometry of the domain's boundary \cite{ABHM19}, \cite{Mo21}, \cite{Mo19}, \cite{ABHM17}, \cite{BH17}, \cite{GMT18}, \cite{AAM19}, \cite{Az16}, \cite{Az21}. The slightly stronger notion of a \textit{Lipschitz graph domain} has also played an important role in quantitative geometric measure theory.
\begin{definition}[Lipschitz graph domains]\label{def:lip-graph-domain}
We say that an open, connected set $\Omega\subseteq \R^{d+1}$ is an \textit{$M$-Lipschitz graph domain} if the following holds: There exists a composition of a translation, dilation, and rotation $A$ with image domain $\widetilde{\Omega} = A(\Omega)$ such that there exists a function $r_{\widetilde{\Omega}}:\mathbb{S}^d\rightarrow\R^+$ with
\begin{equation*}
    \partial\widetilde{\Omega} = \left\{ r_{\widetilde{\Omega}}(\theta)\theta : \theta\in \mathbb{S}^d \right\}
\end{equation*}
and, for any $\theta,\psi\in\mathbb{S}^d$
\begin{align*}
    |r_{\widetilde{\Omega}}(\theta) - r_{\widetilde{\Omega}}(\psi)| \leq M|\theta-\psi|,
\end{align*}
\begin{equation*}
    \frac{1}{1+M} \leq r_{\widetilde{\Omega}}(\theta) \leq 1.
\end{equation*}
\end{definition}
Intuitively, a Lipschitz graph domain is a ``Lipschitz graph over a sphere''. These domains appear in the following striking result due to Peter Jones which is the primary inspiration for this paper:
\begin{theorem}[\cite{Jon90} Theorem 2]\label{pj-decomp}
There exists a constant $M>0$ such that the following holds: For any simply connected domain $\Omega\subseteq \C$ with $\sH^1(\partial\Omega) < \infty$, there is a rectifiable curve $\Gamma$ such that 
$$\Omega\setminus\Gamma = \bigcup_j\Omega_j$$
where $\Omega_j$ is an $M$-Lipschitz graph domain for each $j$, $\Omega_j\cap\Omega_{j'}=\varnothing$ for $j\not= j'$, and
\begin{equation}\label{e:pj-bdy-msr}
    \sum_j\sH^1(\partial\Omega_j) \leq M \sH^1(\partial\Omega).
\end{equation}
\end{theorem}
Also see \cite{GJM92} for a similar result for minimal surfaces in $\R^n$. Note that the Lipschitz constant for the domains in Theorem \ref{pj-decomp} cannot be made small as pointed out by Jones. Indeed, a Lipschitz graph domain with constant $\delta \ll 1$ is a small perturbation of a ball, and such domains cannot tile the unit square while maintaining \eqref{e:pj-bdy-msr}. We informally say that Theorem \ref{pj-decomp} gives a \textit{Lipschitz decomposition} of a domain $\Omega$ in the sense that $\Omega$ is written as a union of closures of disjoint Lipschitz graph domains with boundary lengths controlled by the boundary length of $\Omega$.  Despite being geometrically interesting in and of itself, Theorem \ref{pj-decomp} has an important place in the history of quantitative geometric measure theory because it is central to Jones's original proof of the Analyst's Traveling Salesman Theorem in $\R^2$. This central result gives a characterization of subsets of rectifiable curves and an estimate on their lengths in terms of a quantity called the \textit{Jones beta number} which measures how close a subset $E\subseteq \R^2$ is to being linear locally. 
\begin{definition}[Jones beta number]
Let $E,Q\subseteq \R^2$ where $Q$ has finite diameter. We define the $\beta$-number for $E$ in the ``window'' $Q$ by
\begin{equation*}
    \beta_{E}(Q) = \inf_L\sup_{x\in Q\cap E}\frac{\text{dist}(x,L)}{\diam(Q)},
\end{equation*}
where $L$ ranges over all affine lines in $\R^2$.
\end{definition}
\begin{theorem}[cf. \protect\cite{Jon90} Theorem 1, \cite{Ok92} in $\R^n,\ n > 2$]\label{t:Rn-tst}
Let $E\subseteq \R^2$. $E$ is contained in a rectifiable curve if and only if
\begin{equation*}
    \beta_E^2(\R^2)= \diam(E) + \sum_{Q\in\Delta(\R^2)}\beta_{E}(3Q)^2\diam(Q) < \infty
\end{equation*}
where $\Delta(\R^2)$ is the set of all dyadic cubes in $\R^2$ and $3Q$ is the cube with the same center as $Q$ but three times the side length. If $\Sigma$ is a connected set of shortest length containing $E$, then
\begin{equation}\label{e:Rn-necessary}
    \beta_\Sigma^2(\R^2) \lesssim \sH^1(\Sigma)
\end{equation}
and
\begin{equation}\label{e:Rn-sufficient}
      \beta_E^2(\R^2) \gtrsim \sH^1(\Sigma).
\end{equation}
\end{theorem}
There are now many results referred to as ``traveling salesman theorems'' that share the general structure and philosophy of Theorem \ref{t:Rn-tst} but take place in different spaces such as Hilbert space \cite{Sc07}, Banach spaces \cite{BM22a}, \cite{BM22b}, Carnot groups \cite{Li22}, graph inverse limit spaces \cite{DS16}, and general metric spaces \cite{DS21}, \cite{Ha05}. Many also apply to different geometric objects such as Jordan arcs \cite{Bi22}, H\"{o}lder curves \cite{BNV19}, higher-dimensional sets \cite{AS18}, \cite{Hyd21}, \cite{Hyd22}, \cite{Ghi17}, or measures \cite{BS17}, \cite{BLZ22}. 

Jones proves Theorem \ref{t:Rn-tst} essentially as a corollary of Theorem \ref{pj-decomp}. Roughly speaking, given a rectifiable curve $\Gamma\subseteq \mathbb{D}$, one can apply the Lipschitz decomposition result to each component of $\D\setminus\Gamma$ and use the boundaries of the produced Lipschitz graph domains to control the beta numbers of $\Gamma$. In fact, this shows that rectifiable curves in $\R^2$ admit extensions of controlled length that are quasiconvex: if one considers the union of the boundaries as a new curve $\widetilde{\Gamma} = \cup_j\partial\Omega_j\cup\Gamma$, then $\sH^1(\widetilde{\Gamma}) \lesssim \sH^1(\Gamma)$ and $\widetilde{\Gamma}$ is quasiconvex (see \cite{AS12} for a generalization of this corollary to higher dimensions). 

Jones's result is powerful, but it is confined to two dimensions. In this paper, we consider the following question:
\begin{question}
    For $\Omega\subseteq \R^{d+1},\ d > 1$, what geometric conditions on $\partial\Omega$ are sufficient for $\Omega$ to admit a Lipschitz decomposition?
\end{question}
One of the attractive features of Theorem \ref{pj-decomp} is the minimality of its assumptions on $\Omega$; Jones only assumes simple connectivity and finite boundary length. These assumptions suffice essentially because they give access to a nicely behaved parameterization in the form of a conformal map $\varphi:\D\rightarrow\Omega$. The lack of similar conformal maps in higher dimensions precludes one from directly translating Jones's original argument from $\R^2$ to higher dimensions, but, by assuming stronger control of the geometry of $\partial\Omega$, one does get access to nicely behaved parameterizations that are sufficient replacements. The vital geometric condition on $\partial\Omega$ is called \textit{Reifenberg flatness}, which states that $\partial\Omega$ is bilaterally close to a $d$-plane at all scales and all locations. This bilateral closeness is measured by the \textit{bilateral beta number}.
\begin{definition}[bilateral beta number]
    For $E\subseteq\R^n$, $P$ a $d-$plane, and $B$ a ball, the $d$-\textit{bilateral beta number relative to $P$} for $E$ inside $B$ is
    $$b\beta^d_E(B,P) = \frac{2}{\diam(B)}d_B(E, P)$$
    where $d_B(E,P)$ is the Hausdorff distance between $E$ and $P$ inside $B$ (See \eqref{e:haus-dist}.).
    The full \textit{bilateral beta number} for $E$ inside $B$ is then
    $$b\beta^d_E(B) = \inf_{P\ \text{$d$-plane}}b\beta_E^d(B,P).$$
\end{definition}
\begin{definition}[$(\epsilon,d)$-Reifenberg flatness]
    For fixed $\epsilon > 0$ and $d,n\in\N$ with $0 < d < n$, we say a set $E\subseteq\R^n$ is \textit{$(\epsilon,d)$-Reifenberg flat} if, for all $x\in E$ and $r > 0$,
    $$b\beta_E^d(B(x,r)) \leq \epsilon.$$
\end{definition}
Sets that are $(\epsilon,d)$-Reifenberg flat for small enough $\epsilon \leq \epsilon_0(d,n)$ admit bi-H{\"o}lder parameterizations which we informally call \textit{Reifenberg parameterizations}. This was first shown by Reifenberg in \cite{Re60}, but was later generalized by David and Toro \cite{DT12} to produce parameterizations of Reifenberg flat sets ``with holes'' along with giving a condition under which the parameterization can be upgraded from bi-H{\"o}lder to bi-Lipschitz.
\begin{theorem}[cf. \cite{DT12} Theorem 1.10]\label{t:reif-param-1}
    For any $d,n\in\N$ with $0 < d < n$ and $0 < \tau < \frac{1}{10}$, there exists a constant $\epsilon_0(d,n)$ such that if $\epsilon \leq\epsilon_0$ and $0\in E\subseteq\R^n$ is $(\epsilon,d)$-Reifenberg flat, then there exists a bijection $g:\R^n\rightarrow\R^n$ satisfying the following conditions: For any $z,x,y\in\R^n$ with $z$ arbitrary, $|x-y|\leq 1$,
    $$|g(z) - z| \leq \tau,$$
    $$\frac{1}{4}|x-y|^{1+\tau} \leq |g(x) - g(y)| \leq 3|x-y|^{1-\tau},$$
    and, for some $d$-plane $P$ such that $b\beta_E(B(0,10), P) \leq \epsilon$,
    $$E\cap B(0,1) = g(P) \cap B(0,1).$$
\end{theorem}
Given a domain $\Omega\subseteq\R^{d+1}$ such that $\partial\Omega$ is $(\epsilon,d)$-Reifenberg flat, we use the Reifenberg parameterization $g$ produced by the bi-Lipschitz version of Theorem \ref{t:reif-param-1} (See Theorem \ref{DT-thm}.) as a replacement for the conformal map in Jones's original argument to first prove the following new result.
\begin{maintheorem}\label{t:thmA}
    Let $\Omega\subseteq\R^{d+1}$ be a domain with $0\in\partial\Omega$. There exists $\epsilon_0(d) > 0$ such that for any $L > 0$, if $\epsilon \leq \epsilon_0$ and 
    \begin{enumerate}[label=(\roman*)]
        \item $\partial\Omega$ is $(\epsilon,d)$-Reifenberg flat, \label{i:RF}
        \item $\sum_{k=1}^\infty\beta_{\partial\Omega}^{d,1}(B(x,2^{-k}))^2 \leq L$ for all $x\in\partial\Omega$, \label{item:beta-squared}
    \end{enumerate}
    then there exists an Ahlfors $d$-regular, $d$-rectifiable set $\Sigma$ such that
    $$\Omega\cap B(0,1) \setminus \Sigma = \bigcup_{j=1}^\infty\Omega_j$$
    and there exists $L_1(\epsilon,L,d) > 0$ such that $\scL = \{\Omega_j\}_{j\in J_\scL}$ is a collection of disjoint  $L_1$-Lipschitz graph domains. In addition, for any $y\in \partial\Omega\cap B(0,1)$ and $0 < r < 1$, we have
    $$\sum_{j=1}^\infty \sH^d(\partial\Omega_j\cap B(y,r)) \lesssim_{\epsilon,L,d} r^d.$$
\end{maintheorem}
In condition \ref{item:beta-squared}, $\beta_{\partial\Omega}^{d,1}(B(x,2^{-k}))$ denotes a \text{content beta number} that gives an integral average measurement of the deviation of $\partial\Omega\cap B(x,2^{-k})$ from a $d$-plane using the Hausdorff $d$-content (See Definition \ref{def:betas}). Condition \ref{item:beta-squared} is used to ensure that David and Toro's bi-Lipschitz condition for the Reifenberg parameterization is satisfied. If this hypothesis is not satisfied, then one can still run the proof of Theorem \ref{t:thmA} to produce a collection of Lipschitz graph domains whose total boundary measure and Lipschitz constants blow up near where the sum in \ref{item:beta-squared} diverges. However, we conjecture that a result similar to Theorem \ref{t:thmA} holds without assumption \ref{item:beta-squared}.

If one is willing to weaken the conclusion of $\{\Omega_j\}$ being disjoint to having bounded overlap, then one can show that similar Lipschitz decompositions exist for domains with weaker assumptions on the boundary. We prove the following result of this type:
\begin{maintheorem}\label{t:thmB}
    Let $\Omega\subseteq\R^{d+1}$ be a domain with $0\in\partial\Omega$. There exist constants $A(d), L(d), \epsilon(d) > 0$ such that if $0\in\partial\Omega$ is $(\epsilon,d)$-Reifenberg flat, then there exists a collection of $L$-Lipschitz graph domains $\{\Omega_j\}_{j\in\scL}$ such that
        \begin{enumerate}[label=(\roman*)]
            \item $\Omega_j\subseteq\Omega$, \label{i:om-subset}
            \item $\Omega\cap B(0,1) \subseteq \bigcup_{j=1}^\infty \Omega_j$, \label{i:om-cover}
            \item $\exists C(d) > 0$ such that $\forall x\in\Omega$, $x\in \Omega_j$ for at most $C$ values of $j$, \label{i:om-bounded-overlap}
            \item For any $y\in \partial\Omega\cap B(0,1)$ and $0 < r \leq 1$, we have 
            \begin{equation*}
                \sum_{j=1}^\infty \sH^d(\partial\Omega_j\cap B(y,r)) \lesssim_{\epsilon, d, L} \sH^d(\partial\Omega\cap B(y,Ar)).
            \end{equation*} \label{i:om-surface-measure}
        \end{enumerate}
\end{maintheorem}
To prove this result, we use a collection of $(1+C\delta)$-bi-Lipschitz Reifenberg parameterizations to produce a large collection of disjoint Lipschitz graph domains with controlled boundaries and expand these domains with Whitney-type ``buffer zones'' to form a true covering of $\Omega\cap B(0,1)$. This method carries over to the well-known \textit{uniformly $d$-rectifiable} sets of David and Semmes who give many different equivalent definitions of $d$-uniform rectifiability \cite{DS91}, \cite{DS93}. The standard definition is a set that is Ahlfors $d$-regular and has ``big pieces of Lipschitz images of $\R^d$".
\begin{definition}[uniform $d$-rectifiability]\label{def:ur}
    A set $E\subseteq\R^n$ is \textit{uniformly $d$-rectifiable} if $E$ has \textit{big pieces of Lipschitz images of $\R^d$} (BPLI) and $E$ is \textit{Ahlfors $d$-regular}. By having BPLI, we mean there exist constants $L, \theta > 0$ such that for all $x\in E$ and $0 < r < \diam(E)$, there exists an $L$-Lipschitz map $f:A_{x,r}\subseteq B(0,r)\subseteq\R^d\rightarrow \R^n$ such that 
    \begin{equation}
        \sH^d(E\cap B(x,r)\cap f(A_{x,r})) \geq \theta r^d.
    \end{equation}
    By being Ahlfors $d$-regular, we mean that there exists $C_0 > 0$ such that for all $x\in E$ and $0 < r < \diam(X)$,
    \begin{equation}
        C_0^{-1}r^d \leq \sH^d(E\cap B(x,r)) \leq C_0r^d.
    \end{equation}
\end{definition}
The equivalent characterization of uniform rectifiability that is most relevant to us here involves the \textit{bilateral weak geometric lemma} (BWGL), which roughly says that $E$ looks Reifenberg flat on most scales and locations.
\begin{definition}[bilateral weak geometric lemma]
    Given a family of Christ-David cubes $\mathscr{D}$ for $E$ (see Theorem \ref{t:Christ}) and constants $M,\epsilon > 0$, define
    \begin{equation*}
        \BWGL(M,\epsilon) = \{Q\in\mathscr{D} : b\beta^d_E(MB_Q) > \epsilon \}.
    \end{equation*}
    For $Q\in\scD$, define
    \begin{equation*}
        \BWGL(Q,M,\epsilon) = \sum_{\substack{R\subseteq Q \\ R\in\BWGL(M,\epsilon)}}\ell(R)^d.
    \end{equation*}
    We say that $E$ satisfies the \textit{bilateral weak geometric lemma} if for any $M,\epsilon > 0$, there exists a constant $C_0(M,\epsilon)$ such that for all $Q\in\mathscr{D}$,
    \begin{equation}\label{e:BWGL-Carleson}
        \BWGL(Q,M,\epsilon)\leq C_0\ell(Q)^d.
    \end{equation}
\end{definition}
If $E$ is $(\epsilon,d)$-Reifenberg flat, then $\BWGL(Q,M,\epsilon) = 0$ for all $Q$ and $M$. Equation \eqref{e:BWGL-Carleson} is often referred to as a \textit{Carleson packing condition}. The following deep theorem says that one could equivalently define a uniformly $d$-rectifiable set as an Ahlfors $d$-regular set that satisfies the BWGL.
\begin{theorem}[\cite{DS93} Theorem I.2.4]\label{t:BWGL}
   An Ahlfors $d$-regular set $E\subseteq\R^n$ is \textit{uniformly $d$-rectifiable} if and only if $E$ satisfies the BWGL.
\end{theorem}
Making use of the characterization in Theorem \ref{t:BWGL}, we use similar methods to those of the proof of Theorem \ref{t:thmB} to prove an analogue of Theorem \ref{t:thmB} for uniformly $d$-rectifiable sets.
\begin{maintheorem}\label{t:thmC}
     Let $\Omega\subseteq\R^{d+1}$ be a domain with $0\in\partial\Omega$. If $\partial\Omega$ is uniformly $d$-rectifiable, then there exists $L(d),A(d)>0$ such that there exists a collection of $L$-Lipschitz graph domains $\{\Omega_j\}_{j\in J_\scL}$ such that conclusions \ref{i:om-subset}, \ref{i:om-cover}, \ref{i:om-bounded-overlap}, and \ref{i:om-surface-measure} (with additional dependence on uniform rectifiability constants) of Theorem \ref{t:thmB} hold.
\end{maintheorem}
Uniform rectifiability was studied in detail by David and Semmes in \cite{DS93} where thea authors explore connections between the BWGL and numerous other equivalent definitions involving boundedness of singular integral operators, approximation by Lipschitz graphs (the existence of corona decompositions), ``big piece'' parameterizations by Lipschitz maps, and more. Uniform rectifiability has recently become of interest in the study of harmonic measure and the solvability of the homogeneous Dirichlet problem in rough domains. In \cite{AHMMT20}, the authors give a geometric characterization of open sets $\Omega\subseteq\R^{d+1}$ such that there exists $p < \infty$ such that the $L^p(\partial\Omega)$-Dirichlet problem is solvable given the background hypotheses that $\partial\Omega$ is Ahlfors $d$-regular and $\Omega$ satisfies the interior corkscrew condition. They prove that solvability is equivalent to $\partial\Omega$ being uniformly $d$-rectifiable and $\Omega$ satisfying a quantitative connectivity condition called the weak local John condition. 

A related line of research studies the $L^p$ regularity problem in rough domains. In recent work, Mourgoglu and Tolsa proved a major result in this area by showing that for $\frac{1}{p} + \frac{1}{p'} = 1$, the $L^{p'}$-Dirichlet problem and the $L^p$-regularity problem are equivalent in corkscrew domains with Ahlfors regular boundaries \cite{MT24}. Their construction is based on a corona-type decomposition of the domain into bounded star-shaped Lipschitz domains with disjoint closures and a ``buffer'' region. This decomposition is very similar in style to that used here in the proof of Theorem \ref{t:thmC}: We too have a collection of disjoint ``core'' Lipschitz graph domains and a an added ``buffer'' region, but our ``core'' Lipschitz graph domains often share boundaries so that their closures are not disjoint. The construction in \cite{MT24} was conceived independently from ours and is constructed using ideas more closely related to David and Semmes's original work than the ideas used here involving Reifenberg parameterizations. We also mention that recent related work of Mourgoglu, Poggi, and Tolsa aims to extend the result of \cite{MT24} to elliptic operators satisfying the Dahlberg-Kenig-Pipher (DKP) condition \cite{MPT22}. To do so, they use a refined version of the construction from \cite{MT24}. In this refined version, they must have domains that can be locally represented by Lipschitz graphs with small constants in which the DKP condition holds with small constant. 

In both of the above constructions, the authors use the corkscrew condition, making their setting slightly less general than that of Theorem \ref{t:thmC}. Although we cannot hope to get Lipschitz graph domains with small constant (See the discussion after the statement of Theorem \ref{pj-decomp}.), it seems likely that a refined version of the construction for Theorem \ref{t:thmC} could give domains that are a finite union of Lipschitz graphs with small constant as in \cite{MPT22}, although we do not pursue the details here.

\subsection{Outlines of the paper and of the proofs of the theorems}

In Section \ref{sec:preliminaries}, we introduce the necessary notation and basic facts about Reifenberg parameterizations, Whitney decompositions, Christ-David cubes, coronizations, Reifenberg flat sets, and uniformly rectifiable sets. In Section \ref{sec:lipschitz-graph-domains}, we prove a number of preliminary results on decomposing ``stopping time domains" of Whitney cubes into Lipschitz graph domains and mapping Lipschitz graph domains forward by small perturbations of affine maps. Section \ref{sec:Dg} contains further preliminary results on controlling the change in the derivative of a Reifenberg parameterization above geometrically nice regions. In Sections \ref{sec:thmA} and \ref{sec:thmB-thmC} we prove Theorems \ref{t:thmA} and Theorems \ref{t:thmB} and \ref{t:thmC} respectively.

Roughly speaking, the proof of Theorem \ref{t:thmA} in Section \ref{sec:thmA} proceeds as follows. The fact that $\partial\Omega$ is Reifenberg flat means that we can produce a Reifenberg parameterization $g:\R^{d+1}\rightarrow\R^{d+1}$ such that $g(\R^d\times\{0\})\cap B(0,1) \supseteq \partial\Omega\cap B(0,1)$. The uniform bound on the beta-squared sum in condition \ref{item:beta-squared} of Theorem \ref{t:thmA} ensures that $g$ is $L'(d,L)$-bi-Lipschitz so that $\partial\Omega$ is in fact a bi-Lipschitz image, hence uniformly rectifiable. This means that there exists a Christ-David lattice $\scD$ for $\partial\Omega$ with a graph coronization whose stopping time regions $\scF=\{S\}$ consist of cubes well-approximated by Lipschitz graphs with Lipschitz constant small in terms of $d$ and $L'$ (this is a coronization that produces a corona decomposition). Proposition \ref{p:DgDginv-expanded} implies that $Dg$ is nearly constant on parts of its domain which are mapped into regions of $\Omega$ sitting ``above'' a stopping time region $S$ on the scale of the cubes inside $S$. By the results of Section \ref{sec:lipschitz-graph-domains}, $g$ maps forward Lipschitz graph domains to Lipschitz graph domains when the change in $Dg$ is small compared to the Lipschitz constants of the mapped domains. Therefore, to produce a Lipschitz decomposition of $\Omega\cap B(0,1)$, it suffices to produce a Lipschitz decomposition $\scL_0$ (see Definition \ref{def:domain-decomp}) of the domain of $g$ into domains over which $Dg$ is nearly constant so that the collection of images $\scL = \{g(\sD) : \sD\in\scL_0\}$ is a Lipschitz decomposition.

In order to form this decomposition, we produce a ``coronization'' of a lattice of Whitney boxes which parallels the coronization for $\scD$ on $\partial\Omega$ (see \ref{def:whitney-coronization}). That is, we separate Whitney boxes into bad boxes that $g$ maps near bad cubes in $\scD\cap\scB$ or cubes on the ``edges'' in scale and location of stopping time regions in $\scD$. This decomposition then maps forward to a collection of domains whose total surface measure is bounded by the surface measure of $\partial\Omega$ plus the Carleson packing sums for the bad and ``edge'' cubes of $\scD$. 

The proofs of Theorems \ref{t:thmB} and \ref{t:thmC} both follow a single similar argument to that of Theorem \ref{t:thmA}. In the Reifenberg flat case, the difference is that any single global Reifenberg parameterization $g$ produced for the set is not in general bi-Lipschitz, so we have no uniform estimates on how $g$ distorts any given cube. In the uniformly rectifiable case, we have no global Reifenberg parameterization because there can be many scales and locations at which Reifenberg flatness fails. In either case, we sidestep these by producing a collection of local $(1+\delta)$-bi-Lipschitz parameterizations by parameterizing pieces of the domain above stopping time regions in a graph coronization (see Definition \ref{def:graph-coronization}) using single stopping time domains composed of Whitney cubes. By similar arguments, the surface measure of these domains is controlled by the surface measure of $\partial\Omega\cap B(0,1)$ plus the Carleson packing sum of the same bad set of cubes in $\scD\cap\scB$ and near ``edges'' of stopping time domains. We then fill parts of $\Omega\cap B(0,1)$ that are missed by these domains with ``buffer zones'' of cubes on the exterior of these domains as well as families of cubes which sit above surface cubes in the bad set. By similar reasoning, the surface measure of these domains is bounded by the same Carleson packing sums as above.

\section{Preliminaries}\label{sec:preliminaries}

\subsection{Conventions and basic definitions}
Whenever we write $A \lesssim B$, we mean that there exists some constant $C$ independent of $A$ and $B$ such that $A \leq CB$. If we write $ A \lesssim_{a,b,c} B$ for some constants $a,b,c$, then we mean that the implicit constant $C$ mentioned above is allowed to depend on $a,b,c$. We will sometimes write $A\simeq B$ to mean that both $A\lesssim B$ and $B\lesssim A$ hold.

In many computations, we use a constant $C$ to denote a catch-all, general constant which is allowed to vary significantly from one line to the next.
\begin{definition}[Hausdorff measure, Hausdorff distance, Nets]
    For $F,E\subseteq \R^{d+1}$ and $a\in\R^{d+1}$, we let
\begin{align*}
    \dist(E,F) &= \inf\{|x-y| : x\in F,\ y\in E\},\\
    \dist(a,E) &= \dist(\{a\},E)
\end{align*}
and define
\begin{equation*}
    \diam(F) = \sup\{|x-y| : x,y\in F\}.
\end{equation*}
For any $r > 0$, we let
\begin{equation*}
    B(E,r) = \{x\in\R^{d+1} : \dist(x,E) < r\}.
\end{equation*}

For any subset $F\subseteq\R^{d+1}$, an integer $m\geq 0$, and constant $0 < \delta \leq \infty$, we define
\begin{equation*}
    \sH^m(F) = \inf\left\{\sum \diam(E_i)^m : F\subseteq\bigcup E_i, \diam(E_i) < \delta\right\}.
\end{equation*}
The Hausdorff $m$-measure of $F$ is defined as
\begin{equation*}
    \sH^m(F) = \lim_{\delta\rightarrow 0}\sH^m_\delta(F).
\end{equation*}
We will only use this in the case $m=d$. We refer to the function $\sH^m_\infty$ as the $m$-dimensional Hausdorff content. Given two closed sets $E,F\subseteq\R^{d+1}$, and a third set $B\subseteq \R^{d+1}$ we define the Hausdorff distance between $E$ and $F$ inside $B$ as
\begin{equation}\label{e:haus-dist}
    d_B(E,F) = \frac{2}{\diam B}\max\left\{ \sup_{y\in E\cap B}\dist(y,F),\ \sup_{y\in F\cap B}\dist(y,E) \right\}.
\end{equation}
When $x\in\R^{d+1}$ and $r > 0$, we also define
\begin{equation*}
    d_{x,r}(E,F) = d_{B(x,r)}(E,F).
\end{equation*}
Given a subset $E\subseteq\R^{d+1}$ and $r > 0$, we let $\Net(E,r)$ denote the set of $r$-nets of $E$. That is, $X\in\Net(E,r)$ if $X\subseteq E$ is such that both
\begin{enumerate}
    \item[(i)] For any $x\not=y\in X$, $|x - y| \geq r$,
    \item[(ii)] $E\subseteq \bigcup_{x\in X} B(x,r)$.
\end{enumerate}

\end{definition}

\subsection{Reifenberg parameterizations}
In this section, we record the basic facts about Reifenberg parameterizations needed from \cite{DT12}.

\subsubsection{Coherent Collections of Balls and Planes (CCBP)}
Set $r_k = 10^{-k}$ and let $x_{j,k}\in\R^{d+1},\ j\in J_k$ satisfy
\begin{equation}\label{DT-net}
    |x_{j,k} - x_{i,k}| \geq r_k.
\end{equation}
Put $B_{j,k} = B(x_{j,k},r_k)$ and for $\lambda > 0$ define $V_{k}^\lambda = \bigcup_{j\in J_k}\lambda B_{j,k} = \bigcup_{j\in J_k}B(x_{j,k},\lambda r_k)$ where $\lambda B$ is always the ball with the same center as $B$ and radius dilated by a factor of $\lambda$. We also assume
\begin{equation}\label{DT-net-co}
    x_{j,k}\in V_{k-1}^2.
\end{equation}
We will always use a $d$-plane as the initial surface $\Sigma_0$. We require
\begin{equation}\label{DT-sig}
    \dist(x_{j,0},\Sigma_0) \leq \epsilon \text{ for $j\in J_0$}.
\end{equation}
Finally, the coherent collection of planes is a collection of planes (in general of any dimension $m < d+1$, although here we only take $d$-planes) $P_{j,k}$ associated to $x_{j,k}$ such that the following compatibility conditions are satisfied:
\begin{align}\label{DT-co1}
    d_{x_{j,k},100r_k}(P_{i,k},P_{j,k}) \leq \epsilon
    \ \hbox{ for $k \geq 0$ and $i,j \in J_k$ such that }
    |x_{i,k}-x_{j,k}| \leq 100 r_k,
\end{align}
\begin{equation}\label{DT-co2}
    d_{x_{i,0},100}(P_{i,0},P_{x}) \leq \varepsilon
    \ \hbox{ for $i \in J_0$ and $x\in \Sigma_0$ such that }
    |x_{i,0}-x| \leq 2, 
\end{equation}
and
\begin{equation}\label{DT-co3}
    d_{x_{i,k},20r_k}(P_{i,k},P_{j,k+1}) \leq \varepsilon
    \ \hbox{ for $i \in J_k$ and $j\in J_{k+1}$ such that }
    |x_{i,k}-x_{j,k+1}| \leq 2 r_k.
\end{equation}
With these conditions, we can define a CCBP.
\begin{definition}
A CCBP is a triple $(\Sigma_0,\{B_{j,k}\},\{P_{j,k}\})$ such that conditions \eqref{DT-net}, \eqref{DT-net-co}, \eqref{DT-sig}, \eqref{DT-co1}, \eqref{DT-co2}, and \eqref{DT-co3} are satisfied with $\epsilon$ sufficiently small in terms of $d$. 
\end{definition}
Given a family of planes $\{P_{i,k}\}_{i\in J_k}$, we can precisely measure the deviation between planes on scales $k$ and $k-1$ near a given point $y\in\R^{d+1}$ using the following quantity. Whenever $y\in V_k^{10}$, we define
\begin{align}\label{e:epsilonk'}
    \epsilon'_k(y) = \sup\big\{ &
    d_{x_{i,l},100r_{l}}(P_{j,k},P_{i,l}) \, ; \,
    j\in J_k, \, l \in \{ k-1, k \}, \\\nonumber
    & \hskip 3.5cm
    i\in J_{l}, \hbox{ and } y \in 10 B_{j,k} \cap 11 B_{i,l} 
    \big\}.
\end{align}
When $y\not\in V_k^{10}$, we set $\epsilon_k'(y) = 0$. We now give a small modification of a lemma in \cite{AS18} which gives criteria for a triple $(\Sigma_0, \{B_{j,k}\},\{P_{j,k}\})$ to be a CCBP using these numbers.
\begin{lemma}[cf. \cite{AS18} Theorem 2.5]\label{l:CCBP-AS}
For any $k\in\N\cup\{0\}$, let $r_k = 10^{-k}$. Let $\{x_{j,k}\}_{j\in J_k}$ be a collection of points such that for some $d$-plane $P_0$ we have
\begin{equation*}
    \dist(x_{j,0},P_0) < \epsilon,
\end{equation*}
\begin{equation*}
    |x_{j,k} - x_{i,k}| \geq r_k,
\end{equation*}
and, with $B_{j,k} = B(x_{j,k},r_k)$,
\begin{equation}\label{e:CCBP-tree}
    x_{i,k}\in V_{k-1}^2
\end{equation}
where
\begin{equation*}
    V_k^\lambda = \bigcup_{j\in J_k}\lambda B_{j,k}.
\end{equation*}
Let $P_{j,k}$ be a $d$-plane such that $x_{j,k}\in P_{j,k}$. There is $\epsilon_0 > 0$ such that for any $0 < \epsilon < \epsilon_0$, if 
\begin{equation*}
    \epsilon_k'(x_{j,k}) \lesssim \epsilon \text{ for all $k \geq 0$ and $j\in J_k$}
\end{equation*}
then $(P_0, \{B_{j,k}\},\{P_{j,k}\})$ is a CCBP.
\end{lemma}

CCBPs allow the construction of Reifenberg parameterizations which we will denote by the letter $g$. David and Toro give the following Theorem
\begin{theorem}[\cite{DT12} Theorems 2.15, 2.23]\label{DT-thm}
Let $(\Sigma_0,\{B_{j,k}\},\{P_{j,k}\})$ be a CCBP with $\epsilon$ sufficiently small. There exists a bijection $g:\R^n\rightarrow\R^n$ such that
\begin{equation}\label{DT-far-id}
    g(z) = z \ \hbox{ when } \dist(z,\Sigma_0) \geq 2,
\end{equation}
\begin{equation}
    |g(z)-z| \leq C \varepsilon
\ \hbox{ for } z\in \R^n,
\end{equation}
and
\begin{equation}
    {1 \over 4} |z'-z|^{1+C\varepsilon} \leq |g(z')-g(z)| 
\leq 3 |z'-z|^{1-C\varepsilon}
\end{equation}
for $z,z'\in \R^n$ such that $|z'-z| \leq 1$,
and $\Sigma = g(\Sigma_0)$ is a $C\varepsilon$-Reifenberg
flat set that contains the accumulation set
\begin{align*}
    E_\infty &= \big\{ x\in \R^n \, ; \, 
    x \hbox{ can be written as }
    x = \lim_{m\to +\infty} x_{j(m),k(m)}, 
    \hbox{ with } k(m) \in \mathbb{N} \\
    &\hskip 2.5cm\hbox{ and } 
    j(m) \in J_{k(m)} \hbox{ for } m \geq 0,  \hbox{ and }
    \lim_{m \to  +\infty} k(m) = +\infty \big\}.
\end{align*}
If in addition there exists $L>0$ such that
\begin{equation}\label{e:eps-sq}
    \sum_{k\geq0}\epsilon'_k(f_k(z))^2 \leq L\ \text{ for all $z\in\Sigma_0$},
\end{equation}
then $g$ is bi-Lipschitz: There is a constant $C(n,d,L) \geq 1$ such that for any $z,z'\in\R^n$
\begin{equation*}
    C(n,d,L)^{-1}|z-z'| \leq |g(z) - g(z')| \leq C(n,d,L)|z-z'|.
\end{equation*}
\end{theorem}

\subsubsection{The definition of $g$}
Following Chapter 3 of \cite{DT12}, we take $\psi_k$ to be a smooth function vanishing outside $V_k^8$ and $\theta_{j,k}$ to be a collection of smooth compactly supported functions in $10B_{j,k}$ such that $|\nabla^m\theta_{j,k}(y)|\leq C_mr_k^{-m}$ and $\psi_k(y) + \sum_{j\in J_k}\theta_{j,k}(y) = 1$. We then define a sequence of maps $f_k$ by
\begin{equation*}
    f_0(y) = y,\ f_{k+1} = \sigma_k\circ f_k
\end{equation*}
where
\begin{equation*}
    \sigma_k(y) =  y + \sum_{j \in J_k} \theta_{j,k}(y) \, [\pi_{j,k}(y)-y]
=\psi_k(y) y + \sum_{j \in J_k}\theta_{j,k}(y) \, \pi_{j,k}(y)
\end{equation*}
where $\pi_{j,k}$ is orthogonal projection onto $P_{j,k}$. In our application, we will usually only care about points inside $V_k^8$. In this case, $\psi_k(y) = 0$ and the formula simplifies to 
\begin{equation*}
    \sigma_k(y) = \sum_{j \in J_k}\theta_{j,k}(y) \, \pi_{j,k}(y).
\end{equation*}
Equation (6.8) in \cite{DT12} says that the map $\sigma_k$ also satisfies
\begin{equation}\label{DT-sigma-id}
    |\sigma_k(y) - y| \leq C\epsilon r_k
\end{equation}
for $k\geq 0$ and $y\in\Sigma_k$. It is not hard to show by iterating \eqref{DT-sigma-id} and using the triangle inequality that the following holds: for any $k\geq n \geq 0$ and $x\in\Sigma_0$,
\begin{equation}\label{e:fn-diff}
    |f_k(x) - f_n(x)| \leq C\epsilon r_n
\end{equation}
where $C$ is independent of $k$. In particular, this implies that the limiting mapping
\begin{equation}
    f(x) = \lim_{k}f_k(x)
\end{equation}
exists and is continuous. In fact, $f$ is also H{\"o}lder continuous. See \cite{DT12} Proposition 8.1. Taking the limit in $k$ in \eqref{e:fn-diff} then gives
\begin{equation}\label{e:f-diff}
    |f(x)-f_n(x)| \leq C\epsilon r_n.
\end{equation}

The map $g$ is constructed by interpolating between adjacent maps in the sequence $f_k$ at distance $r_k$ from the surface $\Sigma_k$ defined by
\begin{equation}\label{e:sigmak}
    \Sigma_k = f_k(\Sigma_0).
\end{equation}
As long as $\Sigma_0$ is $C^2$ (Recall we will always choose $\Sigma_0$ to be a plane, so this holds for us.), the fact that $\sigma_k$ is smooth for every $k$ implies that $\Sigma_k$ is a $C^2$ submanifold of $\R^{d+1}$. Therefore, $\Sigma_k$ has a well-defined affine tangent plane at every $z\in\Sigma_k$ that we denote by $T\Sigma_k(z)$. This also allows us to define a field of linear isometries $R_k$ on $\Sigma_0$ for each $k$ such that $R_k$ maps each $x\in\Sigma_0$ to an isometry that maps the tangent plane to $\Sigma_0$ at $x$ to the tangent plane to $\Sigma_k$ at $f_k(x)$. The following proposition summarizes the properties of $R_k$ that we need
\begin{proposition}[\cite{DT12} Proposition 9.29]\label{DT-isometries}
Let $\mathcal{R}$ denote the set of linear isometries of $\R^n$. Set
\begin{equation*}
    T_k(x) = T\Sigma_k(f_k(x))
    \ \hbox{ for } x \in \Sigma_0
    \hbox{ and } k \geq 0.
\end{equation*}
There exist $C^1$ mappings $R_k : \Sigma_0 \to {\mathcal{R}}$, with the
following properties:
\begin{equation*}
    R_0(x) = I \hbox{ for } x \in \Sigma_0,
\end{equation*}
\begin{equation*}
    R_k(x)(T_0(x))=T_k(x)
    \ \hbox{ for } x \in \Sigma_0
    \hbox{ and } k \geq 0,
\end{equation*}
and
\begin{equation}\label{DT-isometry-change}
    |R_{k+1}(x)-R_{k}(x)| \leq C \varepsilon
    \ \hbox{ for } x \in \Sigma_0
    \hbox{ and } k \geq 0.
\end{equation}
\end{proposition}
In addition, we record bounds on the distance between consecutive generations of tangent planes and between planes at different locations. This will use the following distance between $d$-dimensional vector subspaces $V,V'$ of $\R^{d+1}$:
\begin{equation}
    D(V,V') = \max\left\{ \sup_{v\in V\cap B(0,1)}\dist(v,V'),\ \sup_{v'\in V'\cap B(0,1)}\dist(v',V) \right\}.
\end{equation}
This distance also extends to affine $d$-plane $P,P'$ by applying it to translations of $P$ and $P'$ that pass through the origin.
\begin{lemma}[\cite{DT12} Lemma 9.2]\label{DT-plane-dist}
For $k \geq 0$ and $x,x'\in \Sigma_0$,
such that $|x'-x| \leq 10$,
\begin{equation*}
    D(T\Sigma_{k+1}(f_{k+1}(x)),T\Sigma_k(f_k(x)))
    \leq C_1 \varepsilon,
\end{equation*}
and
\begin{equation*}
    D(T\Sigma_{k}(f_{k}(x')),T\Sigma_k(f_k(x)))
\leq C_2 \varepsilon \, r_k^{-1} |f_{k}(x')-f_{k}(x)|.
\end{equation*}
\end{lemma}
\begin{remark}\label{rem:tangent-plane-angles}
    The condition $|x-x_0| \leq 10$ above is technically used to control the angle between $T\Sigma_0(x)$ and $T\Sigma_0(x')$. When $\Sigma_0$ is a plane, this angle is $0$ so that it is easy to see the statement holds for any $x,x'\in\Sigma_0$ without extra work.
\end{remark}
Now, following Chapter 10 in \cite{DT12}, we define a collection $\rho_k$ of positive, smooth, radial functions such that $\sum_{k\geq 0}\rho_k(y) = 1$ for $y\in\R^n\setminus\{0\}$ and $\rho_k(y) = 0$ unless $r_k < |y| < 20r_k$. Because $[r_k,20r_k]\cap[r_{k-2},20r_{k-2}] = [r_k,20r_k]\cap[100r_{k},2000r_{k}] = \varnothing$, we always have at most two values of $k$ such that $\rho_k(y)\not=0$ for any fixed $y$. In order to single out specific values of $k$, we define functions $l,n:\R^+\rightarrow\N$ by 
\begin{align}
    l(y) &= \min\{k\in\N : \rho_k(y) > 0\},\label{e:def-l}\\
    n(y) &= \max\{k\in\N : \rho_k(y) > 0\} = l(y) + 1.\label{e:def-g}
\end{align}
More concretely, we have
\begin{equation}\label{e:g-char}
    n(y) = n \iff 20r_{n+1} = 2r_n < y \leq 20r_n
\end{equation}
because then $\rho_{n+1}(y) = 0$ while $\rho_n(y) > 0$. Roughly speaking, $n(y)$ gives the index such that $f_{n(y)}$ is most relevant for the behavior of $g$ on points roughly distance $|y|$ from $\Sigma_0$. We will now assume $\Sigma_0 = \R^d$ and define
\begin{equation*}
    g(z) = \sum_{k\geq 0}\rho_k(y)\left\{ f_k(x) + R_k(x)\cdot y \right\}\ \hbox{ for } z=(x,y)
\end{equation*}
when $y\not=0$ and
\begin{equation*}
    g(z) = f(z)
\end{equation*}
when $y = 0$. We will commonly use the notation $z = (x,y)$ as understood above when discussing points in the domain of $g$.
\begin{remark}\label{rem:f-bilip}
    The fact that $g$ is bi-Lipschitz when \eqref{e:eps-sq} is satisfied is a consequence of the fact that $f$ (and $f_n$ for all $n$) is similarly bi-Lipschitz when $\eqref{e:eps-sq}$ holds. This is proven in Proposition 8.34 in \cite{DT12}.
\end{remark}

\subsection{Coronizations for Reifenberg flat and uniformly rectifiable sets}

The boundary measure bounds for our Lipschitz decompositions come from Carleson packing conditions for well-chosen \textit{coronizations} of a \textit{Christ-David lattice for $\partial\Omega$}. A Christ-David lattice is best thought of as a family of intrinsic dyadic cubes for a doubling metric space. These were originally devised by Christ in \cite{Chr90}, but the formulation given here is due to Hytonen and Martikainen from \cite{HM12}. 

\begin{theorem}[Christ-David cubes]
Let $X$ be a doubling metric space. Let $X_{k}$ be a nested sequence of maximal $\rho^{k}$-nets for $X$ where $\rho<1/1000$ and let $c_{0}=1/500$. For each $k\in\bZ$ there is a collection $\mathscr{D}_{k}$ of ``cubes,'' which are Borel subsets of $X$ such that the following hold.
\begin{enumerate}[label=(\roman*)]
\item $X=\bigcup_{Q\in \mathscr{D}_{k}}Q$,
\item If $Q,Q'\in \mathscr{D}=\bigcup \mathscr{D}_{k}$ and $Q\cap Q'\neq\emptyset$, then $Q\subseteq Q'$ or $Q'\subseteq Q$,
\item For $Q\in \mathscr{D}$, let $k(Q)$ be the unique integer so that $Q\in \mathscr{D}_{k}$ and set $\ell(Q)=5\rho^{k(Q)}$. There is $x_{Q}\in X_{k}$ so that
\begin{equation*}
B_X(x_{Q},c_{0}\ell(Q) )\subseteq Q\subseteq B_X(x_{Q},\ell(Q))
\end{equation*}
and
\[ X_{k}=\{x_{Q}: Q\in \mathscr{D}_{k}\}.\]
\end{enumerate}
If in addition we assume $X\subseteq\R^{d+1}$ and that $X$ is Ahlfors $d$-regular, then these cubes also satisfy
\begin{enumerate}[label=(\roman*)]
    \setcounter{enumi}{3}
    \item $|Q| \simeq_d (\diam Q)^d \simeq_d \ell(Q)^d$.
\end{enumerate}
\label{t:Christ}
\end{theorem}
For any $Q\in\scD$, we will use the notation $Q^{(1)}$ to refer to the parent of $Q$. For any $\lambda \geq 1$, we also define
\begin{equation*}
    \lambda Q = \{x\in X : \dist(x,Q) \leq (\lambda-1)\diam(Q)\}.
\end{equation*}
We will now refer to any family of Christ-David cubes for $\partial\Omega$ by $\mathscr{D}$ and define
\begin{equation*}
    B_Q = B(x_Q, \ell(Q)).
\end{equation*}

A coronization consists of a partition of $\scD$ into ``good'' cubes $\scG$ and ``bad'' cubes $\scB$ and a further partition of $\scG$ into disjoint \textit{stopping time regions} $\scF = \{S_i\}_i$.

\begin{definition}[stopping time regions]\label{def:stopping-time-regions}
    We call $S\subseteq \scD$ a \textit{stopping time region} if it is \textit{coherent}. That is,
    \begin{enumerate}
        \item[(i)] There exists a \textit{top cube} $Q(S)\in S$ such that $R\subseteq Q(S)$ for all $R\in S$,
        \item[(ii)] If $Q\in S$ and $R\in\scD$ satisfies $Q\subseteq R \subseteq Q(S)$, then $R\in S$,
        \item[(iii)] If $Q\in S$, then either every child of $Q$ is in $S$, or none of them are.
    \end{enumerate}
\end{definition}
\begin{remark}\label{rem:Whitney-stop-region}
    We note that Definition \ref{def:stopping-time-regions} makes sense with any partially ordered collection of subsets of $\R^{d+1}$ in place of $\scD$ as long as the notion of \textit{child} is well-defined. For instance, we will use the term stopping time region to refer to such collections of Whitney boxes with the partial order defined in Definition \ref{def:whitney-cubes}.
\end{remark}

\begin{definition}[Coronizations (cf. \cite{DS93} Definition 3.13)]
    We say that a triple $(\scG, \scB, \scF)$ is a \textit{coronization} of $\scD$ if
    \begin{enumerate}[label=(\roman*)]
        \item $\scF$ is a collection of disjoint stopping time regions as in Definition \ref{def:stopping-time-regions} with $\scG = \bigcup_{S\in\scF}S$,
        \item $\scG\cup\scB = \scD$ and $\scG\cap\scB = \varnothing$,
        \item $\scB$ and $\{Q(S)\}_{S\in\scF}$ satisfy Carleson packing conditions. That is, there exist constants $C_1,C_2 > 0$ such that for any $Q\in \scD$
        $$\sum_{\substack{R\in\scB \\ R\subseteq Q}} \ell(R)^d \leq C_1\sH^d(Q),\ \text{ and } \sum_{\substack{S\in\scF \\ Q(S)\subseteq Q}}\ell(Q(S))^d \leq C_2\sH^d(Q). $$
    \end{enumerate}
\end{definition}

The stopping time regions in coronizations collect scales and locations into good, ``connected'' packages on which $\partial\Omega$ behaves well. David and Semmes used the concept of a coronization to produce a definition of uniform rectifiability involving \textit{corona decompositions}.

\begin{definition}[Corona decomposition (cf. \cite{DS93} Definition 3.19)]
    We say that a set $E\subseteq\R^n$ admits a $d$-dimensional \textit{corona decomposition} if for any constants $\eta,\theta > 0$, there exists a coronization $(\scG,\scB,\scF)$ of a $d$-dimensional lattice $\scD$ for $E$ such that for each $S\in\scF$, there exists a $d$-dimensional Lipschitz graph $\Gamma(S)$ with Lipschitz constant less than $\eta$ such that for each $x\in 2Q$ and $Q\in S$
    \begin{equation}\label{e:corona-decomp}
        \dist(x,\Gamma(S)) \leq \theta\diam(Q).
    \end{equation}
\end{definition}

If one has an appropriate coronization, then one can use Reifenberg parameterizations to produce the approximating Lipschitz graphs in the definition of the corona decomposition directly. We call these specific good coronizations \textit{graph coronizations}

\begin{definition}[graph coronizations]\label{def:graph-coronization}
    For constants $M,\epsilon,\delta > 0$, we say that $(\scG,\scB,\scF)$ is a $d$-dimensional $(M,\epsilon,\delta)-$\textit{graph coronization} if it is a coronization such that $\scB \supseteq \BWGL(M,\epsilon)$ and for each $S\in\scF$ and $Q\in S$, there exists a $d$-plane $P_Q\ni x_Q$ such that 
    \begin{enumerate}[label=(\roman*)]
        \item $b\beta_E(MB_Q,P_Q) \leq 2b\beta_E(MB_Q) \leq 2\epsilon$, \label{i:gc-plane}
        \item $\sum_{\substack{x\in Q\in S}} \beta^{d,1}_E(MB_Q)^2 \leq \epsilon^2$ for any $x\in Q(S)$, \label{i:gc-beta}
        \item $\ang(P_Q, P_{Q(S)}) \leq \delta$.\label{i:gc-angle}
    \end{enumerate}
    
\end{definition}
Condition \ref{i:gc-beta} above uses the \textit{content beta number} introduced by Azzam and Schul in \cite{AS18}. This is closely related to the more standard \textit{$L^p$ beta numbers} used by David and Semmes in characterizing uniform rectifiability via the \textit{strong geometric lemma} (See Proposition \ref{p:ur-equiv}.).
\begin{definition}[$L^p$ beta numbers and content beta numbers]\label{def:betas}
    Let $B = B(x,r) \subseteq \R^{d+1}$ and let $P$ be a $d$-plane. We define
    \begin{equation*}
        \beta_{E,p}^d(B,P) = \left(\frac{1}{r^d}\int_{B\cap E}\left( \frac{\dist(y,P)}{r} \right)^p\ d\sH^d(y) \right)^{1/p},
    \end{equation*}
    and we define the \textit{$L^p$ beta number} as
    \begin{equation*}
        \beta_{E,p}^d(B) = \inf\{\beta_{E,p}^d(B,P):\text{$P$ is a $d$-dimensional plane in $\R^{d+1}$}\}.
    \end{equation*}
    Similarly, we define
    \begin{equation*}
        \beta_E^{d,p}(B,P) = \left( \frac{1}{r_{B}^{d}}\int_{0}^{\infty}\sH^{d}_{\infty}\{x\in B\cap E: \dist(x,P)>tr_{B}\}t^{p-1}dt \right)^{1/p},
    \end{equation*}
    and we define the \textit{$L^p$ content beta number} as
    \begin{equation*}
        \beta_E^{d,p}(B) = \inf\{\beta_E^{d,p}(B,P) : P \text{ is a $d$-dimensional plane in $\R^{d+1}$}\}.
    \end{equation*}
    If $E$ is Ahlfors $d$-regular, then these two beta numbers are comparable with constants depending on $d$ and the regularity constant. They also satisfy a form of monotonicity property. For content beta numbers, Lemma 2.14 in \cite{AS18} says that, for any balls $B',B\subseteq\R^{d+1}$ centered on $E$ satisfying $B'\subseteq B$, we have
    \begin{equation}\label{e:beta-mono}
        \beta_E^{d,p}(B') \leq \left(\frac{r_B}{r_{B'}}\right)^{1+\frac{d}{p}}\beta_E^{d,p}(B).
    \end{equation}
\end{definition}

\begin{proposition}[cf. \cite{DS93} Part I, Theorem 1.57 and Theorem 2.4; Part IV Proposition 2.1]\label{p:ur-equiv}
Let $E\subseteq \R^{d+1}$ be Ahlfors $d$-regular for $d\geq 1$. The following are equivalent:
\begin{enumerate}[label=(\roman*)]
    \item $E$ is uniformly $d$-rectifiable, \label{i:ur1}
    \item $E$ satisfies the strong geometric lemma: For any $Q\in \mathscr{D}$, $M > 1$, and $1 \leq p < \frac{2d}{d-2}$, \label{i:ur2}
    \begin{equation*}
        \sum_{R\subseteq Q}\beta_{E,p}^d(MB_R)^2\ell(R)^d \lesssim_{M,d} \ell(Q)^d,
    \end{equation*}
    \item $E$ admits a corona decomposition,
    \item $E$ admits an $(M,\epsilon,\delta)$-graph coronization for any $M > 1$ and $\epsilon,\delta > 0$.
\end{enumerate}
\end{proposition}

The main tool we will use to create Lipschitz decompositions is the graph coronization. In Appendix \ref{sec:RF-appendix}, we review the $d$-dimensional traveling salesman results of \cite{AS18} and \cite{Hyd21} which give a similar analysis for general Reifenberg flat sets. By collecting these results, we prove the following proposition:

\begin{proposition}\label{p:RF-graph-coronization}
    For any $d,n\in\N$ with $0 < d < n$, there exists $\epsilon_0(d,n), \delta(d,n) > 0$ such that if $\epsilon \leq \epsilon_0 \ll \delta^4$ and $E\subseteq \R^n$ is $(\epsilon,d)$-Reifenberg flat, then $E$ admits an $(M,\epsilon,\delta)$-graph coronization for any $M > 0$.
\end{proposition}

We will use the existence of graph coronizations as in the previous two propositions to prove Theorems \ref{t:thmB} and \ref{t:thmC}. 

We also record some important facts about using beta numbers to control the Hausdorff distance of planes. Given a set $E\subseteq\R^{d+1}$ and a Christ-David lattice $\scD$ for $E$, we define epsilon numbers adapted to the lattice $\scD$ and a collection of planes $\{P_Q\}_{Q\in\scD}$. Fix $K = \frac{10^4}{\rho}$. We define
\begin{equation*}
    \epsilon(Q) = \sup\left\{d_{KB_R}(P_U,P_R) : k(R) \in \{k(Q),k(Q) - 1\},\ k(U) = k(Q),\ x_Q\in \frac{K}{10}B_U \cap \frac{K}{10}B_R\right\}.
\end{equation*}
This is essentially a version of David and Toro's $\epsilon_k'$ numbers that is adapted to a cube structure rather than a general collection of nets. Now, let $M \geq \frac{10K}{\rho^2}$. We will use these to control $\epsilon_k'$ in terms of $\beta_E^{d,1}(MB_Q)$ in the second lemma below. First, we recall a general result that allows one to bound the Hausdorff distance between planes by beta numbers:
\begin{lemma}[\cite{AS18} Lemma 2.16]\label{l:beta-angle-bound}
Suppose $E\subseteq\R^n$ and $B$ is a ball centered on $E$ such that for all balls $B'\subseteq B$ centered on $E$, $\sH^d_\infty(E\cap B') \geq cr_{B'}^d$. Let $P$ and $P'$ be two $d$-planes. Then
\begin{equation*}
    d_{B'}(P,P') \lesssim_{d,c} \left(\frac{r_B}{r_{B'}}\right)^{d+1}\beta^{d,1}_E(B,P) + \beta^{d,1}_E(B',P'). 
\end{equation*}
\end{lemma}
The next lemma applies this to bound $\epsilon(Q)$ by $\beta_E^{d,1}(MB_Q)$ on the class of \textit{$(c,d)$-lower content regular sets}.
\begin{definition}\label{def:lcr}
    We say that $E\subseteq\R^n$ is \textit{$(c,d)$-lower content regular} if for every $x\in E$ and $0 < r  <\diam(E)$, we have
    \begin{equation*}
        \sH_\infty^d(E\cap B(x,r)) \geq cr^d.
    \end{equation*}
\end{definition}

\begin{lemma}\label{l:ep-by-beta}
Let $\scD$ be a Christ-David lattice for a $(c,d)$ lower content regular set $E$ and $K,M > 0$ be constants such that $\frac{10^4}{\rho} \leq K \leq 10^{-1}\rho^2 M$. If $\{P_Q\}_{Q\in\scD}$ is a family of $d$ planes satisfying $\beta^{d,1}_E(2\rho^{-1}KB_Q,P_Q) \lesssim \beta^{d,1}_E(2\rho^{-1}K B_Q)$, then
\begin{equation*}
    \epsilon(Q) \lesssim_{\rho,M,d} \beta^{d,1}_E(MB_Q). 
\end{equation*}
\end{lemma}
\begin{proof}
Let $U,R\in\scD$ be cubes which achieve the supremum in the definition of $\epsilon(Q)$. Then
\begin{equation*}
    \epsilon(Q) = d_{KB_R}(P_U,P_R).
\end{equation*}
We want to apply Lemma \ref{l:beta-angle-bound} with $B = B' = KB_R$. First, we prove some ball inclusions. We claim 
\begin{equation}\label{e:ball-contain-1}
    KB_R\subseteq 2\rho^{-1}KB_U.
\end{equation}
Indeed, we let $y\in KB_R$ and we compute
\begin{align*}
    |y - x_U| &\leq |y - x_R| + |x_R - x_Q| + |x_Q - x_U| \\
    &\leq K\ell(R) + \frac{K}{10}\ell(R) + \frac{K}{10}\ell(U) \leq 2K\ell(R) \leq 2\rho^{-1}K\ell(U).
\end{align*}
Second, we claim
\begin{equation}\label{e:ball-contain-2}
    2\rho^{-1}KB_U \subseteq MB_Q \text{ and } 2\rho^{-1}KB_R \subseteq MB_Q.
\end{equation}
Because $\ell(R) \geq \ell(U)$, it suffices to prove $2\rho^{-1}KB_R \subseteq MB_Q$. We let $y\in 2\rho^{-1}KB_R$ and compute
\begin{equation*}
    |y - x_Q| \leq |y-x_R| + |x_R - x_Q| \leq 4\rho^{-1}K\ell(R) + \frac{K}{10}\ell(R) \leq 10K\rho^{-2}\ell(Q) < M\ell(Q).
\end{equation*}
Now, we apply Lemma \ref{l:beta-angle-bound} with $B = B' = KB_R$, then 
\begin{align*}
    d_{KB_R}(P_U,P_R) &\lesssim \beta^{d,1}_E(KB_R,P_R) + \beta^{d,1}_E(KB_R,P_U)\\
    &\lesssim_\rho \beta^{d,1}_E(2\rho^{-1}KB_R,P_R) + \beta^{d,1}_E(2\rho^{-1}KB_U,P_U)\\
    &\lesssim\beta^{d,1}_E(2\rho^{-1}KB_R) + \beta^{d,1}_E(2\rho^{-1}KB_U)\\
    &\lesssim_M \beta^{d,1}_E(MB_Q).
\end{align*}
where the second line follows from \eqref{e:ball-contain-1}, the third line follows from the hypothesis on $P_Q$, and the final line from \eqref{e:ball-contain-2} using the monotonicity property in \eqref{e:beta-mono}.
\end{proof}

\section{Preliminaries on Lipschitz graph domains}\label{sec:lipschitz-graph-domains}
Our plan to produce Lipschitz graph domains is to produce a Reifenberg parameterization $g$ from Theorem \ref{DT-thm}, find a nice collection of Lipschitz graph domains in the domain of $g$ on which $g$ is a small perturbation of an affine map, then map these domains forward to Lipschitz graph domains in the image. In this section, we prove some preliminary geometric results that allow us to carry out the last two steps. 

\subsection{Whitney cubes and stopping time domains}
In Section \ref{sec:thmA}, we will show how to decompose the domain of $g$ into nice stopping time regions (Recall Definition \ref{def:stopping-time-regions}.) of a family of good \textit{Whitney cubes}. We will make significant use of the standard Whitney decomposition of the upper half-space with respect to $\R^d\times\{0\}\subseteq\R^{d+1}$. 
\begin{definition}[Whitney cubes]\label{def:whitney-cubes}
    Define
\begin{equation*}
    \mathscr{W} = \left\{[k_12^{-n},(k_1+1)2^{-n}]\times\cdots\times[k_d2^{-n},(k_d+1)2^{-n}]\times[2^{-n},2^{-n+1}]:k_1,\ldots,k_d,n\in\Z\right\}.
\end{equation*}
The collection $\mathscr{W}$ consists of exactly the dyadic cubes in $\R^{d}\times[0,\infty)$ which satisfy 
\begin{equation}\label{e:def-height}
    \ell(W) = h(W) \vcentcolon= \dist(W,\R^d)
\end{equation}
where $\ell(W)=h(W)$ denote the \textit{side length} of $W$ and the \textit{height} of $W$. Cubes $W,R\in\mathscr{W}$ have a natural partial order induced by distance to $\R^d\times\{0\}$. We define the projection $\pi:\R^{d+1}\rightarrow\R^{d}\times\{0\}$ by $\pi(x,y) = x$ where $(x,y)\in\R^{d}\times\R$ and write
$$W \leq R$$
if and only if $\pi(W)\subset\pi(R)$. If $W\leq R$ and $h(W) = \frac{1}{2}h(R)$, we call $W$ a \textit{child} of $R$ and $R$ a \textit{parent} of $W$. This gives a partial order on $\mathscr{W}$ which we use to define the \textit{descendants} of $W$
\begin{equation}\label{e:whit-desc}
    D(W) = \{R\in\mathscr{W} : R \leq W\}.
\end{equation}
\end{definition}
The partial order on Whitney cubes gives a natural tree structure on $\mathscr{W}$ which we will use in stopping time constructions.

\begin{definition}[Stopping time regions and domains of Whitney cubes]
    The partial order $\leq$ and definition of child cubes in Definition \ref{def:whitney-cubes} give a well-defined notion of \textit{stopping time region} for subsets of $\scW$. We say that $T\subseteq\scW$ is a stopping time region if it satisfies natural analogs of the three coherency conditions in Definition \ref{def:stopping-time-regions}. That is, the following hold:
    \begin{enumerate}
        \item[(i)] There exists a \textit{top cube} $W(T)\in T$ such that $R\subseteq W(T)$ for all $R\in T$,
        \item[(ii)] If $R\in T$ and $U\in\scD$ satisfies $R\subseteq U \subseteq W(T)$, then $U\in T$,
        \item[(iii)] If $R\in T$, then either every child of $R$ is in $T$, or none of them are.
    \end{enumerate}
    Given a stopping time region $T\subseteq\scW$, we can form an associated \textit{stopping time domain} $\sD_T$ by taking the union of cubes in $T$, defining
    \begin{equation}
        \sD_T = \bigcup_{W\in T}W.
    \end{equation}
    We also define the set of \text{minimal cubes} in $T$ to be
    \begin{equation}
        m(T) = \{W\in T: \text{ If $R\in T$ such that $R \leq W$, then $R = W$}\}.
    \end{equation}
    This is the set of cubes that have no children inside $T$.
\end{definition}

Observe that stopping time domains are not necessarily Lipschitz graph domains because their boundaries contain the faces of cubes, and any face whose normal vector is perpendicular to the vector between the center of the face and a choice of center point for the domain will not have a well-defined radius function in Definition \ref{def:lip-graph-domain}. Therefore, we need a process to carve stopping time domains into Lipschitz graph domains with uniform constant in a controlled way.

Fix some stopping time region $T\subseteq\scW$. In the rest of this section, we will construct an Ahlfors $d$-regular union of subsets of $d$-planes $\Sigma_T$ that carves $\sD_T$ into a collection of $c(d)$-Lipschitz graph domains. The images of these nicer domains under a Reifenberg parameterization whose derivative is nearly constant on the domain will then map them forward to Lipschitz graph domains as desired in the conclusions of Theorems \ref{t:thmA}, \ref{t:thmB}, and \ref{t:thmC}.

\subsection{Carving up stopping time domains}
We want to prove the following proposition:
\begin{proposition}\label{p:DT-lip-graph}
There exists a constant $L_0(d) > 0$ such that for any stopping time region $T\subseteq\scW$, there exists an Ahlfors upper $d$-regular set $\Sigma_{T}$ that is a union of subsets of $d$-planes such that 
\begin{equation*}
    \sD_T\setminus \Sigma_T = \bigcup_{j\in J_T}\sD_T^j
\end{equation*}
where
\begin{align}\label{e:DSj-bound}
    \sum_{j\in J_T}\mathcal{H}^d(\partial \sD_T^j) \lesssim_d \mathcal{H}^d(\partial\sD_T) \simeq_d \sH^d(\sD_T\cap\R^d) + \sum_{W\in m(T)} \ell(W)^d,
\end{align}
and $\sD_T^j$ is an $L_0$-Lipschitz graph domain.
\end{proposition}
The set $\Sigma_T$ will be defined as a union of more local sets $\Sigma_W$ for $W\in m(T)$. The basic idea is to use a ``cover'' emanating from the bottom face of every minimal cube $W$ downwards at a $\frac{\pi}{4}$ angle with the vertical in order to turn the jagged right angles made by stopped cubes into smoother $\frac{\pi}{4}$ angles which look Lipschitz to a point sitting above them higher up in the domain. This is essentially a modification of Peter Jones's algorithm for turning chord arc domains composed of Whitney boxes in the disk into Lipschitz graph domains in his proof of the Analyst's Traveling Salesman Theorem in the complex plane (see pg. 8 of \cite{Jon90}). We now construct $\Sigma_W$.

Fix $T$ and $W\in m(T)$. By translating and dilating, we can without loss of generality assume $W = [-1,1]^d\times[2,4]$. For any function $f:\bR^d \rightarrow \bR$, we let $\Graph(f)$ denote the graph of $f$ in $\R$ over $\R^d\times\{0\}$. We begin by defining, for $1 \leq j \leq d$,
\begin{align*}
    H_0(x) &= 2,\\
    H_{2j-1}(x) &= 3 + x_j,\\
    H_{2j}(x) &= 3 - x_j.
\end{align*}
The graphs of these functions (except $H_0$) over $\bR^d$ are planes which make an angle of $\frac{\pi}{4}$ with $\bR^d$ and contain the edges of $\Bot(W)$ with $x_j = -1$ and $x_j = 1$ respectively. We define
\begin{align*}
    H_W(x) &= \min_{0 \leq i \leq 2d} H_i(x),\\
    \Cover(W) &= \Graph(H_W)\cap\Hp. 
\end{align*}
\begin{figure}
    \centering
    \includegraphics[scale=0.5]{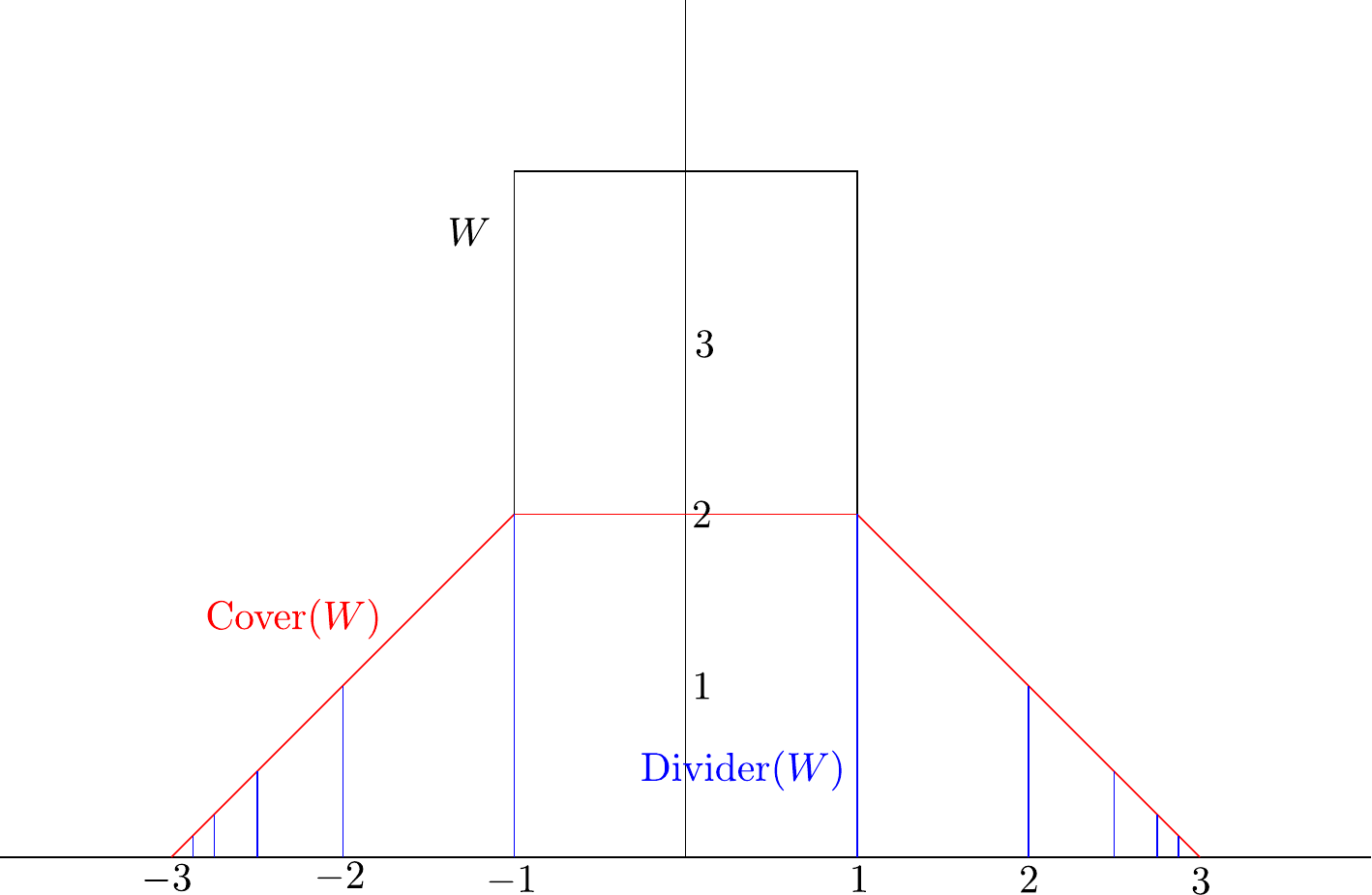}
    \caption{A representation of $W$, $\Cover(W),\ \text{and } \Divider(W)$ in $\R^2$.}
    \label{fig:Sigma_W_1}
\end{figure}
$\Cover(W)$ is the lower envelope of the collection of planes given by the graphs of the $H_i$. In $\bR^3$, $\Cover(W)$ forms the sides of a square pyramid minus its tip with base $[-3,3]^2\times\{0\}$. In general, $\Cover(W)$ divides $\Hp$ into two components: a bounded component $C_W$ with boundary $\Cover(W)\cup [-3,3]^d\times\{0\}$ and the unbounded complimentary component. It also follows that
\begin{equation}\label{e:cover}
    \sH^d(\Cover(W)) \lesssim_d \sH^d(\Bot(W)) = \ell(W)^d.
\end{equation}
$\Cover(W)$ is one of two parts of $\Sigma_W$. The second part will be called $\Divider(W)$ because its purpose will be to ensure that all future domains beneath $\Cover(W)$ look similar to the top domain by separating future domains from one another with vertical plane extensions of the sides of cubes sliced by $\Cover(W)$. See Figure \ref{fig:Sigma_W_1}.

We begin by defining $t_n = 1 + \sum_{j=0}^{n-1}2^{-j}$ and
\begin{align*}
    \mathscr{Q}_n = \bigg\{Q\in\Delta^d([-3,3]^d\times\{0\}) :\ &\ell(Q) = t_{n+1} - t_n = 2^{-n},\\
    &\exists j,\ 1 \leq j \leq d,\ a_j = \pm t_n,\ Q = \prod_{j=1}^d[a_j,b_j]\bigg\}
\end{align*}
where $\Delta^d([-3,3]^d\times\{0\})$ is the set of $d$-dimensional dyadic cubes contained in $[-3,3]^d\times\{0\}$. Intuitively, we think of $t_n$ as the radii of growing balls centered at 0 in the $\ell_\infty$ metric, and the cubes inside $\mathscr{Q}_n$ as the natural collection of dyadic cubes tiling the set difference between successive balls with side length exactly equal to the gap between the two square rings forming the boundaries of the $\ell_\infty$ balls. See Figure \ref{fig:Sigma_W_cubes}. Set $\mathscr{Q} = \bigcup_{n=1}^\infty \mathscr{Q}_n$ and define
\begin{equation*}
    \Divider(W) = C_W\cap\bigcup\left\{ F_j\times[0,2\ell(Q)] : F_j\in\face(Q),\ Q\in\mathscr{Q} \right\}.
\end{equation*}
Because $\sum_{j=1}^{2d}\sH^d(F_j\times[0,2\ell(Q)]) \lesssim_d \sH^d(Q)$ and $[-3,3]^d\times\{0\} = \bigcup_{Q\in\mathscr{Q}} Q$ is a disjoint union, it follows immediately that
\begin{equation}\label{e:divider-bound}
   \sH^d(\Divider(W)) \lesssim_d \sH^d(\Bot(W)) = \ell(W)^d.
\end{equation}
\begin{figure}
    \centering
    \includegraphics[scale=0.4]{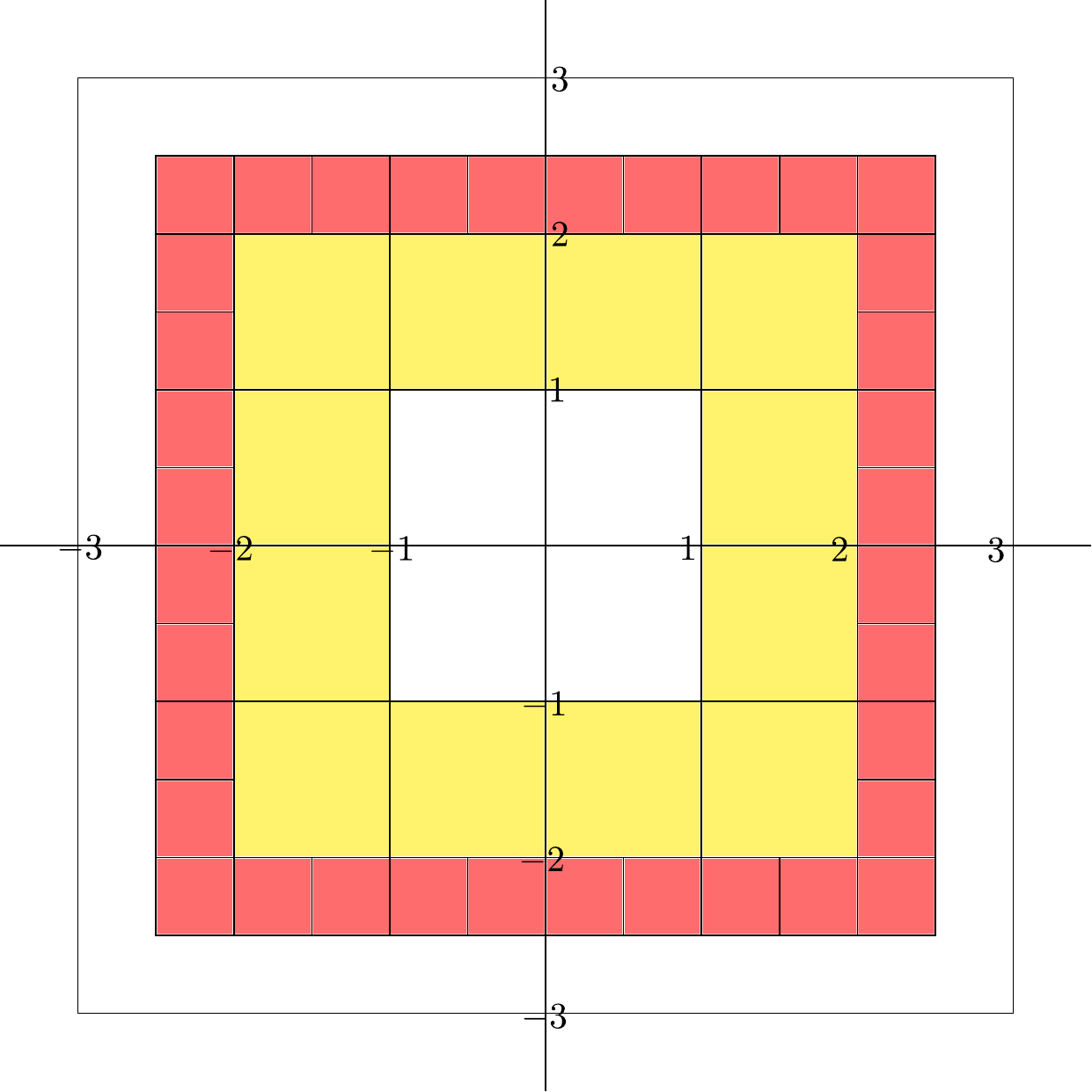}
    \caption{A representation of $[-3,3]^2\times\{0\}$ split into $\mathscr{Q}_1$ in yellow, $\mathscr{Q}_2$ in red, and $\cup_{n=3}^\infty\mathscr{Q}_n$ left uncolored at the edge of $\mathscr{Q}_2$ (The white square in the middle sits below the cube $W\in m(T)$, hence nothing above it lies in $\sD_T$). The set $\Divider(W)$ shoots out of the page as a union of extensions of the sides of the squares up to the points at which they hit the slanting top of $\Cover(W)$.}
    \label{fig:Sigma_W_cubes}
\end{figure}
Now, we define
\begin{equation*}
    \Sigma_W = \Cover(W)\cup \Divider(W),
\end{equation*}
\begin{equation*}
    \Sigma_T = \bigcup_{W\in m(T)}\Sigma_W\cap\sD_T.
\end{equation*}
We first prove the upper regularity claim of Proposition \ref{p:DT-lip-graph}.
\begin{lemma}\label{l:sigma-upper-reg}
$\Sigma_T$ is upper Ahlfors upper $d$-regular with constant $C\lesssim_d 1$.
\end{lemma}
\begin{proof}
Fix $R>0$ and $x\in\Sigma_W\subseteq \Sigma_T$ for some $W\in m(T)$. We write
\begin{equation*}
    \sH^d(\Sigma_T\cap B(x,R)) = \sum_{\substack{W\in m(T) \\ h(W) < 10R}}\sH^d(\Sigma_W \cap B(x,R)) + \sum_{\substack{W\in m(T) \\ h(W) \geq 10R}}\sH^d(\Sigma_W \cap B(x,R)).
\end{equation*}
We note that $\pi(W)$ and $\pi(W')$ have disjoint interiors for any $W,W'\in m(T)$ with $W\not=W'$, so that 
\begin{equation*}
    \sum_{\substack{W\in m(T) \\ h(W) < 10R}}\sH^d(\Sigma_W \cap B(x,R)) \lesssim_d \sum_{\substack{W\in m(T) \\ \Sigma_W\cap B(y,R)\not=\varnothing\\ h(W) < 10R}}\sH^d(\Bot(W)) \leq (20R)^d.
\end{equation*}
On the other hand, there are a uniformly bounded number of minimal cubes $N(d)$ with $h(W) \geq 10R$ such that $B(x,R)\cap \Sigma_W\not=\varnothing$ so that
\begin{equation*}
    \sum_{\substack{W\in m(T) \\ h(W) \geq 10R}}\sH^d(\Sigma_W \cap B(x,R)) \leq N(d)\cdot c(d)R^d \lesssim_d R^d
\end{equation*}
because $\sH^d(\Sigma_W\cap B(x,R)) \leq c(d)R^d$ for any particular $W$ by construction. Therefore, $\Sigma_T$ is upper regular.
\end{proof}
We now finish the proof of Proposition \ref{p:DT-lip-graph}.

\begin{proof}[Proof of Proposition \ref{p:DT-lip-graph}]
It follows from \eqref{e:cover} and \eqref{e:divider-bound} that
\begin{equation*}
    \sH^d(\Sigma_T) \leq \sum_{W\in m(T)}\sH^d(\Sigma_W) \lesssim_d \sum_{W\in m(T)}\sH^d(\Bot(W)) \leq \sH^d(\Bot(W(T))) \lesssim_d \sH^d(\partial \sD_T),
\end{equation*}
which proves \eqref{e:DSj-bound}. We now need to show that the resulting domains $\sD_T^j$ are Lipschitz-graphical. If $\sD_T^j$ is the domain containing $W(T)$, then the claim follows with the choice of central point $c_{W(T)}$. Indeed, the cube $W(T)$ is clearly Lipschitz-graphical with respect to $c_{W(T)}$, and any boundary point of $\sD_T^j$ not in $\partial W(T)$ is either in a vertical plane containing one of the vertical faces of $W(T)$, or is part of the Lipschitz graph consisting of the horizontally planar faces $\Bot(W)$ for $W\in m(T)$ and the planes of $\Cover(W)$ making $\frac{\pi}{4}$ angles with the bottom faces.

Now, suppose $\sD_T^j\cap W(T) = \varnothing$. We have set up the construction such that this will not differ too much from the top cube case. Let $W\in m(T)$ be a cube of minimal height such that $\sD_T^j\subseteq C_W$ and $\sH^d( \partial \sD_T^j\cap\Cover(W)) > 0$. Such $W$ exists because its minimality implies that for any $W'\in m(T)$ of smaller side length than $W$, $\Cover(W')$ can only be part of the ``lower'' boundary of $\sD_T^j$ while the only non-vertical planar pieces in $\Sigma_T$ are bottoms and covers of minimal cubes. Then the cube $R$ of maximal height such that $R\cap \sD_T^j\not=\varnothing$ is exactly the cube of length $\ell(Q)$ sitting above $Q\subseteq\R^{d}\times\{0\}, Q\in\mathscr{Q}$ used in the definition of $\Divider(W)$. 

Therefore, $R\cap\sD_T^j$ is a cube sliced by finitely many $d$-planes passing through its sides and corners at $\frac{\pi}{4}$ angles. By the geometry described above, $\sD_T^j$ contains the convex hull of $c_R$ and $\Bot(R)$, so we have that $\sD_T^j$ is Lipschitz-graphical with respect to $\frac{1}{2}(c_R + c_{\Bot(R)})$. Indeed, Lipschitz-graphicality follows for points in $R\cap \sD_T^j$ immediately, and follows for the rest of $\sD_T^j$ by the same argument as for the region containing $W(T)$ because the definition of $\Divider(W)$ ensures that all cubes which make up $\sD_T^j$ are children of $R$. Indeed, the boundary outside of $\partial R$ consists of vertical planes containing one of the vertical faces of $R$ or is part of a Lipschitz graph consisting of horizontally planar faces $\Bot(W')$ for $W'\in m(T)$ with $W'\leq R$ and the planes of $\Cover(W')$ making $\frac{\pi}{4}$ angles with the bottom faces.
\end{proof}

\subsection{Images of Lipschitz graph domains}

We now show that the Lipschitz graph domain property is preserved under images of maps whose derivatives are nearly constant. We begin by observing that linear transformations preserve Lipschitz graph domains

\begin{lemma}\label{l:linear-graph-domain}
    Let $\Omega\subseteq\R^{d+1}$ be an $L_0$-Lipschitz graph domain and let $A:\R^{d+1}\rightarrow\R^{d+1}$ be an $L'$-bi-Lipschitz affine map. Then there exists a constant $L_1(L_0,L')$ such that $A(\Omega)$ is an $L_1$-Lipschitz graph domain.
\end{lemma}
\begin{proof}
    Without loss of generality, assume $A(0) = 0$ and set $\Omega' = A(\Omega)$. Then since $\Omega = \{t\theta : 0\leq t \leq r(\theta),\ \theta\in\mathbb{S}^d\}$, we know that $\Omega' = \{tA(\theta) : 0\leq t \leq r(\theta),\ \theta\in \mathbb{S}^d\}$ so that $\Omega'$ is star-shaped and $r_{\Omega'}$ is well-defined. We have 
    $$\partial\Omega' = A(\partial\Omega) = \{A(r(\theta)\theta) = r(\theta)A(\theta) : \theta\in\mathbb{S}^d\}.$$
    Therefore, given $\psi\in\mathbb{S}^d$, we see that
    \begin{equation*}
        r_{\Omega'}(\psi) = r_{\Omega}\left(\frac{A^{-1}(\psi)}{|A^{-1}(\psi)|}\right)\frac{1}{|A^{-1}(\psi)|}.
    \end{equation*}
    Because $A^{-1}$ is $L'$-bi-Lipschitz and $r_{\Omega}$ is $L_0$-Lipschitz on $\mathbb{S}^d$, $r_{\Omega'}$ is composed of products and compositions of bounded Lipschitz functions and it follows that there exists $L_1(L_0,L')$ such that $r_{\Omega'}$ satisfies the requirements of Definition \ref{def:lip-graph-domain} after scaling.
\end{proof}
We now move from affine maps to maps whose derivative is sufficiently close to the identity. In preparation, define $\ell_z$ for any $z\in \R^{d+1}$ to be the line passing through $0$ and $z$ and let $P_z = \ell_z^\perp + z$. Define the radial cone at $x$ of aperture $\alpha$ and radius $R$ as
\begin{equation*}
    C_x(\alpha,R) = \left\{y\in B(x,R) : \frac{\dist(y,\ell_x)}{\dist(y,P_x)} < \tan(\alpha)\right\}\setminus \{x\}.
\end{equation*}

\begin{lemma}\label{l:id-lip-graph}
Let $\Omega\subseteq \R^{d+1}$ be an $L_0$-Lipschitz graph domain. There exists a constant $\delta_0(L_0,d) > 0$ such that if $\delta < \delta_0$ and $\varphi:\overline{\Omega}\rightarrow\varphi(\overline{\Omega})$ is a $(1+\delta)$-bi-Lipschitz $C^1$ map satisfying
\begin{equation}\label{e:dphi-almost-id}
    \lvert D\varphi(z) - I\rvert \leq \delta
\end{equation}
for all $z\in \overline{\Omega}$, then there exists $L_1 \lesssim_{L_0,d} 1$ such that $\varphi(\Omega)$ is a $L_1$-Lipschitz graph domain.
\end{lemma}

\begin{proof}
Assume without loss of generality that $\Omega$ is Lipschitz graphical with respect to $0$ and $\varphi(0) = 0$. We first verify that $r_{\Omega}:\mathbb{S}^d\rightarrow\R^+$ is well-defined, i.e., the domain is star-shaped with respect to 0. Let $\varphi(x)\in\partial \Omega$ and let $\gamma(t) = t\varphi(x)$. We want to show $\gamma\cap\partial\varphi(\Omega) = \{\varphi(x)\}$. Set $\tilde{\gamma}(t) = \varphi^{-1}(\gamma(t))$. We would like to prove
\begin{equation}\label{e:tgamma-x}
    |\tilde{\gamma}'(t) - x| \leq 5\delta|x|
\end{equation}
for all $t\in[0,1]$. First note that
\begin{equation*}
    \lvert D\varphi(z)^{-1} - I\rvert = \lvert D\varphi(z)^{-1}\cdot\left[I-D\varphi(z)\right]\rvert \leq 2\delta\lvert D\varphi(z)^{-1} \rvert \leq 3\delta
\end{equation*}
using the bound $|D\varphi(z)^{-1}| \leq \frac{1}{\sigma_{\text{min}}(D\varphi(z))} \leq (1+2\delta)$ where $\sigma_{\text{min}}(D\varphi(z))$ is the smallest singular value of $D\varphi(z)$. This means
\begin{align*}
    |\tilde{\gamma}'(t) - x| &= |D\varphi^{-1}(\gamma(t))\cdot\gamma'(t) - x| = \lvert\left[D\varphi(\tilde{\gamma}(t))^{-1} - I\right]\cdot \gamma'(t) + \gamma'(t) - x\rvert\\
    &\leq 3\delta|\varphi(x)| + |\varphi(x) - x| \leq 5\delta|x|
\end{align*}
where the final line follows from the fact that $\varphi(x) = \int_0^1D\varphi(tx)\cdot x\ dt = x + \int_0^1(D\varphi(tx) - I)\cdot x\ dt$ so that $|\varphi(x) - x|\leq \delta|x|$. It follows from the mean value theorem that $\tilde{\gamma}\subseteq C_x(10\delta,|x|)$. Since $C_x(10\delta,|x|)\cap\partial \Omega = \varnothing$ for $\delta$ sufficiently small in terms of $L_0$, it follows that choosing $\delta_0$ small enough gives $\tilde{\gamma}\cap\partial \Omega = \{x\}$ so that $\gamma\cap \partial\varphi(\Omega) = \{\varphi(x)\}$ as desired.

Set $\Omega' = \varphi(\Omega)$. Now, $r_{\Omega'}$ is well-defined and \eqref{e:dphi-almost-id} implies 
$$\frac{1}{2(L_0+1)} \leq r_{\Omega'}(\theta) \leq 2.$$
Using the fact that we only need a dilation and translation of $\Omega'$ to satisfy the inequalities in Definition \ref{def:lip-graph-domain}, we only need to show that $r_{\Omega'}$ satisfies the a Lipschitz bound as in Definition \ref{def:lip-graph-domain} for some constant $L_1(L_0,d)$. Let $a,b\in\partial \Omega'$ with $a = |a|\psi_1$ and $b = |b|\psi_2$. Let $\psi = |\psi_1-\psi_2|$. If $\psi \geq \frac{\pi}{4}$, then the result follows directly from the fact that $\varphi$ is $(1+\delta)$-bi-Lipschitz. If instead $\psi < \frac{\pi}{4}$, then there exist unique $x,y\in \partial\Omega$ such that $a = \varphi(x)$ and $b = \varphi(y)$ and we assume without loss of generality that $|x| \geq |y|$. Let $x = r_\Omega(\theta_1)\theta_1 = |x|\theta_1,\ y = r_{\Omega}(\theta_2)\theta_2 = |y|\theta_2$ and set $\theta = |\theta_1 - \theta_2|$.

We first claim that it suffices to show 
\begin{equation}\label{e:ang-low-bound}
    |\theta_1-\theta_2| \lesssim_{L_0,d} |\psi_1-\psi_2|.
\end{equation}
Indeed, if \eqref{e:ang-low-bound} holds, then
\begin{align*}
    |r_{\Omega'}(\psi_1) - r_{\Omega'}(\psi_2)| &= \left||a| - |b|\right| \leq |a-b| = |\varphi(x) - \varphi(y)| \leq (1+\delta)|x - y|\\ &\leq (1+\delta)(|x - \theta_1|y|| + |\theta_1|y| - y|)\\
    &= (1+\delta)(r_{\Omega}(\theta_1) - r_\Omega(\theta_2) + |y||\theta_1 - \theta_2|)\\
    &\leq (1+\delta)(L_0 + 1)|\theta_1 - \theta_2| \lesssim_{L_0,d}|\psi_1 - \psi_2|.
\end{align*}
Now, we concentrate on proving \ref{e:ang-low-bound}. Put $z = (1-|x-y|)x$ and $c = (1-|a-b|)a$ and define 
\begin{align*}
    \alpha = \angle zxy,\  \alpha' = \angle\varphi(z)ab,\ \beta = \angle cab.
\end{align*}
By the law of cosines,
\begin{align*}
    \cos\alpha &= \frac{|z-x|^2 + |x-y|^2 - |z-y|^2}{2|z-x||x-y|} = 1 - \frac{|z-y|^2}{2|z-x|^2},\\
    \cos\alpha' &= \frac{|\varphi(z)-\varphi(x)|^2 + |\varphi(x)-\varphi(y)|^2 - |\varphi(z) - \varphi(y)|^2}{2|\varphi(x)-\varphi(z)||\varphi(x)-\varphi(y)|} \\
    &\leq \frac{2(1+\delta)^2|z-x|^2 - (1-\delta)^2|z-y|^2}{2(1-\delta)^2|z-x|^2} \leq 1 - \frac{|z-y|^2}{2|z-x|^2} + 5\delta = \cos\alpha + 5\delta.
\end{align*}
Because $\Omega$ is $L_0$-Lipschitz-graphical, $\alpha \gtrsim_{L_0} 1$ so that if $\delta$ is sufficiently small, then $\alpha' \geq \frac{\alpha}{2}$. In addition, \eqref{e:tgamma-x} implies that $\varphi([x,z])\subseteq C_{\varphi(x)}(10\delta,2(|\varphi(x)|-|c|)$ so that $|\beta - \alpha'| \leq 20\delta$, meaning $\beta \geq \frac{\alpha}{4}$ as long as $\delta$ is small enough. To complete the proof, observe that $|\psi_1-\psi_2| \simeq \angle a0b$, $|\theta_1 - \theta_2| \simeq \angle x0y$, $\beta = \angle 0ab$, and $\alpha = \angle 0xy$ so that $\beta \geq \frac{\alpha}{4}$ implies \eqref{e:ang-low-bound} using the fact that $\varphi$ is $(1+\delta)$-bi-Lipschitz.
\end{proof}

Finally, by chaining Lemmas \ref{l:linear-graph-domain} and \ref{l:id-lip-graph}, we can prove the following desired proposition:
\begin{proposition}\label{p:slow-vary-lip-graph}
    Let $\Omega\subseteq \R^{d+1}$ be an $L_0$-Lipschitz graph domain and suppose $g:\Omega\rightarrow g(\Omega)\subseteq\R^{d+1}$ is $C^1$ and $L$-bi-Lipschitz. There exist constants $L_1,\delta_0(L_0,L) > 0$ such that if $\delta < \delta_0$ and 
    \begin{equation}
        |Dg(z)\cdot Dg(w)^{-1} - I| \leq \delta
    \end{equation}
    for all $z,w\in \Omega$, then $g(\Omega)$ is an $L_1$-Lipschitz graph domain.
\end{proposition}
\begin{proof}
    Suppose $\Omega$ is Lipschitz graphical around $0$ and set
    $$L(z) = Dg(0)\cdot z.$$
    By Lemma \ref{l:linear-graph-domain}, $L(\Omega)$ is $L_0'(L,L_0)$-Lipschitz graphical. The map $\varphi:L(\Omega)\rightarrow g(\Omega)$ given by
    $$\varphi = g\circ L^{-1}$$
    satisfies
    $$D\varphi(z)(L(w)) = Dg(L^{-1}(L(w)))\cdot DL^{-1}(L(w)) = Dg(w)\cdot Dg(0)^{-1}\cdot w$$ 
    so that 
    $$|D\varphi - I| \leq \delta.$$
    By taking $\delta_0$ sufficiently small in terms of $L_0'$, Lemma \ref{l:id-lip-graph} implies that there exists $L_1(L_0')$ such that $g(\Omega)$ is $L_1$-Lipschitz graphical.
\end{proof}

\section{Controlling the change in the derivative of Reifenberg parameterizations}\label{sec:Dg}
The goal of this section is to give conditions under which we can say that the change in the derivative of a Reifenberg parameterization $g:\R^{d+1}\rightarrow\R^{d+1}$ over some subset is small, and hence we have a chance of applying a result like Lemma \ref{l:id-lip-graph}. This is specified exactly in Proposition \ref{p:DgDginv-expanded} below.

\subsection{Preliminary derivative estimates and regularity}
In this subsection, we review some properties of the maps used in the construction of a Reifenberg parameterization $g$ that we need to make specific estimates on the change in $Dg$. First, the surface $\Sigma_k$ has a nice local Lipschitz representation:
\begin{lemma}[\cite{DT12} Lemma 6.12]\label{DT-loc-lip}
For $k \geq 0$ and $y\in \Sigma_k$, there is an affine $d$-plane
$P$ through $y$ and a $C\varepsilon$-Lipschitz and $C^2$ function 
$A : P \to P^\perp$ such that  $|A(x)| \leq C\epsilon r_k$ for all $x\in B(y,19r_k)$ and
$$
\Sigma_k \cap B(y,19r_k) = \Gamma \cap B(y,19r_k).
$$
where $\Gamma$ denotes the graph of $A$ over $P$.
\end{lemma}

Now, we record distortion estimates for $D\sigma_k$ as in \cite{DT12} chapter 7. Importantly, $D\sigma_k$ is very close to the identity in the following sense:
\begin{lemma}[\cite{DT12} Lemma 7.1]\label{DT-Dsig-1}
For $k \geq 0$, $\sigma_k$ is a $C^2$-diffeomorphism
from $\Sigma_k$ to $\Sigma_{k+1}$ and, for $y\in \Sigma_k$,
\begin{equation*}
    D\sigma_k(y) : T\Sigma_k(y) \to T\Sigma_{k+1}(\sigma_k(y))
    \hbox{ is bijective and $(1+C\varepsilon)$-bi-Lipschitz.}
\end{equation*}
In addition,
\begin{equation*}
    |D\sigma_k(y) \cdot v - v| \leq C \varepsilon |v|
    \ \hbox{ for $y\in \Sigma_k$ and } v \in T\Sigma_k(y),
\end{equation*}
\begin{equation*}
    |\sigma_k(y)-\sigma_k(y') - y+y'| \leq C \varepsilon |y-y'|
    \ \hbox{ for } y,y'\in \Sigma_k.
\end{equation*}
\end{lemma}
More precise estimates can be obtained when restricting $D\sigma_k$ to its action on vectors tangent to $\Sigma_k$. The best way to capture this is to use the $\epsilon_k'$ numbers defined in \eqref{e:epsilonk'}.
\begin{lemma}[\cite{DT12} Lemma 7.32]\label{DT-Dsig-2}
For $k \geq 1$ and $y\in \Sigma_k \cap V^8_k$, choose $i\in J_k$
such that $|y-x_{i,k}| \leq 10r_k$. Then
\begin{equation}\label{DT-epk-1}
    |D\pi_{i,k} \circ D\sigma_{k}(y) \circ D\pi_{i,k} -D\pi_{i,k}| 
\leq C \varepsilon'_k(y)^2,
\end{equation}
and 
\begin{equation}\label{DT-epk-2}
    \big| |D\sigma_k(y)\cdot v| - 1 \big| \leq C \varepsilon'_k(y)^2 
    \ \hbox{ for every unit vector } v \in T\Sigma_k(y).
\end{equation}
\end{lemma}
Similarly, these numbers also control the distance between tangent planes to the surface and nearby $P_{j,k}$. For any $k\geq 0$ and $y\in\Sigma_k\cap V_k^8$ and $i\in J_k$ such that $|y-x_{i,k}|\leq 10r_k$, we have (\cite{DT12} (7.22))
\begin{equation}\label{DT-plane-comparison}
    \ang(T\Sigma_k(y),P_{i,k})\leq C\epsilon_k'(y).
\end{equation}

Finally, we also use an estimate on $D^2\sigma_k$ obtained in by Ghinassi in \cite{Ghi17} in work on constructing $C^{1,\alpha}$ parametrizations.
\begin{proposition}[\cite{Ghi17} Proposition 3.12 \label{Ghi-Dsquare}]
For $k\geq 0$, $j\in J_k,\ y\in\Sigma_k\cap 45B_{j,k}$, we have
\begin{equation*}
    |D^2\sigma_k(y) - 2D\psi_k(y)D\pi_{j,k}^\perp|\leq C\frac{\epsilon}{r_k}.
\end{equation*}
\end{proposition}
where we interpret the norm on the tensor $D^2\sigma_k$ as the Euclidean norm on $\R^{n^3}$. We will apply this result to get the following bound on $D^2\sigma_k$ when acting on vectors tangent to $\Sigma_k$.
\begin{lemma}\label{l:D2-sig-control}
    Let $y\in\Sigma_k$ and $v\in T\Sigma_k(y)$. We have
    \begin{equation*}
        |D^2\sigma_k(y)\cdot v| \leq C\frac{\epsilon}{r_k}.
    \end{equation*}
\end{lemma}
\begin{proof}
    If $y\not\in V_k^{11}$, then $\sigma_k$ is the identity in a neighborhood of $y$ and the result follows. Otherwise, $y\in V_k^{11}$ and there exists $j\in J_k$ such that $y\in\Sigma_k\cap11B_{j,k} \subseteq\Sigma_k\cap 45B_{j,k}$. This means we can apply Proposition \ref{Ghi-Dsquare} to get
    \begin{equation*}
        |D^2\sigma_k\cdot v| \leq C\frac{\epsilon}{r_k}|v| + C|D\psi_k(y)\cdot D\pi_{j,k}^\perp\cdot v| \leq C\frac{\epsilon}{r_k}|v|
    \end{equation*}
    where in the last inequality we used the fact $|D\psi_k|\leq \frac{C}{r_k}$ and that $v\in T\Sigma_k(y)$ implies $|D\pi_{j,k}^\perp\cdot v| \leq C\epsilon|v|$ due to the fact that $\ang(P_{j,k},T\Sigma_k(y)) \leq C\epsilon$.
\end{proof}

We also provide the following lemma and proof adapted from a proof of \cite{DT12} to fit our needs.
\begin{lemma}[cf. \cite{DT12} (11.22)]\label{l:short-curve}
Suppose $\Sigma_0$ is such that for any $x,x'\in\Sigma_0$, there exists a curve $\gamma_0$ connecting $x$ and $x'$ with $\ell(\gamma_0)\leq (1+C\epsilon)|x-x'|$. Let $1 \leq M^4\epsilon < c(d) < 1$ with $c(d)$ sufficiently small and $k\geq 0$ be such that $|f_k(x)-f_k(x')| < Mr_k$. There is a curve $\gamma:I\rightarrow\Sigma_k$ connecting $f_k(x)$ and $f_k(x')$ such that $\ell(\gamma)\leq 2|f_k(x) - f_k(x')|$. 
\end{lemma}
\begin{proof}
We first prove the following claim:
\begin{claim}
For any $0\leq p \leq k$, 
\begin{equation}\label{e:induct-distance}
    |f_{k-p}(x) - f_{k-p}(x^\prime)| < \frac{Mr_{k-p}}{5^p}.
\end{equation}
\end{claim}
\begin{claimproof}
We prove this by induction. Indeed, observe that the third display in Lemma \ref{DT-Dsig-1} implies $\sigma_{k-p-1}^{-1}$ is $(1+C\epsilon)$-Lipschitz on points in $\Sigma_{k-p}$. This means 
\begin{equation}\nonumber
    |f_{k-p-1}(x) - f_{k-p-1}(x^\prime)| = |\sigma_{k-p-1}^{-1}(f_{k-p}(x)) - \sigma_{k-p-1}^{-1}(f_{k-p}(x^\prime))| \leq (1+C\epsilon)|f_{k-p}(x) - f_{k-p}(x^\prime)|.
\end{equation}
Applying this for $p=1$ gives
\begin{equation}\nonumber
    |f_{k-1}(x) - f_{k-1}(x^\prime)| < (1+C\epsilon)(Mr_k) < \frac{Mr_{k-1}}{5}.
\end{equation}
This proves the base case. Assuming the claim holds for some $p$, we get
\begin{equation}\nonumber
    |f_{k-p-1}(x) - f_{k-p-1}(x^\prime)| \leq (1+C\epsilon)\frac{Mr_{k-p}}{5^p} < \frac{Mr_{k-p-1}}{5^{p+1}}.
\end{equation}
\end{claimproof}

To continue the proof of the lemma, we modify the proof of \cite{DT12} (11.22). If $|f_k(x) - f_k(x')| < 18r_k$, then the claim follows immediately from the local Lipschitz graph description of $\Sigma_k$ in Lemma \ref{DT-loc-lip}. So, assume $|f_k(x) - f_k(x')| > 18r_k$ and suppose first that there exists an integer $0\leq m\leq k$ such that $|f_m(x) - f_m(x')| < 5r_m$. We calculate
\begin{align*}
    \frac{Mr_m}{5^{k-m}} < 5r_m \iff 
    \log_5M - 1 < k-m
\end{align*}
so that by the above claim we can assume $k - m < \log_5M$. Applying the Lipschitz graph lemma for $B(f_m(x),19r_m)$, we see that there exists a path $\gamma_m\subseteq\Sigma_m$ such that
\begin{align*}
    \ell(\gamma_m) \leq (1+C\epsilon)|f_m(x) - f_m(x')| &\leq (1+C\epsilon)(C\epsilon r_m + |f_k(x) - f_k(x')|)\\
    &\leq (1+C\epsilon)|f_k(x)-f_k(x')| + C\epsilon r_k\cdot10^{k-m}\\
    &\leq (1+C\epsilon)|f_k(x)-f_k(x')| + C\epsilon r_kM^\lambda
\end{align*}
where $1 < \lambda < 2$ solves $10^{\log_5M} = M^\lambda$.
On the other hand, since $|f_m(x) - f_m(x')| < 5r_m$, we get $\ell(\gamma_m) \leq 10r_m$ and so we can choose a maximal chain of $N \leq \frac{10r_m}{10r_k} = 10^{k-m} \leq M^\lambda$ points contained in $\gamma_m$ with consecutive points separated by a distance of at least $11r_k$ beginning at $f_m(x)$ and ending at $f_m(x')$. Call this collection of points $\{f_m(x_l)\}_{l=1}^N$ for $x_l\in\Sigma_0$. This implies the total length of the string of points $\{f_k(x_l)\}$ is
\begin{align*}
    L' = \sum_{l=1}^N|f_k(x_l) - f_k(x_{l+1})| &\leq \sum_{l=1}^N[C\epsilon r_m + |f_m(x_l) - f_m(x_{l+1})|] \leq C\epsilon r_kM^\lambda\cdot M^\lambda+ \ell(\gamma_m)\\
    &\leq C\epsilon r_k M^4 + (1+C\epsilon)|f_k(x) - f_k(x')|.
\end{align*}
In addition, for any admissible $l$ we can calculate
\begin{align}\label{e:consecutive-gamma-pts}
    |f_k(x_l) - f_k(x_{l+1})| \leq C\epsilon r_m + |f_m(x_l) - f_m(x_{l+1})| \leq C\epsilon r_kM^2+ 11r_k < 12r_k.
\end{align}
Using \eqref{e:consecutive-gamma-pts} and Lemma \ref{DT-loc-lip} once again, we connect each pair $(f_k(x_l),f_k(x_{l+1}))$ by a curve $\gamma_l\subseteq\Sigma_k$ of length $\ell(\gamma_l)\leq (1+C\epsilon)|f_k(x_l) - f_k(x_{l+1})|$ to get a curve $\gamma$ with
\begin{align*}
    \ell(\gamma) \leq (1+C\epsilon)L' \leq (1+C\epsilon)|f_k(x) - f_k(x)| + C\epsilon r_k M^4 \leq 2|f_k(x) - f_k(x')|
\end{align*}
using the fact that $|f_k(x) - f_k(x')| > 18r_k$ and $M^4\epsilon$ is sufficiently small. This completes the proof if there exists such an $m$ where $|f_m(x)-f_m(x')| < 5r_m$. If there does not exist such an $m$ ($k$ is too small compared to $|f_k(x) - f_k(x')|$), then we instead use the assumed curve $\gamma_0\subseteq\Sigma_0$ in place of $\gamma_m$ and argue as in the previous case.
\end{proof}
We also recall a reverse triangle inequality.
\begin{lemma}(Reverse Triangle Inequality)\label{l:reverse-triangle}
Let $u,v\in\R^{d+1}$ with $\langle u,v\rangle \geq -\frac{1}{2}|u||v|$. Then
\begin{equation}
    |u| + |v| \leq 2|u+v|.
\end{equation}
\end{lemma}

\subsection{Controlling the change in $Dg$}

Proposition \ref{p:DgDginv-expanded} follows from a series of computations involving the derivative of the map $g$ produced by Theorem \ref{DT-thm} for a given CCBP. Proposition \ref{p:DgDginv-expanded} says that given a ``central'' point $z\in\R^{d+1}$ and an inflation factor $1 \leq M_0$ such that $M_0^4\epsilon$ is sufficiently small, we can get a set $G_z^{M_0}$ such that $w\in G_z^{M_0}$ means $Dg(w)$ is very close to $Dg(z)$ in the sense of \eqref{e:DgDginv}.

Proposition \ref{p:DgDginv-expanded} follows from Lemmas \ref{l:horizontal} and \ref{l:vertical} which give separate horizontal and vertical estimates respectively. Define the sets of horizontal and vertical vectors by
\begin{align*}
    H = \R^d\times\{0\},\ V = \{0\}^d\times\R.
\end{align*}
These lemmas show how to appropriately bound the individual pieces of the difference $Dg(x,y) - Dg(x',y)$ and $Dg(x',y) - Dg(x',y')$ between points $z=(x,y),z'=(x',y')$ respectively when acting on $v\in\ H\cup V$. Corollaries \ref{c:Dg-horizontal} and \ref{c:Dg-vertical} put these pieces together to get the requisite $Dg$ estimates, from which we prove Proposition \ref{p:DgDginv-expanded}.

For the rest of this section, fix a CCBP $(\Sigma_0, \{B_{j,k}\},\{P_{j,k}\})$ with a sufficiently small constant $\epsilon > 0$. Fix its associated collection of surfaces $\{\Sigma_k\}$ and its associated mappings $\{R_k\},\ \{f_k\}$, and $g:\R^{d+1}\rightarrow\R^{d+1}$. Let $z=(x,y)\in\R^{d+1}$ and suppose $f_{n(y)}(x)\in V_{n(y)}^8$ (See \eqref{e:def-g}.). Define
\begin{equation}
    N_z = \{z'=(x',y')\in\R^{d+1}:\ |y'|\leq |y|,\ \text{and } f_{n(y')}(x')\in V_{n(y')}^8\}.
\end{equation}
We think of $N_z$ as a sort of neighborhood of points within the domain of $g$ whose images are ``connected'' to $g(z)$ through the CCBP. It is for $z'\in N_z$ that we will be able to control $|Dg(z) - Dg(z')|$ appropriately.

\begin{proposition}\label{p:DgDginv-expanded}
Let $M_0,\delta > 0$ be such that $1 \leq M_0^4\epsilon < \delta < c(d)$ with $c(d)$ sufficiently small.  Let $z=(x,y)\in\R^{d}\times\R$ with $f_{n(y)}(x)\in V_{n(y)}^8$. Define
\begin{align*}
    G_z^{M_0} = \big\{z'=(x',y')\in N_z:|f_{n(y)}(x)-f_{n(y)}(x')| < M_0r_{n(y)}&,\ \sum_{k=n(y)}^{n(y')}\epsilon_k'(f_k(x'))^2 < \epsilon,\ \\
    &\ang(T_{n(y')}(x'),T_{n(y)}(x') \leq \delta\big\}.
\end{align*}
There exists $C(d)>0$ such that for any $w\in G_z^{M_0}$, we have
\begin{equation}\label{e:DgDginv}
    |Dg(w)\cdot Dg(z)^{-1} - I|\leq C(d)\delta.
\end{equation}
\end{proposition}

\begin{lemma}[Horizontal Estimates]\label{l:horizontal}
Let $z, M_0, \epsilon$ be as in Proposition \ref{p:DgDginv-expanded}, let $z'\in N_z$, and let $v\in H\cup V$. Let $k$ be such that $\rho_k(y) > 0$. If $|f_k(x) - f_k(x^\prime)| < M_0r_k$, then
\begin{equation}\label{e:Dfchange}
    |Df_k(x)\cdot v - Df_k(x^\prime)\cdot v| \leq CM_0\epsilon|Df_k(x)\cdot v|,
\end{equation}
\begin{equation}\label{e:DRchange}
    \left| (D_x(R_k(x)\cdot y))\cdot v - (D_x(R_k(x^\prime)\cdot y))\cdot v\right| \leq CM_0\epsilon|Df_k(x)\cdot v|.
\end{equation}
In any case,
\begin{equation}\label{e:Dgyhorizchange}
    \left|\frac{\partial g}{\partial y}(x,y) - \frac{\partial g}{\partial y}(x',y)\right| \leq CM_0\epsilon \left|\frac{\partial g}{\partial y}(x,y)\right|.
\end{equation}
\end{lemma}
\begin{proof}
We begin with proving (\ref{e:Dfchange}). We have
\begin{align*}
     |Df_k(x)\cdot v - Df_k(x^\prime)\cdot v| &= \big|\big[D\sigma_{k-1}(f_{k-1}(x)) - D\sigma_{k-1}(f_{k-1}(x^\prime))\big]Df_{k-1}(x)\cdot v \\
     &\qquad\qquad + D\sigma_{k-1}(f_{k-1}(x^\prime))\big[Df_{k-1}(x) - Df_{k-1}(x^\prime)\big]\cdot v\big|\\
     &\leq |Df_{k-1}(x)\cdot v||D\sigma_{k-1}(f_{k-1}(x)) - D\sigma_{k-1}(f_{k-1}(x^\prime))| \\
     &\qquad\qquad + |D\sigma_{k-1}(f_{k-1}(x^\prime))||Df_{k-1}(x)\cdot v - Df_{k-1}(x^\prime)\cdot v|.\\
\end{align*}
Recursively applying this inequality for decreasing values of $k$ gives
\begin{align}\nonumber
    |&Df_k(x)\cdot v - Df_k(x^\prime)\cdot v| \\\nonumber
    &\leq |Df_{k-1}(x)\cdot v||D\sigma_{k-1}(f_{k-1}(x)) - D\sigma_{k-1}(f_{k-1}(x^\prime))|\\\nonumber
    &\qquad + |D\sigma_{k-1}(f_{k-1}(x^\prime))|\big(|Df_{k-2}(x)\cdot v||D\sigma_{k-2}(f_{k-2}(x)) - D\sigma_{k-2}(f_{k-2}(x^\prime))|\\\nonumber
    &\qquad + |D\sigma_{k-2}(f_{k-2}(x^\prime))||Df_{k-2}(x)\cdot v - Df_{k-2}(x^\prime)\cdot v|\big)\\\label{e:DfMaster}
    &\leq|Df_{k-1}(x)\cdot v||D\sigma_{k-1}(f_{k-1}(x)) - D\sigma_{k-1}(f_{k-1}(x^\prime))|\\\nonumber
    &+\sum_{p=1}^k\left(\prod_{m=1}^p|D\sigma_{k-m}(f_{k-m}(x^\prime))|\right)|Df_{k-p-1}(x)\cdot v|\cdot|D\sigma_{k-p-1}(f_{k-p-1}(x)) - D\sigma_{k-p-1}(f_{k-p-1}(x^\prime)|.
\end{align}
Now, Lemma \ref{DT-Dsig-1} implies
\begin{equation}\label{e:Dsbrutal}
    \prod_{m=1}^p|D\sigma_{k-m}(f_{k-m}(x^\prime))| \leq (1+C\epsilon)^p,
\end{equation}
and
\begin{equation}\label{e:Dfbrutal}
    |Df_{k-p-1}(x)\cdot v| = \left|\prod_{m=1}^{p+1}D\sigma^{-1}_{k-m}(f_{k-m+1}(x))Df_k(x)\cdot v\right| \leq (1+C\epsilon)^{p+1}|Df_k(x)\cdot v|.
\end{equation}
Using Lemma \ref{l:short-curve}, we see that \eqref{e:induct-distance} implies $|f_{k-p-1}(x) - f_{k-p-1}(x')| < \frac{M_0r_{k-p-1}}{5^{p+1}} \leq M_0r_{k-p-1}$ so that we get a rectifiable curve $\gamma_{k-p-1} \subseteq \Sigma_{k-p-1}$ (with its arc length parameterization $\gamma_{k-p-1}(t))$ connecting $f_{k-p-1}(x)$ and $f_{k-p-1}(x')$ such that $\ell(\gamma_{k-p-1}) \leq 2|f_{k-p-1}(x) - f_{k-p-1}(x')|$. Because $\gamma_{k-p-1}'(t)\in T\Sigma_k(\gamma_{k-p-1}(t))$, Lemma \ref{l:D2-sig-control} gives
\begin{align}\nonumber
    |D\sigma_{k-p-1}(f_{k-p-1}(x)) - D\sigma_{k-p-1}(f_{k-p-1}(x^\prime))| &= \left| \int_I D^2\sigma_{k-p-1}(\gamma_{k-p-1}(t))\cdot\gamma_{k-p-1}^\prime(t)\ dt \right|\\\nonumber
    &\leq \int_I|D^2\sigma_{k-p-1}(\gamma_{k-p-1}(t))\cdot\gamma_{k-p-1}^\prime(t)|dt\\\nonumber
    &\leq C\frac{\epsilon}{r_{k-p-1}}\cdot2|f_{k-p-1}(x) - f_{k-p-1}(x^\prime)|\\\nonumber
    &\leq C\frac{\epsilon}{r_{k-p-1}}\frac{M_0r_{k-p-1}}{5^{p+1}}\\
    &\leq CM_0\frac{\epsilon}{5^{p+1}} \label{e:DfDs}.
\end{align}
Applying (\ref{e:Dsbrutal}), (\ref{e:Dfbrutal}), and (\ref{e:DfDs}) to (\ref{e:DfMaster}) gives
\begin{align}\nonumber
    |Df_k(x)\cdot v - Df_k(x^\prime)\cdot v| &\leq (1+C\epsilon)C\frac{\epsilon}{5}|Df_k(x)\cdot v| + \sum_{p=1}^k(1+C\epsilon)^p(1+C\epsilon)^{p+1}|Df_k(x)\cdot v|CM_0\frac{\epsilon}{5^{p+1}}\\\nonumber
    &\leq C\epsilon|Df_k(x)\cdot v| + C\epsilon M_0|Df_k(x)\cdot v|\sum_{p=1}^k\frac{(1+C\epsilon)^{2p}}{5^{p+1}}\\\nonumber
    &\leq CM_0\epsilon|Df_k(x)\cdot v|.
\end{align}

We now prove (\ref{e:DRchange}). For any $t>0$, Proposition \ref{DT-isometries} implies that the quantity $R_k(x+tv)\cdot e_{d+1} - R_k(x)\cdot e_{d+1}$ is the difference between the unit normal vectors to the linear subspaces $T\Sigma_k(f_k(x+tv))$ and $T\Sigma_k(f_k(x))$. But by Lemma \ref{DT-plane-dist}, we have
\begin{equation}\label{e:Rclosediff}
    |R_k(x+tv)\cdot e_{d+1}-R_k(x)\cdot e_{d+1}| \leq D(T\Sigma_k(f_k(x+tv)), T\Sigma_k(f_k(x))) \leq C\frac{\epsilon}{r_k}|f_k(x+tv)-f_k(x)|.
\end{equation}
Hence, we can write
\begin{align}\nonumber
    |(D_x(R_k(x)\cdot y))\cdot v| &\leq |y|\lim_{t\rightarrow 0} \frac{|R_k(x+tv)\cdot e_{d+1}-R_k(x)\cdot e_{d+1}|}{|t|} \leq C\epsilon\frac{|y|}{r_k}\lim_{t\rightarrow0}\frac{|f_k(x+tv) - f_k(x)|}{|t|}\\
    &\leq C\epsilon|Df_k(x)\cdot v|\label{e:DRlimest}
\end{align}
where $|y|\lesssim r_k$ since $\rho_k(y) > 0$. We then have
\begin{equation}\nonumber
    \left| (D_x(R_k(x)\cdot y))\cdot v - (D_x(R_k(x^\prime)\cdot y))\cdot v\right| \leq C\epsilon(|Df_k(x)\cdot v| + |Df_k(x^\prime)\cdot v|) \leq CM_0\epsilon|Df_k(x)\cdot v|
\end{equation}
using (\ref{e:Dfchange}).

Finally, we prove (\ref{e:Dgyhorizchange}). First, we compute
\begin{align}\nonumber
    \frac{\partial g}{\partial y}(x,y) - \frac{\partial g }{\partial y}(x',y) &= \sum_{k\geq 0} \frac{\partial\rho_k}{\partial y}(y)\left\{ f_k(x) - f_k(x^\prime) + R_k(x)\cdot y - R_k(x^\prime)\cdot y \right\}\\\nonumber
    &\qquad + \sum_{k\geq0} \rho_k(y)(R_k(x)\cdot e_{d+1} - R_k(x^\prime)\cdot e_{d+1})\\\nonumber
    &=: I + II.
\end{align}
Let $p = l(y)$ so that $p$ and $p+1$ are the values of $k$ such that $\rho_k(y) > 0$. Since $\rho_p(y) + \rho_{p+1}(y) = 1$, we have $\frac{\partial\rho_p}{\partial y}(y) + \frac{\partial \rho_{p+1}}{\partial y}(y) = 0$. This implies
\begin{align}\nonumber
    I &= \frac{\partial \rho_p}{\partial y}(y)\left(f_{p}(x) - f_{p+1}(x) + R_p(x)\cdot y - R_{p+1}(x)\cdot y\right)\\\nonumber
    &\ \ + \frac{\partial\rho_p}{\partial y}(y)\left(f_{p}(x^\prime) - f_{p+1}(x^\prime) + R_p(x^\prime)\cdot y - R_{p+1}(x^\prime)\cdot y\right)
\end{align}\
But using \eqref{DT-sigma-id} and \eqref{DT-isometry-change}, we have
\begin{equation}\nonumber
    |I| \leq \frac{C}{r_p}(C\epsilon r_p + C\epsilon|y|) \leq C\epsilon.
\end{equation}
By applying (\ref{e:Rclosediff}) with $x+tv = x'$ (and recalling Remark \ref{rem:tangent-plane-angles}), we have
\begin{align}\nonumber
    |II| &\leq |\rho_p(y)|\frac{C\epsilon}{r_p}|f_p(x) - f_p(x^\prime)| + |\rho_{p+1}(y)|\frac{C\epsilon}{r_{p+1}}|f_{p+1}(x)-f_{p+1}(x^\prime)|\\\nonumber
    &\leq CM_0\epsilon(|\rho_p(y)| + |\rho_{p+1}(y)|) \leq CM_0\epsilon
\end{align}
since $\rho_p(y) + \rho_{p+1}(y) = |\rho_p(y)|+|\rho_{p+1}(y)| = 1$. We've proven that $\left| \frac{\partial g}{\partial y}(x,y) - \frac{\partial g }{\partial y}(x',y) \right| \leq CM_0\epsilon$. We will complete the proof of (\ref{e:Dgyhorizchange}) by showing that $\left|\frac{\partial g}{\partial y}(x,y)\right| \gtrsim 1$. Indeed,
\begin{align}\label{e:Dgy}
    \frac{\partial g}{\partial y}(x,y) &= \bigg| \left[\frac{\partial\rho_p}{\partial y}(y)\left(f_p(x) - f_{p+1}(x) + R_p(x)\cdot y - R_{p+1}(x)\cdot y\right)\right]\\ \nonumber
    &\qquad + \left[\rho_p(y)R_p(x)\cdot e_{d+1} + \rho_{p+1}(y)R_{p+1}(x)\cdot e_{d+1}\right]\bigg|.
\end{align}
But the previous computation shows that the first expression has norm $\leq CM_0\epsilon$, while the second expression is a convex combination of two nearly parallel unit vectors because $R_{p}(x)$ and $R_{p+1}(x)$ are orthogonal matrices which are $C\epsilon$ close. Hence, we get
\begin{equation}\label{e:partial-g-big}
    \left|\frac{\partial g}{\partial y}\right| \gtrsim 1.
\qedhere\end{equation}
\end{proof}
\begin{corollary}\label{c:Dg-horizontal}
Let $z,z'$ be as in Lemma \ref{l:horizontal}. For any vector $v\in H\cup V$, we have 
\begin{equation*}
    |Dg(x,y)\cdot v - Dg(x',y)\cdot v|\leq CM_0\epsilon|Dg(x,y)\cdot v|
\end{equation*}
\end{corollary}
\begin{proof}
First, suppose $v = v_x\in H$. Since $v_x\cdot e_{d+1} = 0$, we have
\begin{equation}\label{e:Dgx}
    Dg(x,y)\cdot v_x = \sum_{k\geq 0}\rho_k(y)\left\{ Df_k(x)\cdot v_x + D(R_k(x)\cdot y)\cdot v_x \right\}.
\end{equation}
Therefore, we get
\begin{align}\label{e:Dgxhorizchange}
    |Dg(x,y)&\cdot v_x - Dg(x',y)\cdot v_x| \\\nonumber
    &\leq \sum_{k\geq 0}\rho_k(y)\big\{ |Df_k(x)\cdot v_x - Df_k(x')\cdot v_x| + |D(R_k(x)\cdot y)\cdot v_x - D(R_k(x'))\cdot v_x|\big\} \\\nonumber
    &\leq CM_0\epsilon\sum_{k\geq0}\rho_k(y)\left\{|Df_k(x)\cdot v_x|\right\}.
\end{align}
using (\ref{e:Dfchange}) and (\ref{e:DRchange}). We now want to bound $\sum_{k\geq0}\rho_k(y)\left\{|Df_k(x)\cdot v_x|\right\}$ by $|Dg(x,y)\cdot v_x|$. In order to do so, we first simplify notation by setting $p = l(y)$ and 
\begin{align*}
    s = \rho_p(y)&,\ t = \rho_{p+1}(y),\\
    v_1 = Df_p(x)\cdot v_x, &\ u_1 = Df_{p+1}(x)\cdot v_x,\\
    v_2 = D(R_p(x)\cdot y)\cdot v_x, &\ u_2 = D(R_{p+1}(x)\cdot y)\cdot v_x.
\end{align*}
Putting $v = v_1+v_2,\ u = u_1+u_2$, we have $Dg(x,y)\cdot v_x = sv + tu$. In this notation,
\begin{align}\label{e:v-diff}
    |v_1-u_1| &\leq C\epsilon|v_1|,\\\label{e:v2-small}
    |v_2|,|u_2| &\leq C\epsilon|v_1|
\end{align}
by Lemma \ref{DT-Dsig-1} and (\ref{e:DRlimest}). We then want to prove the following claim:
\vspace{1ex}
\begin{claim}
$s|v_1| + t|u_1| \lesssim |sv+tu|$.
\end{claim}
\vspace{1ex}
\begin{claimproof}
Using (\ref{e:v2-small}), we get $s|v_1| \leq s|v_1| + s|v_2| \leq s|v|$ and similarly $t|u_1| \leq t|u|$. We now just need to show that $|sv| + |tu| \lesssim |sv+tu|$. By Lemma \ref{l:reverse-triangle}, this follows if we can show $\langle sv, tu \rangle \geq -\frac{1}{2}|sv||tu|$. Indeed, we have
\begin{align*}
    \langle sv, tu\rangle &= st\left(\langle v_1,u_1\rangle + \langle v_1, u_2 \rangle + \langle v_2, u_1 \rangle + \langle v_2, u_2 \rangle \right),\\ 
    &\geq st( |v_1|^2 - \langle v_1,u_1-v_1\rangle - C\epsilon|v_1|^2) \geq st(1-C\epsilon)|v_1|^2 \geq 0.
\end{align*}
\end{claimproof}

This completes the proof for $v = v_x$. If instead $v = v_y\in V$, then $Dg(z)\cdot v_y = v_y\cdot\frac{\partial g}{\partial y}(z)$ and the result follows from \eqref{e:Dgyhorizchange} in Lemma \ref{l:horizontal}. 
\end{proof}

\begin{lemma}[Vertical Estimates]\label{l:vertical}
Let $z,z^\prime,$ and $v$ be as in Corollary \ref{c:Dg-horizontal}. Let $p = l(y)$ and $m = n(y')$. If $\sum_{k=p}^{m}\epsilon_k'(f_k(x'))^2 \leq C\epsilon$ and $\ang(T\Sigma_p(f_p(x')), T\Sigma_m(f_m(x'))) \leq C\delta$, we have
\begin{equation}\label{e:Dfvertchange}
    \left|\sum_{k\geq 0} (\rho_k(y) - \rho_k(y^\prime))Df_k(x')\cdot v\right| \leq C\delta|Df_p(x')\cdot v|,
\end{equation}
\begin{equation}\label{e:DRvertchange}
    \left|\sum_{k\geq0}\left[\rho_k(y)D(R_k(x')\cdot y)\cdot v - \rho_k(y^\prime)D(R_k(x')\cdot y^\prime)\cdot v\right] \right| \leq C\delta|Df_p(x')\cdot v|,
\end{equation}
\begin{equation}\label{e:Dgvertchange}
    \left|\frac{\partial g}{\partial y}(x',y) - \frac{\partial g}{\partial y}(x',y')\right| \leq C\delta \left|\frac{\partial g}{\partial y}(x',y)\right|.
\end{equation}
\end{lemma}
\begin{proof}
We being by proving (\ref{e:Dfvertchange}). First, we can use Lemma \ref{DT-Dsig-1} to get
\begin{equation*}
    |Df_{p+1}(x')\cdot v - Df_p(x')\cdot v| \leq C\epsilon|Df_p(x')\cdot v|.
\end{equation*}
This implies
\begin{align}\label{e:Dfsmoothrho}
    \left|\sum_{k\geq0}\rho_k(y)Df_k(x')\cdot v - Df_p(x')\cdot v\right| &\leq \sum_{k\geq0}\rho_k(y)|Df_k(x')\cdot v - Df_p(x')\cdot v|\\\nonumber
    &\leq C\epsilon|Df_p(x')|
\end{align}
because $\rho_k(y)\not=0$ only for $k=p,p+1$ and $\sum_{k\geq0}\rho_{k}(y) = 1$. An identical argument gives (\ref{e:Dfsmoothrho}) with $y$ replaced by $y'$ and $p$ replaced by $m$. We now want a similar bound for $|Df_m(x')\cdot v - Df_p(x')\cdot v|$. For ease of notation, define $u = Df_p(x')\cdot v$ and $w = \prod_{k=p}^{m-1}D\sigma_k(f_k(x'))\cdot u$. We can then write
\begin{align*}
    |Df_m(x')\cdot v - Df_p(x')\cdot v| &= \left| \left[\prod_{k=p}^{m-1}D\sigma_k(f_k(x'))\right]\cdot u - u \right| = |w-u|.
\end{align*}
Recall that $\sum_{k=p}^{m}\epsilon_k'(f_k(x'))^2 \leq  C\epsilon^2$, and observe that the tree-like structure of CCBP net points from \eqref{e:CCBP-tree} combined with our constraint on $z'$ that $f_{n(y')}(x')\in V_{n(y')}^8$ means that $f_k(x')\in V_k^8$ for any $k \leq n(y')$. These facts mean we can use \eqref{DT-plane-comparison} to get
\begin{equation}
    \left|\prod_{k=p}^{m-1}D\sigma_k(f_k(x'))\right| \leq \prod_{k=p}^{m-1}1+C\epsilon_k'(f_k(x'))^2 \leq 1+C\epsilon^2.
\end{equation}
Hence, $\left||w|-|u|\right| \leq C\epsilon^2|u|$. Since $w\in T\Sigma_m(f_m(x')),\ u\in T\Sigma_p(f_p(x'))$, and we've assumed that $\ang(T\Sigma_p(f_p(x')), T\Sigma_m(f_m(x'))) \leq C\delta$, we have $\ang(w,u)\leq C\delta$. It follows that
\begin{equation}\label{e:small-tilt}
    |w-u| \leq C\delta|u|.
\end{equation}
Finally, using \eqref{e:Dfsmoothrho} and \eqref{e:small-tilt}, we see
\begin{align}
    \bigg|\sum_{k\geq 0} (\rho_k(y)& - \rho_k(y^\prime))Df_k(x')\cdot v\bigg|\nonumber\\
    &\leq  \left|\sum_{k\geq0}\rho_k(y)Df_k(x')\cdot v - Df_p(x')\cdot v \right| +  \left|\sum_{k\geq0}\rho_k(y^\prime)Df_k(x')\cdot v - Df_m(x')\cdot v \right|\nonumber\\
    &\qquad + |Df_p(x')\cdot v - Df_m(x')\cdot v |\nonumber\\
    &\leq C\epsilon|Df_p(x')\cdot v| + C\epsilon|Df_m(x')\cdot v| + C\delta|Df_p(x')\cdot v| \nonumber\\
    &\leq C\delta|Df_p(x')\cdot v|.\nonumber
\end{align}
The proof of (\ref{e:DRvertchange}) follows from (\ref{e:DRlimest}) and (\ref{e:Dfvertchange}). Indeed,
\begin{align*}
    \bigg|\sum_{k\geq0}\rho_k(y)D(R_k(x')\cdot y)&\cdot v - \rho_k(y^\prime)D(R_k(x')\cdot y^\prime)\cdot v \bigg| \\
    &\leq C\epsilon|Df_p(x')\cdot v| + C\delta|Df_p(x')\cdot v| + C\epsilon|Df_m(x')\cdot v|\\
    &\leq C\delta|Df_p(x')\cdot v|.
\end{align*}
Finally, we prove (\ref{e:Dgvertchange}). We have
\begin{align}\nonumber
    \left|\frac{\partial g}{\partial y}(x',y) - \frac{\partial g}{\partial y}(x',y')\right| &\leq \sum_{k\geq0}\bigg| \frac{\partial\rho_k}{\partial y}(y)\left\{f_k(x') + R_k(x')\cdot y\right\}\bigg| + \bigg|\frac{\partial\rho_k}{\partial y}(y^\prime)\left\{f_k(x') + R_k(x')\cdot y^\prime\right\}\bigg|\\\nonumber
    &\qquad + |(\rho_k(y) - \rho_k(y^\prime))R_k(x')\cdot e_{d+1}|\\\nonumber
    &=: \delta_1 + \delta_2 + \delta_3.
\end{align}
We first handle $\delta_1$ and $\delta_2$. We have
\begin{align}\nonumber
    \delta_1 \leq \left|\frac{\partial\rho_p}{\partial y}(y)\right|\left(|f_p(x') - f_{p+1}(x')| + |R_p(x') - R_{p+1}(x')||y|\right) \leq \frac{C}{r_p}(C\epsilon r_p + C\epsilon r_p) \leq C\epsilon
\end{align}
by \eqref{DT-sigma-id} and \eqref{DT-isometry-change}. A nearly identical calculation gives the same bound for $\delta_2$. We now handle $\delta_3$. First, notice that
\begin{align}\label{e:Rsmooth}
    \left|\sum_{k\geq0}\rho_k(y)R_k(x')\cdot e_{d+1} - R_p(x')\cdot e_{d+1}\right| &\leq \sum_{k\geq0}|\rho_k(y)||R_k(x')\cdot e_{d+1}-R_p(x')\cdot e_{d+1}|\leq C\epsilon
\end{align}
by \eqref{DT-isometry-change} and the fact that $\sum_{k\geq0}\rho_k(y) = 1$. Because $R_k(x')$ is an isometry such that $R_k(x')(T\Sigma_0(x')) = T\Sigma_k(f_k(x))$, $R_k(x')\cdot e_{d+1}$ is the unit normal to $T\Sigma_k(f_k(x'))$ so that 
\begin{equation}\label{e:Rverttriangle}
    |R_p(x')\cdot e_{d+1} - R_m(x')\cdot e_{d+1}| \leq C\ang(T\Sigma_p(f_p(x')), T\Sigma_m(f_m(x'))) \leq C\delta.
\end{equation}
Finally, (\ref{e:Rsmooth}) and (\ref{e:Rverttriangle}) imply
\begin{align}\nonumber
    \delta_3 &\leq \left|\sum_{k\geq0}\rho_k(y)R_k(x')\cdot e_{d+1} - R_p(x')\cdot e_{d+1}\right| + \left|\sum_{k\geq0}\rho_k(y
    ^\prime)R_k(x')\cdot e_{d+1} - R_m(x')\cdot e_{d+1}\right| \\\nonumber
    &\qquad + |R_p(x')\cdot e_{d+1} - R_m(x')\cdot e_{d+1}|\\\nonumber
    &\leq C\epsilon + C\epsilon + C\delta \leq C\delta\left|\frac{\partial g}{\partial y}(y^\prime)\right|.
\end{align}
where the final inequality uses \eqref{e:partial-g-big}.
\end{proof}

\begin{corollary}\label{c:Dg-vertical}
Let $z,z'$ be as in Lemma \ref{l:vertical}. For any vector $v\in H\cup V$, we have
\begin{equation*}
    |Dg(x',y)\cdot v - Dg(x',y')\cdot v| \leq C\delta|Dg(x',y)\cdot v|.
\end{equation*}
\end{corollary}
\begin{proof}
Suppose first that $v = v_x\in H$. Using (\ref{e:Dgx}), we compute
\begin{align}\label{e:Dgxvertchange}
    |Dg(x',y)&\cdot v_x - Dg(x',y')\cdot v_x|\\
    &= \bigg| \sum_{k\geq0}(\rho_k(y) - \rho_k(y'))Df_k(x')\cdot v_x + \rho_k(y)D(R_k(x')\cdot y)\cdot v_x - \rho_k(y')D(R_k(x')\cdot y')\cdot v_x \bigg|\\\nonumber
    &\leq C\delta|Df_p(x')\cdot v_x| \leq C\delta|Dg(x',y)\cdot v_x| \\
    &\leq C\delta(1+C\delta)|Dg(x,y)\cdot v_x| \leq C\delta|Dg(x,y)\cdot v_x| 
\end{align}
using (\ref{e:Dfvertchange}) and (\ref{e:DRvertchange}) in the first inequality, (\ref{e:DRlimest}) in the second, and (\ref{e:Dgxhorizchange}) in the third. If instead $v = v_y\in V$, then $Dg(x',y) = v_y\cdot\frac{\partial g}{\partial y}(x',y)$ and the result follows from \eqref{e:Dgvertchange} and \eqref{e:Dgyhorizchange}.
\end{proof}

Using Corollaries \ref{c:Dg-horizontal} and \ref{c:Dg-vertical}, we can prove Proposition \ref{p:DgDginv-expanded}.
\begin{proof}[Proof of Proposition \ref{p:DgDginv-expanded}]
Let $z' = (x',y')\in G_z^{M_0}$. We will show that for any vector $v\in H\cup V$,
\begin{equation}\label{Dg-special-vectors}
    |Dg(x,y)\cdot v - Dg(x',y)\cdot v| \leq C\delta|Dg(x,y)\cdot v|.
\end{equation}
The set $G_z^{M_0}$ is designed exactly so that $z'\in G_z^{M_0}$ implies that the hypotheses of Lemmas \ref{l:horizontal} and \ref{l:vertical} are satisfied. Hence, we can apply Corollaries \ref{c:Dg-horizontal} and \ref{c:Dg-vertical} so that
\begin{align*}
    |Dg(x,y)\cdot v -& Dg(x',y')\cdot v| \\
    &\leq |Dg(x,y)\cdot v - Dg(x',y)\cdot v| + |Dg(x',y)\cdot v - Dg(x',y')\cdot v|\\
    &\leq C\delta|Dg(x,y)\cdot v| + C\delta|Dg(x',y)\cdot v|\\
    &\leq C\delta|Dg(x,y)\cdot v|.
\end{align*}
By decomposing an arbitrary $v'\in\R^{d+1}$ as $v' = v_x + v_y$ where $v_x\in H $ and $v_y\in V$, we write
\begin{align}\nonumber
    |Dg(x,y)\cdot v' -& Dg(x',y')\cdot v'| \\
    &\leq |Dg(x,y)\cdot v_x - Dg(x',y')\cdot v_x| + |Dg(x,y)\cdot v_y - Dg(x',y')\cdot v_y| \nonumber \\
    \nonumber&\leq C\delta(|Dg(x,y)\cdot v_x| + |Dg(x,y)\cdot v_y|)\\
    &\leq C\delta|Dg(x,y)\cdot v'|\label{e:Dg-arbitrary-vectors}
\end{align}
where the final inequality follows from an application of the reverse triangle inequality in Lemma \ref{l:reverse-triangle}. We justify the application of the lemma by looking at the equations (\ref{e:Dgx}) and (\ref{e:Dgy}). These imply that the vector $Dg(x,y)\cdot v_x$ is nearly parallel to $T\Sigma_k(x)$ while the vector $Dg(x,y)\cdot v_y$ is nearly perpendicular to $T\Sigma_k(x)$ for some value of $k$ where the deviations described are on the order of $\epsilon$. This implies $|\langle Dg(x,y)\cdot v_x, Dg(x,y)\cdot v_y \rangle| \leq \frac{1}{2}|Dg(x,y)\cdot v_x|\cdot|Dg(x,y)\cdot v_y|$ so that the lemma applies. With this, we now compute,
\begin{align}\nonumber
    |Dg(z')\cdot Dg(z)^{-1}\cdot v' - v'| = |[Dg(z') - Dg(z)]\cdot Dg(z)^{-1}\cdot v'| &\leq C\delta|Dg(z)\cdot Dg(z)^{-1}\cdot v'|\\
    &= C\delta|v'|.\nonumber
\qedhere\end{align}
\end{proof}

This concludes the computations we need to bound the change in $Dg$. By integrating $Dg$ over paths in a quasiconvex domain $\Omega$, we get a companion result to Proposition \ref{p:DgDginv-expanded} that roughly states that the map $g|_{\Omega}$ is a $(1+C\delta)$-bi-Lipschitz perturbation of $Dg(z_0)$ for any $z_0\in\Omega$. More precisely, for any $z\in\R^{d+1}$ define
\begin{equation}
    L_{z_0}(z) = z_0 + Dg(z_0)(z-z_0).
\end{equation}
This is the affine transformation that approximates $g$ near $z_0$. Define 
\begin{equation}
    \varphi_{z_0} = g\circ L_{z_0}^{-1}
\end{equation}

\begin{proposition}\label{p:phibilip}
Let $z_0, M_0,\epsilon$ be as in Proposition \ref{p:DgDginv-expanded} where $z = z_0$. Let $\Omega\subseteq\R^{d+1}$ be a quasiconvex domain with constant $M_0$ such that $z_0\in\Omega\subseteq G_{z_0}^{M_0}$. Then the map $\varphi_{z_0}:L_{z_0}(\Omega)\rightarrow g(\Omega)$ is $(1+C\delta)$-bi-Lipschitz and
\begin{equation}\label{e:Dphi-id}
    |D\varphi_{z_0}(w) - I|\leq C\delta
\end{equation}
for all $w\in L_{z_0}(\Omega)$.
\end{proposition}
\begin{proof}
Because $w\in L_{z_0}(G_{z_0}^{M_0})$ by assumption, we get
\begin{equation*}
    D\varphi_{z_0}(w) = Dg(L_{z_0}^{-1}(w))\cdot DL_{z_0}^{-1}(w) = Dg(z)\cdot Dg(z_0)^{-1}
\end{equation*}
for $z=L_{z_0}^{-1}(w)\in G_{z_0}^{M_0}$. Equation \eqref{e:Dphi-id} follows from \eqref{p:DgDginv-expanded}. 

To prove that $\varphi_{z_0}$ is $(1+C\delta)$-bi-Lipschitz, let $\gamma:[0,1]\rightarrow\R^{d+1}$ be a path with $\gamma(0) = z_0,\ \gamma(1) = z$, and $\ell(\gamma)\lesssim_{M_0} |z_0-z|$. Put $\tilde{\gamma}(t) = L_{z_0}(\gamma(t))$, $w = L_{z_0}(z)$, and $w_0 = L_{z_0}(z_0) = z_0$. Observe that
\begin{equation}\nonumber
    L_{z_0}^{-1}(w) = z_0 + Dg(z_0)^{-1}(w-z_0).
\end{equation}
We estimate
\begin{align}\nonumber
    |\varphi_{z_0}(w) - \varphi_{z_0}(w_0)| &= \left| \int_0^1 D(g\circ L_{z_0}^{-1})(\tilde{\gamma}(t))\cdot\tilde{\gamma}^\prime(t)dt \right|\\\nonumber
    &= \left| \int_0^1 Dg(\gamma(t))\cdot DL_{z_0}^{-1}(\tilde{\gamma}(t))\cdot\tilde{\gamma}^\prime(t)dt \right|\\\nonumber
    &= \left| \int_0^1 Dg(\gamma(t))\cdot Dg(z_0)^{-1}\cdot\tilde{\gamma}^\prime(t)dt \right|\\\nonumber
    &= \left| w-w_0 + \int_0^1 \left[Dg(\gamma(t))\cdot Dg(z_0)^{-1} - I\right]\cdot\tilde{\gamma}^\prime(t)dt\right|.
\end{align}
Using the fact that $\gamma(t)\in G_{z_0}^{M_0}$ for all $t$, Proposition \ref{p:DgDginv-expanded} implies, on one hand
\begin{align}\nonumber
    |\varphi_{z_0}(w) - \varphi_{z_0}(w_0)| &\leq | w-w_0 | + \int_0^1 \left|Dg(\gamma(t))\cdot Dg(z_0)^{-1} - I\right|\cdot|\tilde{\gamma}^\prime(t)|dt\\\nonumber
    &\leq |w-w_0| + C\delta|Dg(z_0)|\cdot\ell(\gamma) \\\nonumber
    &\leq (1+C\delta)|w-w_0|.
\end{align}
On the other,
\begin{align}\nonumber
    |\varphi_{z_0}(w) - \varphi_{z_0}(w_0)| &\geq | w-w_0 | - \int_0^1 \left|Dg(\gamma(t))\cdot Dg(z_0)^{-1} - I\right|\cdot|\tilde{\gamma}^\prime(t)|dt\\\nonumber
    &\geq |w-w_0| - C\delta|Dg(z_0)|\cdot\ell(\gamma) \\\nonumber
    &\geq (1-C\delta)|w-w_0|
\end{align}
where the final inequality on both hands comes from the fact that $|w'-w| = |Dg(z)\cdot(z'-z)| \leq |Dg(z)|\cdot|z-z'|$ and our assumption that $\ell(\gamma)\lesssim_{M_0} |z-z'|$.
\end{proof}

\section{The proof of Theorem \ref{t:thmA}}\label{sec:thmA}
Fix constants $\rho = \frac{1}{1000},\ K = \frac{10^4}{\rho},\ M = \frac{10K}{\rho^2},$ and  $A_0 = \frac{1000\sqrt{d}}{c_0\rho}$. Throughout this section, assume that $\Omega\subseteq\R^{d+1}$ satisfies the hypotheses of Theorem \ref{t:thmA}. We begin by constructing a Reifenberg parameterization for $\partial\Omega\cap B(0,1)$.

\subsection{The CCBP adapted to $\scD$}\label{subsec: CCBP}
We want to form a CCBP adapted to the Christ-David lattice $\scD$ for $\partial\Omega$ with the aim of applying Theorem \ref{DT-thm}, David and Toro's bi-Lipschitz Reifenberg parameterization result. For any $k\in\Z$, let $s(k)$ be an integer such that
\begin{equation}\label{e:net-cube-radii}
    50\rho^{s(k)} \leq r_k < 50\rho^{s(k)-1}.
\end{equation}
We note that if $Q\in\scD_{s(k)}$, then this means
\begin{equation}\label{e:diamQ-rk}
    10\ell(Q) \leq r_k < 10\rho^{-1}\ell(Q)
\end{equation}
and
\begin{equation*}
    \frac{\rho}{5000}r_k \leq \frac{c_0}{10\rho}r_k \leq c_0\ell(Q) \leq \diam Q \leq \ell(Q) \leq \frac{r_k}{10}.
\end{equation*}
For any $k \geq 0$, define
\begin{align}\label{e:Y_k-CCBP}
    Y_k &= \{x_Q : Q\in\scD_{s(k)},\ Q\cap B(0,A_0)\not=\varnothing\},\\
    X_k &\in \Net(Y_k, r_k).
\end{align}
We enumerate $X_k = \{x_{j,k}\}_{j\in J_k}$ and often use the notation $x_{j,k} = x_Q = x_{Q_{j,k}}$. Let $P_0$ achieve the infimum in the definition of $b\beta_{\partial\Omega}(B(0,10A_0))$ and define
\begin{align*}
    B_{j,k} &= B(x_{j,k},r_k),\\
    P_{j,k} &= P_{Q_{j,k}}
\end{align*}
where $P_{Q_{j,k}}\ni x_{Q_{j,k}}$ are such that $\beta_{\partial\Omega}^{d,1}(2\rho^{-1}KB_{Q_{j,k}}) \lesssim \beta_{\partial\Omega}^{d,1}(2\rho^{-1}KB_{Q_{j,k}})$ as in the hypotheses of Lemma \ref{l:ep-by-beta}. We first show that $\epsilon_k'(x_{j,k})$ is controlled by $\epsilon(Q_{j,k})$.
\begin{lemma}\label{l:epsilonQ}
Fix $k \geq 0$ and $Q\in\scD_{s(k)}$. For any $z\in\R^{d+1}$ such that $|z - x_{Q}| < 200\rho^{-1}\ell(Q)$,
\begin{equation*}
    \epsilon_k'(z) \leq K\epsilon(Q).
\end{equation*}
\end{lemma}
\begin{proof}
If $z\not\in V_k^{10}$, then $\epsilon_k'(z) = 0$, and the desired inequality holds trivially. We now assume $z\in V_k^{10}$. We first show that the supremum in the definition of $\epsilon(Q)$ is over a larger collection of pairs of planes than that in the definition of $\epsilon_k'(z)$. Let $i\in J_k$ be such that $z\in 10B_{i,k}$. By \eqref{e:diamQ-rk},
\begin{equation*}
    |x_Q - x_{i,k}| < |x_Q - z| + |z - x_{i,k}| < 200\rho^{-1}\ell(Q) + 10r_k < 300\rho^{-1}\ell(Q_{i,k}) < \frac{K}{10}\ell(Q)
\end{equation*}
because $K \geq 10^4\rho^{-1}$ and $\ell(Q) = \ell(Q_{i,k})$. Therefore, $x_Q\in\frac{K}{10}B_{Q_{i,k}}$. If instead $z\in 11B_{i,k-1}$ for some $i\in J_{k-1}$, then 
\begin{equation*}
    |x_Q - x_{i,k-1}| < |x_Q - z| + |z - x_{i,k-1}| < 200\rho^{-1}\ell(Q) +  11r_{k-1} < 310\rho^{-1}\ell(Q_{i,k-1}) < \frac{K}{10}\ell\left(Q_{i,k-1}\right).
\end{equation*}
Therefore, $x_Q\in\frac{K}{10}B_{Q_{i,k-1}}$. In addition, for any admissible $x_{i,l}$ in the definition of $\epsilon_k'(z)$ we can write $100r_l \leq 1000\rho^{-1}\ell(Q_{i,l}) < K\ell(Q_{i,l})$ so that $100B_{i,l} \subseteq K B_{Q_{i,l}}$. Let $P_{i,k}$ and $P_{m,l}$ be planes which achieve the supremum in the definition of $\epsilon_k'(z)$. Then
\begin{equation*}
    d_{x_{m,l},100B_{m,l}}(P_{i,k},P_{m,l}) \leq \frac{K\ell(B_{Q_{m,l}})}{100r_l}d_{KB_{Q_{m,l}}}(P_{Q_{i,k}},P_{Q_{m,l}}) \leq Kd_{KB_{Q_{m,l}}}(P_{Q_{i,k}},P_{Q_{m,l}})
\end{equation*}
using the fact that $\ell(Q_{m,l}) < r_l$.
\end{proof}
Applying this result for $z = x_{j,k}$ shows that $\epsilon_k'(x_{j,k}) \lesssim \epsilon(Q_{j,k})$ which we can use to prove that the triple $\mathscr{Z} = (P_0,\{B_{j,k}\}, \{P_{j,k}\})$ is a CCBP.

\begin{lemma}\label{l:CubeCCBP}
$\mathscr{Z}$ is a CCBP.
\end{lemma}
\begin{proof}
We will use Lemma \ref{l:CCBP-AS}. First, we will show that for any $j\in J_0$, $\dist(x_{j,0}, P_0) \lesssim \epsilon$. Indeed, $x_{j,0} = x_Q$ for some $Q\in\scD_{s(0)}$ with $Q\cap B(0,A_0)\not=\varnothing$. Hence, $x_Q\in B(0,2A_0)\cap \partial\Omega$ so that $b\beta_{\partial\Omega}(B(0,10A_0), P_0) \leq \epsilon$ implies
\begin{equation*}
    \dist(x_Q,P_0) \lesssim b\beta(B(0,10A_0))\cdot 10A_0 \lesssim_d \epsilon.
\end{equation*}
Now, we fix $k > 0$ and $j\in J_k$ and prove the following claim:
\vspace{1ex}
\begin{claim}
There exists $i\in J_{k-1}$ such that $x_{j,k}\in 2B_{i,k-1}$
\end{claim}
\vspace{1ex}
\begin{claimproof}
Indeed, let $x_{j,k} = x_{Q_{j,k}}$. If $s(k) = s(k-1)$, then $Y_{k-1} = Y_k$ so that $x_{Q_{j,k}}\in Y_{k-1}$. The claim follows since $X_{k-1}$ is an $r_{k-1}$-net for $Y_{k-1}$. If instead $s(k) > s(k-1)$, then $x_{Q_{j,k}^{(1)}}\in Y_{k-1}$ so that there exists $i\in J_{k-1}$ such that $x_{Q_{j,k}^{(1)}}\in B_{i,k-1}$. We have
\begin{equation*}
    \ell\left(Q_{j,k}^{(1)}\right) = 5\rho^{s(k-1)} \leq 5\rho^{s(k) - 1} \leq r_{k-1}
\end{equation*}
so that
\begin{equation*}
    \dist(x_{Q_{j,k}}, x_{i,k-1}) \leq \dist(x_{Q_{j,k}},x_{Q_{j,k}^{(1)}}) + \dist(x_{Q_{j,k}^{(1)}}, x_{i,k-1}) \leq \ell\left(Q_{j,k}^{(1)}\right) + r_{k-1} \leq 2r_{k-1},
\end{equation*}
which proves $x_{Q_{j,k}}\in 2B_{i,k-1}$.
\end{claimproof}

By Lemma \ref{l:epsilonQ}, it suffices to show that $\epsilon(Q_{j,k}) \lesssim \epsilon$. But by the definition of $P_{Q_{j,k}}$ and Lemma \ref{l:ep-by-beta}, we have $\epsilon(Q_{j,k})\lesssim_{M,d}\beta^{d,1}_{\partial\Omega}(MB_{Q_{j,k}}) \lesssim \epsilon$.
\end{proof}

Since we've shown that $\mathscr{Z}$ is a CCBP, Theorem \ref{DT-thm} gives a Reifenberg parameterization $g:\R^{d+1}\rightarrow\R^{d+1}$ such that
$$g(P_0)\cap B(0,1) \supseteq \partial\Omega\cap B(0,1).$$
Without loss of generality, we can assume $P_0 = \R^d\times\{0\}$ and translate the Whitney decomposition $\mathscr{W}$ as in Definition \ref{def:whitney-cubes} to a new decomposition $\mathscr{W}'$ such that $W_0 = [-2,2]^d\times[4,8]\in\mathscr{W}'$. We have that
$$\Omega\cap B(0,1) \subseteq g([-2,2]^d\times[0,8])$$
because $\partial\Omega$ is contained in the closure of $\cup_kX_k$, so in practice we only need to consider the set $D(W_0)$ (See \eqref{e:whit-desc}.) of descendants of $W_0$ to cover $\Omega\cap B(0,1)$:
\begin{equation}\label{e:centered-Whitney}
    \mathscr{W}_0 = \{W\in\mathscr{W}' : W\in D(W_0)\}.
\end{equation}

We can now derive some useful properties of $g$.

\begin{lemma}[Properties of $g$]\label{l:g-prop}\ 
\begin{enumerate}[label=(\roman*)]
    \item For any $x\in[-2,2]^d\times\{0\}$ and $n\in\N$, 
    $$f_n(x)\in V_n^8.$$ \label{i:nearby-net} \label{i:net}
    \item For any $z = (x,y)\in[-2,2]^d\times[0,8]$
    \begin{equation}\label{e:g-dist-E}
        (1-C(d)\epsilon)|y| \leq \dist(g(z),\partial\Omega) \leq (1+C(d)\epsilon)|y|.
    \end{equation} \label{i:g-dist}
    \item For any $x\in [-2,2]^d\times\{0\}$ and $p,n\in\N$ with $p < n$, there exists a collection of cubes $Q_{n}\subseteq Q_{n-1} \subseteq\cdots\subseteq Q_{p}$ such that for any $k$ with $p \leq k \leq n,\ \dist(g(x,r_k), Q_k) \lesssim r_k$, and
    \begin{equation*}
        \sum_{k=p}^n\epsilon'(f_k(x))^2 \lesssim_{M,\rho,d} \sum_{k=p}^n\beta^{d,1}_{\partial\Omega}(MB_{Q_k})^2.
    \end{equation*}
    In particular, $g$ is $L'(L,\rho,M,d)$-bi-Lipschitz.\label{i:epsilon-bound}
\end{enumerate}  
\end{lemma}
\begin{proof}
    Let $z = (x,y)$ be as in \ref{i:g-dist} with $n = n(y)$ (See \eqref{e:def-g}.). We first prove \ref{i:g-dist} with the added hypothesis that $f_n(x)\in V_n^8$. We will then prove \ref{i:net}, which will complete the proof of \ref{i:g-dist}. 

    Observe that
    \begin{equation*}
        g(z) - f_n(x) = \sum_{k}\rho_k(y)\left\{f_k(x) - f_n(x) + R_k(x)\cdot y\right\}.
    \end{equation*}
    Recall that $|f_k(x) - f_n(x)|\lesssim \epsilon r_n$ by \eqref{e:fn-diff} using the fact that $\rho_k(y)\not=0$ only if $k\in \{n-1,n\}$ by \eqref{e:def-g}. We also know $R_k(x)$ is an isometry such that $R_k(x)\cdot y$ is a vector of norm $|y|$ that is orthogonal to the tangent plane $T_k(x)$ to $\Sigma_k$ at $f_k(x)$. The fact that $f_n(x)\in V_n^8$ implies the existence of $Q\in\scD_{s(n)}$ such that $|f_n(x) - x_Q| \leq 8r_n$. The fact that $b\beta_{\partial\Omega}(MB_Q,P_Q) \lesssim \epsilon$ combined with Lemma \ref{DT-loc-lip} and \eqref{DT-plane-comparison} implies
    \begin{equation}\label{e:Sigma-E-close2}
        d_{f_n(x),19r_n}(\Sigma_n,\partial\Omega) \leq d_{f_n(x),19r_n}(\Sigma_n,P_Q) + d_{f_n(x),19r_n}(P_Q,\partial\Omega) \lesssim \epsilon.
    \end{equation}
    We conclude $d_{f_n(x),19r_n}(T_n(x) + f_n(x),\partial\Omega) < C\epsilon$, which implies $(1-C\epsilon)|y| \leq \dist(g(z),\partial\Omega) \leq (1 + C\epsilon)|y|$ as desired.
    
    We now prove \ref{i:net} by induction on $n$. For the base case $n=0$, notice that $f_0(x) = x\in B(0,5\sqrt{d})\cap P_0$ so that $b\beta_{\partial\Omega}(B(0,10A_0), P_0) \leq \epsilon$ implies the existence of $y\in\partial\Omega$ with $\dist(y,x) \lesssim_{A_0}\epsilon$. There exists $Q_0\in\scD_{s(0)}$ such that $y\in Q_0$ and $\dist(Q_0,0) \leq 10\sqrt{d}$ so that $x_{Q_0}$ is a member of the set $Y_0$ (see \eqref{e:Y_k-CCBP}) from which the maximal net $X_0$ forming the CCBP is taken. Notice that
    \begin{equation*}
        |f_0(x) - x_{Q_0}| \leq |x - y| + |y - x_{Q_0}| \leq C\epsilon + \ell(Q_0) \leq 2r_0.
    \end{equation*}
    Hence, we are done if $x_{Q_0}\in X_0$. Otherwise, there exists $x_{Q_0'} \in X_0$ such that $|x_{Q_0} - x_{Q_0'}| \leq r_0$ so that $|f_0(x) - x_{Q_0'}| \leq 3r_0$ implying $f_0(x)\in V_0^3$. This proves the base case for \ref{i:net}. 

    For the inductive step, assume that $f_n(x)\in V_n^8$ for some $n\in\N$. Using \eqref{e:Sigma-E-close2}, we find $y\in\partial\Omega$ such that
    \begin{equation*}
        |f_{n+1}(x) - y| \leq |f_{n+1}(x) - f_{n}(x)| + |f_n(x) - y| \lesssim \epsilon r_{n+1}
    \end{equation*}
    and hence there exists $Q_{n+1}\in \scD_{s(n+1)}$ with $\dist(Q_{n+1},0)\leq 10\sqrt{d}$ such that $|f_{n+1}(x) - x_{Q_{n+1}}| \leq 2r_{n+1}$. By a similar argument to the base case, this finishes the proof of \ref{i:net}.

    To prove \ref{i:epsilon-bound}, we first claim that \ref{i:net} implies $f(x)\in \partial\Omega$. Indeed, the continuity of $\dist(\cdot,\partial\Omega)$ and the fact that $x_{j,n}\in\partial\Omega$ for any $x_{j,n}\in X_n$ give
    \begin{equation}
        \dist(f(x),\partial\Omega) = \lim_n\dist(f_n(x),\partial\Omega) \leq \lim_n 8r_n = 0.
    \end{equation}
    This proves the claim. This means there exists an infinite chain of (possibly repeating) cubes $Q_{0}\supseteq Q_{1} \supseteq \cdots \ni f(x)$ where $Q_k\in\scD_{s(k)}$. We claim that
    \begin{equation*}
        \epsilon_{k}'(f_k(x)) \lesssim \epsilon(Q_k) \lesssim \beta^{d,1}_{\partial\Omega}(MB_{Q_k}).
    \end{equation*}
     Indeed, by Lemmas \ref{l:epsilonQ} and \ref{l:ep-by-beta}, we only need to show that $|f_k(x) - x_{Q_k}| < 200\rho^{-1}\ell(Q_k)$ to verify the first inequality. But we have $|f_k(x) - f(x)|\leq C\epsilon r_k$ by \eqref{e:f-diff} so that
    \begin{align*}
        |f_k(x) - x_{Q_k}| &\leq |f_k(x) - f(x)| + |f(x) - x_{Q_k}|\\
        &\leq C\epsilon r_k + 10r_k + \ell(Q_k) \leq (C\epsilon + 100)\rho^{-1}\ell(Q_k) +    \ell(Q_k) \leq 102\rho^{-1}\ell(Q_k)
    \end{align*}
    as desired. Because the set $\{n:s(k) = s(n)\}$ has a uniformly bounded number of elements in terms of $\rho$, it follows that
    \begin{equation*}
        \sum_{k=p}^n \epsilon_k'(f_k(x))^2 \lesssim_{M,\rho} \sum_{k=p}^n\epsilon(Q_k)^2 \lesssim_{M,\rho,d} \sum_{k=p}^n \beta^{d,1}_{\partial\Omega}(MB_{Q_k})^2.
    \end{equation*}
    The claim that $\dist(g(x,r_k), Q_k) \lesssim r_k$ follows from \ref{i:g-dist}. By the hypotheses on $\Omega$, we have $\sum_{k=1}^\infty \epsilon'(f_k(x))^2 \lesssim \sum_{f(x)\in Q} \beta^{d,1}_{\partial\Omega}(MB_{Q})^2 \leq L$ so that $g$ is $L'(L,\rho, M, d)$-bi-Lipschitz by Theorem \ref{DT-thm}.
\end{proof}

Now that we know that $g$ is $L'$-bi-Lipschitz, we have an absolute bound on the degree it stretches the cubes in $\scW_0$. Since we would like each element of our stopping time regions to not be stretched too far (Otherwise, they would see too much of $\partial \Omega$.), we now refine our cubes into thin rectangular Whitney boxes in which the side length of the boxes in the first $d$ coordinate directions is allowed to vary.
\begin{definition}[Whitney boxes]\label{def:whitney-boxes}
    We define the set of $p$-th order \textit{Whitney boxes} by
    $$\mathscr{R}_p = \left\{[k_12^{-p-n},(k_1+1)2^{-p-n}]\times\cdots\times[k_d2^{-p-n},(k_d+1)2^{-p-n}]\times[2^{-n},2^{-n+1}]:k_1,\ldots,k_d,n\in\Z\right\}.$$
    These are like Whitney cubes, but they have lengths along the first $d$ coordinate directions contracted by a factor of $2^{p}$. Given $R\in\mathscr{R}_p$, we call $\ell(R) = 2^{-p-n}$ the \textit{side length} and $h(R) = 2^{-n} = \dist(R,\R^d)$ the \textit{height} so that
    $$\ell(R) = 2^{-p}h(R).$$ 
    Any collection of Whitney boxes has a tree structure induced by the same partial order as in Definition \ref{def:whitney-cubes}. We set $\mathscr{R} = \cup_p\mathscr{R}_p$.
\end{definition}

Now, we define $p(L')\in\Z$ such that
\begin{equation}\label{e:def-of-p}
    2^{p-1} \leq L' < 2^p,
\end{equation}
and we replace $\scW_0$ with
\begin{equation}\label{e:Rw}
    \scR_w = \{R\in\scR_p : \exists W\in\scW_0,\ R\subseteq W\}.
\end{equation}
That is, $\scR_w$ is the set of Whitney boxes $R$ with $\ell(R) = 2^{-p}h(R)$ which are contained in some member of $\scW_0$. This ensures that
\begin{equation}\label{e:side-length-stretch}
    L'\ell(R) = L'2^{-p}h(R) \leq h(R)
\end{equation}
so that $g$ does not stretch $R$ across too far of a region on the scale of $h(R)$. 

We say more about the shape of image boxes in the following lemma:

\begin{lemma}[Image boxes]\label{l:image-boxes}
    For any $W\in\scR_w$, we have
    \begin{equation}\label{e:im-dist-omega}
        (1-C\epsilon)h(W) \leq \dist(g(W),\partial\Omega) \leq (1+C\epsilon)h(W),  
    \end{equation}
    and
    \begin{equation}\label{e:im-diam-height}
        (1-C\epsilon)h(W) \leq \diam g(W) \leq 5\sqrt{d} h(W).
    \end{equation}
    There exists constants $C_0(L'),C_1(d)$ such that
    \begin{equation}\label{e:im-balls}
        B(g(c_W), C_0^{-1}h(W)) \subseteq g(W) \subseteq B(g(c_W), C_1h(W))
    \end{equation}
    where $c_W$ is the center of $W$.
\end{lemma}
\begin{proof}
    We first note that \eqref{e:im-dist-omega} follows from \eqref{e:g-dist-E} and the fact that $\dist(W,\R^d) = h(W)$ by definition. To prove \eqref{e:im-diam-height}, let $z,z'\in R$ with $z = (x,y), z' = (x',y')$. We write
    \begin{equation}\label{e:vert-diff}
        |g(z) - g(z')| \leq |g(x,y) - g(x,y')| + |g(x,y') - g(x',y')|,
    \end{equation}
    and we will control each term separately. The first term can be written as
    \begin{align}\label{e:vertical-diff}
        |g(x,y) - g(x,y')| = \left|\sum_{k\geq 0}\rho_k(y)\{f_k(x) + R_k(x)(y)\} - \rho_k(y')\{f_k(x) + R_k(x)(y')\}\right|.
    \end{align}
    Now, the fact that $z,z'\in R\in\scR_w$ means that $|y-y'| \leq h(R) = 2^{-m}$ for some $m$. Since $r_k = 10^{-k}$ and $\rho_k$ is supported on $[r_k,20r_k]$, we claim there are at most two values of $k$ such that either $\rho_k(y)\not=0$ or $\rho_k(y') \not = 0$. Indeed, if $20r_k = 2^{-m}$, which is the $y$-value of points on the bottom of $R$, then $r_{k-2} = 100r_{k} > 40r_k = 2h(R)$ so that $\rho_{k-2} = 0$ on $[h(R),2h(R)]$. Set $n = n(y)$. Using \eqref{e:fn-diff} and \eqref{DT-isometry-change}, we have
    \begin{equation}
        \left|\sum_{k\geq 0}\rho_k(y)f_k(x) - f_{n}(x)\right| = \left|\sum_{k\geq 0}\rho_k(y)(f_k(x) - f_{n}(x))\right|  \leq C\epsilon r_{n}
    \end{equation}
    and
    \begin{equation}
        \left|\sum_{k\geq 0}\rho_k(y)R_k(x)(y) - R_{n}(x)(y)\right| = \left|\sum_{k\geq 0}\rho_k(y)(R_k(x)(y) - R_{n}(x)(y))\right| \leq C\epsilon r_{n}.
    \end{equation}
    Analogous inequalities hold with $y$ replaced by $y'$ since $|n(y) - n(y')| \leq 1$. Combining these with \eqref{e:vert-diff} by subtracting $f_n(x)$ in each of the two outer terms on the right of \eqref{e:vertical-diff} and adding and subtracting $R_n(x)(y)$ and $R_n(x)(y')$ gives
    \begin{equation}\label{e:final-vert-bound}
        |g(x,y) - g(x,y')| \leq C\epsilon r_n + |R_n(x)(y) - R_n(x)(y')| \leq C\epsilon r_n + |y-y'| \leq 2h(R)
    \end{equation}
    using that $R_k(x)$ is an isometry, $r_n\simeq h(R)$, and $|y-y'|\leq h(R)$. Now, using \eqref{e:final-vert-bound} and the fact that $g$ is $L'$-Lipschitz,
    \begin{align*}
        |g(x,y) - g(x',y')| &\leq |g(x,y)-g(x,y')| + |g(x,y') - g(x',y')|\\
        &\leq 2h(R) + L'|x-x'|\\
        &\leq 2h(R) + L'\sqrt{d}\ell(R)\\
        &\leq 5\sqrt{d}h(R)
    \end{align*}
    The lower bound follows from \eqref{e:g-dist-E} by considering the distance between images of points in the lower and upper faces of $W$. To prove \eqref{e:im-balls}, we first observe that each box $W\in\scR$ contains a small ball $B(c_W, c(L')h(W))$ around its center. Since $g$ is $L'$-bi-Lipschitz, we get a larger constant $C_0(L')$ such that the lower containment in \eqref{e:im-balls} holds. The existence of $C_1(d)$ as in the upper containment follows from the upper inequality in \eqref{e:im-diam-height}. We also note that because $g$ is injective and distinct boxes $R,W\in\scR_w$ have disjoint interiors, we have
    \begin{equation}\label{e:images-have-disjoint-centers}
        B(g(c_W), C_0^{-1}h(W)) \cap B(g(c_R), C_0^{-1}h(R)) = \varnothing.
    \qedhere\end{equation}
\end{proof}

\subsection{Whitney coronizations and the Lipschitz decomposition}
In item \ref{i:epsilon-bound} of Lemma \ref{l:g-prop}, we showed that the mapping $g$ was bi-Lipschitz so that $\partial\Omega$ is locally a bi-Lipschitz image. Hence, $\partial\Omega$ is uniformly $d$-rectifiable and therefore has an $(M,\epsilon,\delta)$-graph coronization for arbitrarily small values of $\epsilon$ and $\delta$  by Proposition \ref{p:ur-equiv}. Take $\epsilon'(d,L), \delta'(d,L) > 0$ fixed later sufficiently small and let $\scC = (\scG,\scB,\scF)$ be an $(M,\epsilon',\delta')$-graph coronization for $\partial\Omega$.

The plan for the proof of Theorem \ref{t:thmA} is to construct a ``coronization'' of $\scR_w$ which ``follows'' the coronization $\scC$ of $\partial\Omega$. That is, we will construct a triple
\begin{equation*}
    \scC_w = (\scG_w,\scB_w,\scT)
\end{equation*}
of good boxes, bad boxes, and stopping time regions $\scT = \{T\}_{T\in\scT}$ (see Remark \ref{rem:Whitney-stop-region}) partitioning $\scG_w$ such that for each $T\in\scT$, there exists some $S\in\scF$ such that the images of all boxes in $T$ under $g$ are ``surrounded'' in scale and location by cubes in $S$. We will need the following notion of ``closeness''.
\begin{definition}[$A$-close subsets]\label{def:Aclose}
     We call two subsets $W,R\subseteq\R^{d+1}$ \textit{$A$-close} (as in \cite{DS93} pg. 59) if the following hold:
\begin{align*}
    \frac{1}{A}\diam W &\leq \diam R \leq A\diam W, \\
    \dist(W,R) &\leq A(\diam W + \diam R).
\end{align*}
We will also use the notation
\begin{equation*}
    W \simeq_A R
\end{equation*}
when $W$ is $A$-close to $R$.
\end{definition}

\begin{definition}[$g$-Whitney coronizations]\label{def:whitney-coronization}
     Let $g,\scR_w$ be as above. We now give a partition of $\scR_w$ into a bad set $\scB_w$ and good set $\scG_w$ that picks out all Whitney boxes whose images under $g$ are ``$A_0$-surrounded'' by surface cubes within a single stopping time region $S\in\scF$:
    \begin{align}\label{e:G0}
        \scG_w &= \left\{W\in \scR_w : \exists S\in\scF,\ \forall Q\in\scD \text{ such that }Q\simeq_{A_0} g(W) \text{ we have }Q\in S\right\},\\\label{e:B0}
        \scB_w &= \scR_w \setminus \scG_w.
    \end{align}
    (See Definition \ref{def:Aclose}.) Given a root box $W\in\scG_w$, we define the stopping time region $T_W$ with top cube $W$ to be the maximal sub tree of $D(W)\cap\scG_w$ such that for any $R\in T_W$, either all of its children are in $T_W$, or none are. Any such stopping time region has associated minimal cubes and stopped cubes
    \begin{align*}
        m(T_W) &= \{R\in T_W:R \text{ has a child not in $T_W$} \},\\
        \Stop(T_W) &= \{R\in\scR_w:R \text{ has a parent in $m(T_W)$}\}. 
    \end{align*}
    We initialize our construction with the lattice $\scR_w$ and triple $(\scG_w,\scB_w,\scT_0 = \varnothing)$. Given the $k$-th stage stopping time collection $\scT_k$, we choose a maximal root box $W\in\scG_w\setminus \cup_{T\in\scT_k}T$ and form the stopping time region $T_W$. We set $\scT_{k+1} = \scT_k\cup\{T_W\}$. Repeating this process inductively, we obtain a partition $\scT = \bigcup_{k=1}^\infty \scT_k$ of $\scG_w$ into coherent stopping time regions. This gives the triple $\scC_w = (\scG_w, \scB_w, \scT)$. We call $\scC_w$ the \textit{$g$-Whitney coronization} of $\scR_w$ with respect to $\scC = (\scG,\scB,\scF)$. Given any $T\in\scT$, we also refer to the root box as $W(T)$.
\end{definition}
\begin{remark}[improving the stopping time]\label{rem:better-refine}
    In this construction, we used Whitney boxes with side length $\ell(R) = 2^{-p}h(R)$ to ensure that for any $R\in\scR_w$, $\diam g(R) \lesssim_d h(R)$. Without this condition or some other method of controlling the size of image boxes, we could have $z=(x,y),\ z=(x',y')$ such that $h(R) \ll |g(z) - g(z')|$ which would cause us to lose control of the change in $Dg$ across $R$ we desire in Lemma \ref{l:Dg-change} below.
    
    What we really want are image pieces of some kind that satisfy the conclusions of Lemma \ref{l:image-boxes} along with small parameterization derivative change across the pieces as in Lemma \ref{l:Dg-change} below. If one could form reasonable stopping time domains out of similar pieces whose images satisfy the conclusions of \ref{l:image-boxes} with constant $C_0$ dependent only on $d$, this would essentially prove a version of Theorem \ref{t:thmA} without hypothesis \ref{item:beta-squared}. If $g$ were $K(d)$-quasiconformal, then this could likely be accomplished by adding modifications to the stopping time by dynamically either combining or cutting apart children boxes for a given top box $W(T)$ along coordinate directions according to the size and shape of $Dg$ inside. In general though, $Dg$ can distort boxes so badly that coordinate boxes cannot be mapped forward appropriately in general, so one would need to devise a better way of partitioning the domain into pieces that are mapped forward well under a more wild parameterization.
\end{remark}

We will use $\scC_w$ to break up $\scR_w$ into regions which will map forward under $g$ to the Lipschitz graph domains we desire as in the conclusion of Theorem \ref{t:thmA}.
\begin{definition}[Stopping time domains]\label{def:stopping-time-domains}
    Let $\scC_w = (\scG_w,\scB_w,\scT)$ be a $g$-Whitney coronization as above. For each $T\in\scT$, we define a \textit{stopping time domain}
    \begin{equation*}
        \mathcal{D}_T = \bigcup_{W\in T}W.
    \end{equation*}
    For each $W\in\scB_w$, we note that $\ell(W) = 2^{-p}h(W)$ where $p$ is as in \eqref{e:def-of-p} and define a collection of associated trivial domains by chopping $W$ into $2^{p}$ cubes of common side length $\ell(W)$. That is, assuming $W = [0,\ell(W)]^d\times[h(W),2h(W)]$, we set
    \begin{equation*}
        \scL_W = \{[0,\ell(W)]^d\times[h(W)+k\ell(W),h(W)+(k+1)\ell(W)]: 0\leq k \leq 2^{p}-1\}.
    \end{equation*}
    The collection $\mathscr{L}' = \{\mathcal{D}_T\}_{T\in\scT}\cup\bigcup_{W\in\scB_w}\scL_W$ is a partition of $\bigcup_{W\in\scR}W = [-2,2]^d\times[0,8]$ up to finite overlaps on boundaries. Each cube domain $R_W\in\scL_W$ is $C(d)$-Lipschitz graphical, but the domain $\mathcal{D}_T$ is not Lipschitz graphical in general. However, $T$ consists of a coherent collection of boxes (See \ref{def:stopping-time-regions} and Remark \ref{rem:Whitney-stop-region}.) of a given ratio of side length to height $\ell(R) = 2^{-p}h(R)$. Therefore, applying a dilation $A_p$ by a factor of $2^p$ in the first $d$ coordinates gives a domain $A_p(\sD_T)$ consisting of cubes. Proposition \ref{p:DT-lip-graph} then gives the existence of a $d$-rectifiable, Ahlfors upper $d$-regular set $\Sigma_T$ such that there exists a collection of subdomains $\{\sD_T^j\}_{j\in J_T}$ fo $\sD_T$ satisfying
    \begin{equation*}
        A_p(\mathcal{D}_T) \setminus \Sigma_T = \bigcup_{j\in J_T}A_p(\mathcal{D}_T^j)
    \end{equation*}
    where $\{A_p(\mathcal{D}_T^j)\}_{j\in J_T}$ is a collection of $C(d)$-Lipschitz graph domains with disjoint interiors. By Lemma \ref{l:linear-graph-domain}, we then get the existence of a constant $C'(L,d)$ such that $\sD_T^j$ is a $C'(L,d)$-Lipschitz graph domain. We then finally define
    \begin{equation*}
        \mathscr{L}_0 = \left\{ \mathcal{D}_T^j \right\}_{T\in\scT,\ j\in J_T}\cup \bigcup_{W\in\scB_w}\scL_W.
    \end{equation*}
\end{definition}
We can now define the collection of Lipschitz graph domains $\scL$ as desired in Theorem \ref{t:thmA}:
\begin{definition}[Lipschitz decomposition]\label{def:domain-decomp}
    Let $\mathscr{L}_0$ be as in Definition \ref{def:stopping-time-domains}. We define the \textit{Lipschitz decomposition} of $\Omega\cap B(0,1)$ as
    \begin{equation}\label{e:def-of-Lambda}
        \mathscr{L} = \{g(\mathcal{D}) : \mathcal{D}\in\mathscr{L}_0\}.
    \end{equation}
    Additionally, for any $T\in\scT$ and $j\in J_T$ we define
    \begin{align*}
        \Omega_T &= g(\sD_T),\\
        \Omega_T^j &= g(\sD_T^j)
    \end{align*}
    so that we can equivalently write
    \begin{equation}
        \mathscr{L} = \{\Omega_T^j\}_{T\in\scT,\ j\in J_T} \cup \{g(R)\}_{W\in\scB_w,\ R\in\scL_W}.
    \end{equation}
\end{definition}
In order to prove Theorem \ref{t:thmA}, it suffices to prove Propositions \ref{p:RF-Lipschitz} and \ref{p:RF-surface-msr} below.

\begin{proposition}\label{p:RF-Lipschitz}
    Let $\Omega$ be as in Theorem \ref{t:thmA} and $\mathscr{L} = \{\Omega_j\}_{j\in J_{\scL}}$ be as in \eqref{e:def-of-Lambda}. There exists $L_1(L,d,\epsilon) > 0$ such that for any $j\in J_{\scL},\ \Omega_j$ is an $L_1$-Lipschitz graph domain.
\end{proposition}
To prove Proposition \ref{p:RF-Lipschitz}, we use the fact that the graph coronization $\scC$ of $\partial\Omega$ and the Whitney coronization $\scC_w$ of Definition \ref{def:whitney-coronization} adapted to $\scC$ were chosen so that $Dg$ is very close to being constant on any given domain $\mathcal{D}\in\mathscr{L}_0$. This uses the explicit calculations for $Dg$ given in Proposition \ref{p:DgDginv-expanded}. This means $g$ distorts $\mathcal{D}$ only slightly such that $g(\mathcal{D})$ remains a Lipschitz graph domain (see Proposition \ref{p:slow-vary-lip-graph}). The refinement of Whitney cubes to smaller Whitney boxes ensures that $\diam(g(W)) \simeq_d h(W)$ holds for any box $W$ so that $g(W)$ does not stretch across too long of a region of $\partial\Omega$ compared to its distance from $\partial\Omega$. If $W\in\scB_W$, then this ensures that $Dg|_W$ varies at a rate determined at worst by the Reifenberg flatness constant $\epsilon$. Because in this case, $W$ is divided into the set $\scL_W$ of cubes, which are $C(d)$-Lipschitz graph domains (note $C$ is independent of $L$), $g$ maps them forward to Lipschitz graph domains given that $\epsilon$ is fixed small enough with respect to $d$.

The construction of stopping time regions $\scT$ proceeds in such a way that any $T\in\scT$ is a coherent collection of (potentially long and thin) Whitney boxes such that the change in $Dg$ on $\mathcal{D}_T$ is controlled by the geometry of $\partial\Omega$ inside some surface stopping time region $S\in\scF$. These regions are defined such that $\partial\Omega$ looks like a Lipschitz graph with small constant $\epsilon'(L,d)$ and angle variation $\delta'(L,d)$ on the scale of cubes in $S$ from which we derive that $Dg|_{\mathcal{D}_T}$ varies at a rate determined by $\delta'(L,d)$ (See Lemma \ref{l:Dg-change}), giving Lipschitz graphicality for domains in $\{g(\mathcal{D}_T^j)\}_{j\in J_T,T\in\scT}$ by Proposition \ref{p:slow-vary-lip-graph} again as long as $\epsilon',\delta'$ are fixed small enough with respect to $L$ and $d$.

\begin{proposition}\label{p:RF-surface-msr}
     Let $\Omega$ be as in Theorem \ref{t:thmA} and $\mathscr{L} = \{\Omega_j\}_{j\in J_{\scL}}$ be as in \eqref{e:def-of-Lambda}. For any $y\in\partial\Omega\cap B(0,1)$ and $0 < r < 1$, we have
    \begin{equation}\label{e:RF-surface-msr}
        \sum_{j\in j_\scL}\sH^d(\partial\Omega_j\cap B(y,r)) \lesssim_{\epsilon,L,d} r^d.
    \end{equation}
\end{proposition}

To prove \ref{p:RF-surface-msr} we use the fact that the Whitney coronization is chosen in such a way that the images of boxes in the bad set $\scB_w$ have surface measure controlled by the measure of the $A_0$-close bad cubes $\scB$ or cubes in $\scD$ on the ``edges'' of stopping time regions which we collect in the set $\scB_e$ in \eqref{e:bad-scD} below. These cubes form a Carleson set (See Lemma \ref{l:scB-packing}.). This gives Carleson packing type estimates for the surface measure of the image cubes $\{g(R_W)\}_{W\in\scB_w,\ R_W\in\scL_W}$. Because the only time we stop in the construction of $T\in\scT$ is when we hit some $W\in\scB_w$, the surface measure of domains in $\{g(\mathcal{D}_T)\}_{T\in\scT}$ is controlled by the measure of nearby cubes in $\scB_e$. The fact that $g$ is bi-Lipschitz and preserves distances to the boundary means that the family $\{g(W)\}_{W\in\scR_w}$ behaves in many ways like a Whitney decomposition itself (see Lemma \ref{l:boxes-whitney}) so that we can bound the number of image boxes which are $A_0$-close to any fixed bad cube $Q\in\scB_e$, giving the desired Carleson type estimates.

\subsection{Lipschitz bounds for Theorem \ref{t:thmA}}
The goal of this section is to prove Proposition \ref{p:RF-Lipschitz}. The following lemma allows us to control the change in $Dg$ on any stopping time domain $T$.

\begin{lemma}[Variation of $Dg$]\label{l:Dg-change}
    For any $T\in\scT$ and $z,w\in \mathcal{D}_T$, we have
    \begin{equation}\label{e:Dg-change}
        |Dg(z)\cdot Dg(w)^{-1} - I| \leq C\delta'.
    \end{equation}
\end{lemma}
\begin{proof}
    First, fix some $T\in\scT$ and without loss of generality normalize so that $W(T) \subseteq W_0 = [-2,2]^d\times[4,8]$. We want to apply Proposition \ref{p:DgDginv-expanded} with some choice of $M_0\lesssim_d 1$ and $z = (x,y)$ a point in the top face of $W(T)$ by showing that $\mathcal{D}_T\subseteq G_{z}^{M_0}$. So, let $z' = (x',y')\in R\in T$ and let $n = n(y'),\ p = l(y)$. The choice of $z$ means that $|y'|\leq |y|$, and we know $f_{n(y)}(x)\in V_{n(y)}^8$ and $f_{n(y')}(x')\in V_{n(y')}^8$ by Lemma \ref{l:g-prop} item \ref{i:net} using the fact that $D_T\subseteq[-2,2]^d\times[0,8]$ because every box composing $D_T$ is a member of $\scR_w$, (See \eqref{e:Rw}.) which consists of box refinements of members of $\scW_0$ (See \eqref{e:centered-Whitney}.). This shows $z'\in N_z$, so we now only need to prove the following three statements to show $z'\in G_z^{M_0}$:
    \begin{enumerate}[label=(\roman*)]
        \item $|f_p(x) - f_p(x')| \lesssim_d r_p$, \label{i:base-dist}
        \item $\sum_{k=p}^n \epsilon'_k(f_k(x'))^2 \lesssim \epsilon'$, \label{i:epsilon_k-small}
        \item $\ang(T_n(x'), T_p(x')) \lesssim \delta'$. \label{i:angle-small}
    \end{enumerate}

    We begin by observing that \ref{i:base-dist} follows from the fact that $f_p$ is $L'$-bi-Lipschitz due to Remark \ref{rem:f-bilip} so that
    \begin{equation*}
        |f_p(x) - f_p(x')| \leq L'|x-x'| \leq 2L'\sqrt{d}\ell(W(T)) \lesssim_d h(W(T)) \lesssim r_p
    \end{equation*}
    using \eqref{e:side-length-stretch}. To prove \ref{i:epsilon_k-small}, let $Q_p\supseteq Q_{p+1} \supseteq\cdots\supseteq Q_n$ be the cubes given by Lemma \ref{l:g-prop}. For any $k$ with $p\leq k \leq n$ the fact that $\dist(g(x',r_k), Q_k)\lesssim r_k$ means that $(x',r_k)\in R \leq W(T)$ with 
    \begin{align*}
        \diam Q_k \geq c_0\ell(Q_{k+1}) \geq \frac{c_0\rho}{10}r_k \geq \frac{c_0\rho}{200}h(R) \geq \frac{c_0\rho}{1000\sqrt{d}}\diam(R) = A_0\diam R
    \end{align*}
    so that $Q_k \simeq_{A_0} R$. This means there exists $S\in \scF$ such that $Q_k\in S$ for any $k$ by the definition of the stopping time region $T$. We conclude that
    \begin{equation*}
        \sum_{k=p}^n\epsilon'_k(f_k(x'))^2 \lesssim \sum_{k=p}^n \beta^{d,1}_{\partial\Omega}(MB_{Q_k})^2 \lesssim \epsilon'.
    \end{equation*}
    
    To prove \ref{i:angle-small}, observe that
    \begin{equation*}
        \ang(T_p(x'), T_n(x')) \leq \ang(T_p(x'), P_{Q_p}) + \ang(P_{Q_p}, P_{Q_n}) + \ang(P_{Q_n}, T_n(x')) \lesssim \epsilon' + \delta' + \epsilon' \lesssim \delta'
    \end{equation*}
    where we used Lemma \ref{DT-loc-lip} and the fact that $Q_p,Q_n\in S$.
\end{proof}

Using the results of Section \ref{sec:lipschitz-graph-domains}, we can now prove Proposition \ref{p:RF-Lipschitz}.

\begin{proof}[Proof of Proposition \ref{p:RF-Lipschitz}]
    Every domain in $\scL$ is either of the form $g(\sD_T^j)$ for some $T\in\scT$ and $j\in J_T$ or $g(R_W)$ for some $W\in\scB_w,$ and $R\in\scL_W$. We first consider domains of the first form.
    
    Let $T\in\scT$ and let $A_p:\R^{d}\times\R\rightarrow\R^{d+1}$ be given by $A_p(x,y) = (2^px,y)$. By definition, the image stopping time region $\sD_T' = A_p(\sD_T)$ is composed of cubes and Proposition \ref{p:DT-lip-graph} implies there exists a constant $L_0(d)$ such that $\sD_T'$ has a decomposition into $L_0$-Lipschitz graph domains which passes to a decomposition of $\sD_T$ into $L_0'(d,L')$-Lipschitz graph domains $\{\sD_T^j\}_{j\in J_T}$ by applying $A_p^{-1}$. Now, using Lemma \ref{l:Dg-change}, we see \eqref{e:Dg-change} holds on $\sD_T$ so that by taking $\epsilon'(L',d), \delta'(L',d)$ sufficiently small, Proposition \ref{p:slow-vary-lip-graph} implies $g(\sD_T^j)$ is an $L_1(L',d)$-Lipschitz graph domain.

    Now, let $W\in\scB_w$ and $R_W\in\scL_W$. The proof of Lemma \ref{l:Dg-change} shows that $|Dg(z)\cdot Dg(w)^{-1} - I| \leq C\epsilon$ using only the fact that $\partial\Omega$ is $(\epsilon,d)$-Reifenberg flat. Since $R_W$ is a cube, it is a $C(d)$-Lipschitz graph domain so that Proposition \ref{p:slow-vary-lip-graph} implies $g(R_W)$ is a $C'(d)$-Lipschitz graph domain as long as $\epsilon$ is sufficiently small with respect to $d$.
\end{proof}

\subsection{Surface area bounds for Theorem \ref{t:thmA}}

We now focus on proving Proposition \ref{p:RF-surface-msr}. We will justify the name coronization by proving Carleson estimates for the $g$-Whitney coronization which will imply the desired estimates for our domains.

\begin{definition}[$C_0$-Whitney family]\label{def:whit-family}
    Let $\Omega_0\subseteq\R^{d+1}$ be a domain and let $C_0 \geq 1$. We say that a collection $\scV$ of subsets of $\Omega_0$ is a \textit{$C_0$-Whitney family} if for every $V\in \scV$, we have
    \begin{align}\label{e:whit-dist}
        C_0^{-1}\diam V \leq \dist(V,\Omega_0^c) \leq C_0\diam V,
    \end{align}
    there exists $c_V\in V$ such that
    \begin{equation}\label{e:whit-ball}
        B(c_V, C_0^{-1}\diam V) \subseteq V,
    \end{equation}
    and, if $V\not= V'$, then
    \begin{equation}
        B(c_V,C_0^{-1}\diam V) \cap B(c_{V'}, C_0^{-1}\diam V') = \varnothing.
    \end{equation}
\end{definition}

\begin{lemma}\label{l:whit-card}
    Let $\Omega_0,\scV$ be as in Definition \ref{def:whit-family}. Let $A \geq 1$, $U\subseteq \R^{d+1}$ and set
    $$\scV_{A,U} = \{V\in\scV : V\simeq_A U\}.$$
    We have
    \begin{equation}\label{e:whit-card}
        \#(\scV_{A,U}) \lesssim_{A, C_0, d} 1.
    \end{equation}
    If $\scU$ is a collection of subsets such that for any $V\in\scV$, there exists $U\in\scU$ such that $V\simeq_A U$, then
    \begin{equation}\label{e:whit-sum}
        \sum_{V\in\scV}(\diam V)^d \lesssim_{A,C_0,d} \sum_{U\in\scU}(\diam U)^d.
    \end{equation}
\end{lemma}
\begin{proof}
    For any $V\in\scV_{A,U}$, we have
    \begin{equation*}
        \dist(U,V) \leq A \diam U,
    \end{equation*}
    and
    \begin{equation*}
        A^{-1}\diam U \leq \diam V \leq A\diam U.
    \end{equation*}
    Let $B_V = B(c_V,C_0^{-1}\diam V)$ and fix $u\in U$. It follows that $B_V\subseteq V\subseteq B(u,3A\diam U)$ and $C_0^{-1}\diam V \geq (C_0A)^{-1}\diam U$ so that $\{B_V\}_{V\in\scV_{A,U}}$ is a collection of disjoint balls with radius $r(B_V) \geq (C_0A)^{-1}\diam U$ contained in the ball $B(u,3A\diam U)$ and hence has cardinality bounded in terms of $C_0,A,$ and $d$. This proves \eqref{e:whit-card}.

    To prove \eqref{e:whit-sum}, notice that
    \begin{equation*}
        \sum_{V\in\scV}(\diam V)^d \leq \sum_{U\in\scU}\sum_{V\in\scV_{A,U}} (\diam V)^d \lesssim_A \sum_{U\in\scU}\#(\scV_{A,U})(\diam U)^d \lesssim_{A,C_0,d} \sum_{U\in\scU}(\diam U)^d.
    \qedhere\end{equation*}
\end{proof}

 We define
\begin{equation}
    \mathcal{G}_0 = \{g(W) : W\in\scR_w\}
\end{equation}
and observe that $\mathcal{G}_0$ is a $\Lambda_0(L',d)$-Whitney family by equations \eqref{e:im-dist-omega} - \eqref{e:images-have-disjoint-centers}:
\begin{lemma}\label{l:boxes-whitney}
    There exists a constant $\Lambda_0(L',d) > 0$ such that $\mathcal{G}_0$ is a $\Lambda_0(L',d)$-Whitney family.
\end{lemma}

Combining this fact with Lemma \ref{l:whit-card} will allow us to bound the surface measure of images of stopped boxes in terms of the side-length of $A_0$-close bad and stopped cubes in $\scD$. The following two lemmas will give a Carleson packing condition on this bad subset $\scB_e\subseteq\scD$ defined in \eqref{e:bad-scD} below from which we will be able to conclude the desired surface measure bound \eqref{e:RF-surface-msr}. We begin with the following lemma due to David and Semmes.
\begin{lemma}[cf. \cite{DS93} Part I Lemma 3.27, (3.28)]\label{l:DS-CP}
    Let $A \geq 1$, let $\scD$ be a Christ-David lattice with coronization $(\scG,\scB,\scF)$. Then, 
    \begin{enumerate}
        \item[(a)] The set
        $$\mathscr{A} = \{Q\in\scG : \exists Q'\in S' \not= S\ni Q \text{ such that } Q\simeq_A Q'\}$$
        satisfies a Carleson packing condition.
        \item[(b)] Suppose $\scH\subseteq\scD$ satisfies a Carleson packing condition. The set
        $$\scH_A = \{Q\in\scD: \exists Q'\in\scH \text{ such that } Q\simeq_{A_0}Q'\}$$
        satisfies a Carleson packing condition.
    \end{enumerate}
\end{lemma}
This lemma will directly give us a Carleson packing condition on the set
\begin{equation}\label{e:bad-scD}
    \scB_e = \scB \cup \{Q\in\scG : \exists Q'\in S'\not=S\ni Q \text{ such that } Q\simeq_{2A_0^2}Q'\}.
\end{equation}
\begin{lemma}[$\scB_e$ Carleson packing condition]\label{l:scB-packing}
    The family $\scB_e$ satisfies a Carleson packing condition. For any $W\in\scB_w$, there exists $Q_W\in\scB_e$ such that $g(W)\simeq_{A_0}Q_W$.
\end{lemma}
\begin{proof}
    The fact that $\scB_e$ satisfies a Carleson packing condition follows from Lemma \ref{l:DS-CP}. For the second statement, let $W\in\scB_w$. By definition, $W\not\in\scG_w$ so that one of the following two holds:
    \begin{enumerate}[label=(\roman*)]
        \item $\exists Q\in\scB$ such that $g(W)\simeq_{A_0} Q$,
        \item $\exists S_1, S_2\in\scF$ such that $Q_1\in S_1\not= S_2\ni Q_2$ with $g(W)\simeq_{A_0} Q_1$ and $g(W)\simeq_{A_0} Q_2$.
    \end{enumerate}
    The first case gives the desired cube $Q_W$ immediately. In the second case, a calculation using the definition of $A_0$-closeness shows that $Q_1\simeq_{2A^2}Q_2$ so that $Q_1,Q_2\in\scB_e$ and we can set $Q_W = Q_1$.
\end{proof}

We now fix $y\in \partial\Omega\cap B(0,1)$ and $0 < r \leq 1$. In order to pick out the pieces of the domains which actually intersect $B(y,r)$, for any $T\in\scT$ we define
\begin{equation*}
    \scT_{y,r}' = \{T\in\scT : \Omega_T\cap B(y,r)\not=\varnothing\}.
\end{equation*}
We break up $\scT_{y,r}'$ into regions with large and small top cubes:
\begin{align*}
    \scT_{L,r} &= \{T\in\scT_{y,r'} : h(W(T)) > 10r\},\\
    \scT_{y,r} &= \scT_{y,r}'\setminus \scT_L.
\end{align*}
It is also convenient to collect all of the boundaries associated with a given stopping time domain $T\in\scT$ into one set:
\begin{equation*}
    \mathcal{B}_T = \bigcup_{j\in J_T}\partial\Omega_T^j.
\end{equation*}
We note that $\sB_T$ is Ahlfors upper $d$-regular by Proposition \ref{p:DT-lip-graph}. Proposition \ref{p:RF-surface-msr} will follow from the following three lemmas below. The first gives a bound for the domains in $\scT_{L,r}$ while the second gives a bound for those in $\scT_{y,r}$. 

\begin{lemma}\label{l:large-top-cubes}
\begin{equation*}
    \sum_{T\in\scT_{L,r}}\sH^d(\sB_T\cap B(y,r)) \lesssim_{L',d} r^d \leq \sH^d(\partial\Omega\cap B(y,r)).
\end{equation*}
\end{lemma}
\begin{proof}
We will show that $\#(\scT_{L,r})$ is bounded independent of $y$ and $r$. For any $T\in\scT_{L,r}$ we claim that there exists some $W_T\in T$ such that $h(W_T)\simeq r$ and $\dist(g(W_T), y) \simeq r$. Indeed, by definition there exists $R_T\in T$ such that $g(R_T)\cap B(y,r)\not=\varnothing$. There then exists a box $W_T\in T$ with $W_T\geq R_T$ with the desired properties because of \eqref{e:g-dist-E} and the inequality $h(W(T)) > 10r$. But, since the collection $\{g(W_T)\}_{T\in\scT_{L,r}}$ is a Whitney family, it follows that $N = \#(\scT_{L,r}) = \#(\{g(W_T)\}_{T\in\scT_{L,r}}) \lesssim_{L',d} 1$. Therefore, since $\sB_T$ is Ahlfors upper $d$-regular,
\begin{equation*}
     \sum_{T\in\scT_{L,r}}\sH^d(\sB_T\cap B(y,r)) \lesssim_{d} \#(\scT_{L,r})r^d \lesssim_{L',d} r^d.\qedhere
\end{equation*}
\end{proof}

We now handle the regions with small top boxes:
\begin{lemma}\label{l:small-top-cubes}
\begin{equation}\label{e:small-top-cubes}
    \sum_{T\in\scT_{y,r}}\sH^d(\mathcal{B}_T\cap B(y,r)) \lesssim_{L',d,\epsilon'} \sH^d(\partial\Omega\cap B(y,A_0^2r)) \lesssim_{L',d} r^d.
\end{equation}
\end{lemma}
\begin{proof}
We first note that since $\sH^d(\sB_T) \lesssim_{L',d} \sH^d(\partial\Omega_T)$ by \eqref{e:DSj-bound} and the fact that $g$ is $L'$-bi-Lipschitz, we have
\begin{equation*}
    \sum_{T\in\scT_{y,r}}\sH^d(\mathcal{B}_T\cap B(y,r)) \leq \sum_{T\in\scT_{y,r}} \sH^d(\sB_T) \lesssim_{L',d} \sum_{T\in\scT_{y,r}} \sH^d(\partial\Omega_T).
\end{equation*}
Therefore, it suffices to prove $\sum_{T\in\scT_{y,r}}\sH^d(\partial\Omega_T)\lesssim_{L',d,\epsilon} \sH^d(\partial\Omega\cap B(y,A_0^2r))$.

For any $T\in \scT_{y,r}$, \eqref{e:side-length-stretch} gives $\diam g(\Bot(W))) \lesssim_{d} h(W)$ so that Lemma \ref{l:Dg-change} and the fact that $g$ is $L'$-bi-Lipschitz give an analog of \eqref{e:DSj-bound}:
\begin{align}\nonumber 
    \sH^d(\partial\Omega_T) &\lesssim_{d,L'} \sH^d(\partial\Omega_T\cap\partial\Omega) + \sum_{W\in m(T)} \sH^d(g(\Bot(W)))\\ \label{e:partial-omega-bound}
    &\lesssim_d \sH^d(\partial\Omega_T\cap\partial\Omega) + \sum_{W\in m(T)} h(W)^d.
\end{align}
Now, $W\in m(T)$ implies that there exists a child $W'\in\Stop(T)\cap\scB_w$ for which we have $Q\in \scB_e$ with $g(W')\simeq_{A_0} Q$ by Lemma \ref{l:scB-packing}. For any $x\in Q$, we compute
\begin{align}\nonumber
    |x-y| &\leq \diam Q + \dist(Q,g(W')) + \diam g(W') + \dist(y,g(W')) \\\nonumber
    &\leq 2A_0\diam g(W') + 2A_0\diam g(W') + \diam g(W') + 10r \\\label{e:final-inflated-ball}
    &\leq 10\sqrt{d}A_0h(W') + 10r \leq 100\sqrt{d}A_0r \leq A_0^2r.
\end{align}
This shows that $Q\subseteq B(y,A_0^2r)$. Hence, applying Lemma \ref{l:whit-card} with $\scV = \{g(W): W\in m(T)\}$ and $\scU = \{Q\in\scB_e: Q\subseteq B(y,A_0^2r)\}$, we get
\begin{align}
    \sum_{W\in m(T)} h(W)^d \lesssim_{A_0,d,L'} \sum_{\substack{Q\in\scB_e \\ Q\subseteq B(y,A_0^2r)}}\ell(Q)^d \lesssim_{d,\epsilon'}\sH^d(\partial\Omega\cap B(y,A_0^2r))
\end{align}
where the last inequality follows from the Carleson packing condition for $\scB_e$. By observing that $\partial\Omega_T\cap\partial\Omega\subseteq B(y,50\sqrt{d}r)$ for any $T\in\scT_{y,r}$ and $\sH^d(\partial\Omega_T\cap\partial\Omega_{T'}\cap\partial\Omega) = 0$ for $T\not=T'$, \eqref{e:partial-omega-bound} implies 
\begin{equation*}
    \sum_{T\in\scT_{y,r}}\sH^d(\partial\Omega_T) \lesssim_{A_0,L',d,\epsilon'} \sH^d(\partial\Omega \cap B(y,A_0^2r)) \lesssim_{d,L'} r^d
\end{equation*}
using the fact that $g$ is bi-Lipschitz and parameterizes $\partial\Omega$ in the last inequality.
\end{proof}

Finally, we handle the boundaries of ``trivial'' cube domains associated to the bad boxes in $\scB_w$. To do so, we collect the boundaries associated to fixed $W\in\scB_w$ into the set
\begin{equation*}
    \sB_W = \bigcup_{R\in\scL_W}\partial g(R).
\end{equation*}
\begin{lemma}\label{l:bad-cubes-measure}
We have
    \begin{equation*}
        \sum_{W\in\scB_w}\sH^d(\sB_W\cap B(y,r)) \lesssim_{L',d,\epsilon} \sH^d(\partial\Omega\cap B(y,A_0^2r)) \lesssim_{A_0,d,L'} r^d.
    \end{equation*}
\end{lemma}
\begin{proof}
    We first note that
    \begin{equation*}
        \sum_{W\in\scB_w}\sH^d(\sB_W\cap B(y,r)) \leq \sum_{\substack{W\in\scB_w \\ g(W)\cap B(y,r)\not=\varnothing}}\sH^d(\sB_W) \lesssim_{L'} \sH^d(\partial g(W)) \lesssim h(W)^d
    \end{equation*}
     using $\sH^d(g(\Bot(W))) \lesssim_d h(W)^d$ as in \eqref{e:partial-omega-bound} above. In addition, there exists some cube $Q\in\scB_e$ such that $g(W)\simeq_{A_0}Q$ and, as in \eqref{e:final-inflated-ball}, $Q\subseteq B(y,A_0^2r)$. Hence, we have
    \begin{align*}
        \sum_{W\in\scB_w}\sH^d(\sB_W\cap B(y,r)) &\lesssim_d \sum_{\substack{W\in\scB_w \\ g(W)\cap B(y,r)\not=\varnothing}}h(W)^d \lesssim_{A_0,d,L'} \sum_{\substack{Q\in\scB_e \\ Q\subseteq B(y,A_0^2r)}}\ell(Q)^d \\
        &\lesssim_{d,\epsilon'} \sH^d(\partial\Omega\cap B(y,A_0^2r)) \lesssim_{d,L'} r^d.\qedhere
    \end{align*}
\end{proof}

\begin{proof}[Proof of Proposition \ref{p:RF-surface-msr}]
    First, consider $\Omega_j\in\scL$ such that either there exists $j_0, T_0$ such that $\Omega_j = \Omega_{T_0}^{j_0}$ or there exists $W\in\scB_w$ and $R\in\scL_W$ such that $\Omega_j = g(R)$. Therefore, we have
    \begin{align*}
        \sum_{j\in J_{\scL}}&\sH^d(\partial\Omega_j\cap B(y,r))\\
        &\leq\sum_{T\in\scT_{L,r}}\sum_{j\in J_T}\sH^d(\partial\Omega^j_T\cap B(y,r)) + \sum_{T\in\scT_{y,r}}\sum_{j\in J_T}\sH^d(\partial\Omega_T^j\cap B(y,r)) + \sum_{W\in\scB_w}\sum_{R\in\scL_W}\sH^d(\partial R\cap B(y,r))\\
        &\lesssim \sum_{T\in\scT_{L,r}}\sH^d(\sB_T\cap B(y,r)) + \sum_{T\in\scT_{y,r}}\sH^d(\sB_T\cap B(y,r)) + \sum_{W\in\scB_w}\sH^d(\sB_W\cap B(y,r))\\
        &\lesssim_{L',d,\epsilon'} r^d
    \end{align*}
    by Lemmas \ref{l:large-top-cubes}, \ref{l:small-top-cubes}, and \ref{l:bad-cubes-measure}.
\end{proof}

This completes the proof of Theorem \ref{t:thmA}.

\section{The proofs of Theorems \ref{t:thmB} and \ref{t:thmC}}\label{sec:thmB-thmC}
We now turn to proving Theorems \ref{t:thmB} and \ref{t:thmC}. Both of these theorems will follow from the following result.
\begin{theorem}\label{t:gen-stop}
    Let $\Omega\subseteq\R^{d+1}$ be a domain. There exist constants $A(d), L(d), M_0(d) \geq 1$ and $\epsilon_0(d),\delta_0(d) > 0$ such that if $\partial\Omega$ admits a $d$-dimensional $(M,\epsilon,\delta)$-graph coronization with $M \geq M_0$, $\epsilon \leq \epsilon_0$, and $\delta \leq \delta_0$, then there exists a collection $\mathscr{L} = \{\Omega_j\}_{j\in J_\scL}$ of $L$-Lipschitz graph domains such that 
    \begin{enumerate}[label=(\roman*)]
        \item $\Omega_j\subseteq\Omega$, \label{i:om-subset}
        \item $\Omega\cap B(0,1) \subseteq \bigcup_{j=1}^\infty \Omega_j$,\label{i:om-cover}
        \item $\exists C(d) > 0$ such that $\forall x\in\R^{d+1}$, $x\in \Omega_j$ for at most $C$ values of $j$,\label{i:om-bounded-overlap}
        \item For any $y\in \partial\Omega\cap B(0,1)$ and $0 < r \leq 1$, we have 
            \begin{equation*}
                \sum_{j=1}^\infty \sH^d(\partial\Omega_j\cap B(y,r)) \lesssim_{\epsilon, d} \sH^d(\partial\Omega\cap B(y,Ar)).
            \end{equation*}\label{i:om-surface-measure}
    \end{enumerate}
\end{theorem}

The proof will be via relatively minor modifications of the argument for Theorem \ref{t:thmA}. The idea is to construct a collection of CCBPs with associated maps $\{g_i\}_{i\in I}$ where $g_i:\mathcal{D}_i\rightarrow\overline{\Omega}$ which individually parameterize only a little piece of $\overline{\Omega}$ at a time. These maps will be $(1+C\delta)$-bi-Lipschitz at the cost of introducing an outer ``buffer zone'' of domains in the images of these mappings. The ``core'' domains will have disjoint interiors while the domains in the buffer zone will have finite overlap.

We now fix constants $\rho,A_0,K$ as in Section \ref{sec:thmA} and set $A_1 = \max\left\{20A_0^2,\ \frac{2000\sqrt{d}A_0}{c_0\rho}\right\}$ and $M_0 = \max\{\frac{10K}{\rho^2},\ A_1^2\}.$ We fix a lattice $\scD$ for $\partial\Omega$, let $M = M_0$, and let $(\scG,\scB,\scF)$ be an $(M,\epsilon,\delta)$-graph coronization assumed for $\partial\Omega$ in the hypotheses of Theorem \ref{t:gen-stop}. The constants $\epsilon_0$ and $\delta_0$ will be taken small depending on inequalities arising in the proof.  

\subsection{Local CCBPs adapted to $\scD$}
We will construct Reifenberg parameterizations as in subsection \ref{subsec: CCBP} centered around the points of a Whitney-like net $\sC_0$ of $\Omega\cap B(0,1)$ rather than having a single global map. For every $n \geq 0$, define
\begin{align*}
    s_n &= 3\cdot 2^{-n+1},\\
    D_n &= \{z\in B(0,1):\dist(z,\partial\Omega) = s_n\},\\
    C_n &= \Net(D_n, s_n) = \{p_{i,n}\}_{i\in I_n}.
\end{align*}
Set $\sC_0 = \bigcup_{n}C_n$.
\begin{definition}[flat and non-flat points]
    Let $p\in\Omega\cap B(0,1)$. Define
    \begin{equation*}
        \scQ_p = \left\{Q\in\scD : Q\simeq_{10\sqrt{d}A_1} B\left(p,\frac{1}{2}\dist(p,\partial\Omega)\right)\right\}.
    \end{equation*}
    We say that $p$ is \textit{flat} if there exists $S\in\scF$ such that $\scQ_p\subseteq S$. Otherwise, we say that $p$ is \textit{non-flat}. Given the set $\sC_0$ above, we define the flat and non-flat points of $\sC_0$ by
    \begin{align*}
        \sF_0 &= \{p\in \sC_0 : \exists S\in\scF,\ \scQ_p\subseteq S\},\\
        \sN_0 &= \sC_0 \setminus \sF_0.
    \end{align*}
\end{definition}

Fix $p\in\sF_0$ and let $S_p\in\scF$ be such that $\scQ_p\subseteq S_p$. Without loss of generality, assume that $\dist(p, 0) = \dist(p,\partial\Omega) = 6 = s_0$. The fact that $p\in\sF_0$ implies there exists $Q_p\in\scD_{s(0)}$ with $\dist(p,Q_p) \leq 6$ and $c_0\rho \leq \diam(Q_p) \leq 1 = r_0$ so that $Q_p\in\scQ_p$ because $A_1 \geq 10(c_0\rho)^{-1}$. Hence, $b\beta_{\partial\Omega}(MB_{Q_p}) \leq \epsilon$. Without loss of generality, suppose $P_{Q_p} = \R^d$ achieves the infimum in the definition of $b\beta_{\partial\Omega}(MB_{Q_p})$. 

For any $k \geq 0$, let
\begin{align}\label{e:Y_k-local-CCBP}
    Y_k^p &= \{x_Q : Q\in S_p\cap\scD_{s(k)}\},\\
    X_k^p &\in \Net(Y_k^p, r_k).
\end{align}
We enumerate $X_k^p = \{x_{j,k}\}_{j\in J_k}$ and define
\begin{align*}
    B_{j,k} &= B(x_{j,k},r_k),\\
    P_{j,k} &= P_{Q_{j,k}},\\
    \scZ_p &= (P_{Q_p}, \{B_{j,k}\}, \{P_{j,k}\})
\end{align*}
where $P_{Q_{j,k}}\ni x_{Q_{j,k}}$ satisfy $\beta^{d,1}_{\partial\Omega}(2\rho^{-1}KB_{Q_{j,k}}, P_{Q_{j,k}}) \lesssim \beta^{d,1}_{\partial\Omega}(2\rho^{-1}KB_{Q_{j,k}}B_{Q_{j,k}})$ as in the hypotheses of Lemma \ref{l:ep-by-beta}. Using the fact that $Q\in S_p\subseteq\scG$ so that $b\beta(MB_Q) \leq \epsilon$, a nearly identical argument to that of Lemma \ref{l:CubeCCBP} gives that $\scZ_p$ is a CCBP:
\begin{lemma}\label{l:Cube-local-CCBP}
    For any $p\in\sF_0$, $\scZ_p$ is a CCBP.
\end{lemma}

Lemma \ref{l:Cube-local-CCBP} allows us to apply Theorem \ref{DT-thm} to get a Reifenberg parameterization $g_p:\R^{d+1}\rightarrow\R^{d+1}$ produced by the CCBP $\scZ_p$. We let $\scW'_p$ be the Whitney decomposition of $\bH^{d+1}$ such that $W_0=[-2,2]^d\times[4,8]\in\scW'_p$ so that $p = c(W_0) = (0,6)$ and give a localized Whitney lattice $\scW_0 = \{W\in\scW'_p : W\in D(W_0)$ as in \eqref{e:centered-Whitney}. We now give a one-step version of the stopping time construction in Definition \ref{def:stopping-time-regions} to produce a single domain $\sD_p$ and an extended version $\widehat{\sD}_p$ which contains additional ``buffer'' cubes which $g_p$ maps forward to approximating Lipschitz graph domains

\begin{definition}[Stopping time regions around flat $p$]\label{def:flat-stopping-time-region}
    Fix a constant $A_1 > 1$ and $p\in\sF$ and form the map $g_p$ and Whitney lattices $\scW'_p$ and $\scW_0$ as above. As in \eqref{e:G0}, we define
    \begin{align*}
        \scG_p &= \left\{W\in \scW_0 : \forall Q\in\scD \text{ such that }Q\simeq_{A_1} g_p(W) \text{ we have }Q\in S_p\right\}
    \end{align*}
    By definition, $p\in\sF_0$ implies $W_0\in \scG_p$. Define a single stopping time region $T_p\subseteq\scW_0$ by setting $T_p$ to be the maximal subtree of $D(W_0)\cap\scG_p$ such that for any $R\in T_p$, either all of its children are in $T_p$ or none are.
\end{definition}

\begin{definition}[Stopping time domains around flat $p$]
    For any $p\in\sF_0$, we define a \textit{stopping time domain}
    \begin{equation*}
        \mathcal{D}_p = \bigcup_{W\in T_p}W.
    \end{equation*}
    Additionally, we extend $\sD_p$ by a ``buffer'' region of $A_0$-close cubes on the boundary of $\sD_p$ by defining the \textit{extended stopping time region} and \textit{extended stopping time domain} by
    \begin{align*}
        \widehat{T}_p &= \{W\in \scW'_p : \exists R\in T_p,\ W\simeq_{A_0} R\},\\
        \widehat{\sD}_p &= \bigcup_{W\in\widehat{T}_p} W.
    \end{align*}
    We will carve up the image domains
    \begin{align*}
        \Omega_p &= g_p(\sD_p),\\
        \widehat{\Omega}_p &= g_p(\widehat{\sD}_p)
    \end{align*}
    to construct one family of our desired Lipschitz graph domains in the conclusion of Theorem \ref{t:gen-stop}.
\end{definition}

We will now prove an analog of Lemma \ref{l:g-prop}.
\begin{lemma}[properties of $g_p$]\label{l:gp-prop}
    There exists a choice of constant $A_1 \lesssim_d A_0$ such that for any $z = (x,y) \in \widehat{\sD}_p$, the following hold:
    \begin{enumerate}[label=(\roman*)]
    \item $f_{n(y)}(x)\in V_{n(y)}^8,$ \label{i:p-nearby-net}
    \item $(1-C\epsilon)|y| \leq \dist(g_p(z),\partial\Omega) \leq (1+C\epsilon)|y|,$ \label{i:gp-dist}
    \item For any $m\in\N$ with $m < n(y)$, there exists a collection of cubes $Q_{n(y)}\subseteq Q_{n(y)-1} \subseteq\cdots\subseteq Q_m$ such that for any $k$ with $m \leq k \leq n,\ Q_k\in S_p$ and $\dist(g(x,r_k), Q_{k}) \lesssim r_k$,
    \begin{equation*}
        \sum_{k=m}^n\epsilon'(f_k(x))^2 \lesssim_{M,\rho,d} \sum_{k=m}^n\beta^{d,1}_{\partial\Omega}(MB_{Q_k})^2 \lesssim \epsilon.
    \end{equation*}\label{i:p-epsilon-bound}
\end{enumerate}  
\end{lemma}
\begin{proof}
    The proof is similar to that of Lemma \ref{l:g-prop} with the only complication being that we need the map $g_p$ to also behave nicely on the buffer region of $A_0$ close cubes to those in $\scW_0$. We will prove this for fixed $z = (x,y)$ by first assuming that \ref{i:p-nearby-net} holds and showing that items \ref{i:gp-dist} and \ref{i:p-epsilon-bound} hold. We will then prove item \ref{i:p-nearby-net} by induction, considering the points $(x,r_k)\in\widehat{\sD}_p$ for $0\leq k < n(y)$ (assume without loss of generality that $h(W(T)) = 4$). 

    So, first assume that item \ref{i:p-nearby-net} holds. Given this, item \ref{i:gp-dist} follows exactly as in Lemma \ref{l:g-prop} item \ref{i:g-dist}. Similarly, item \ref{i:p-epsilon-bound} follows as in Lemma \ref{l:g-prop} item \ref{i:epsilon-bound} by replacing the infinite chain of cubes with a chain terminating in $Q_{n(y)}\in\scD_{s(n(y))}\cap S_p$.
    
    We now prove item \ref{i:p-nearby-net} by the induction discussed above. For the base case, recall that $f_0(x) = x$ so that $(x,y)\in\widehat{\sD}_p$ means $\dist(x,W(T_p)) \leq 2A_0\diam W(T)$. Since we've chosen $M$ large enough, $x\in MB_{Q_p}\cap P_{Q_p}$ so that $\dist(x,\partial\Omega) \lesssim_M\epsilon$. This means $p\in\sF$ implies that there exists $Q_0\in\scD_{s(0)}\cap S_p$ such that $|x-x_{Q_0}| \leq 2r_0$ from which the claim follows.
    We will finish the proof by proving the following claim:
    \vspace{1ex}
    \begin{claim}
        For any $k < n(y),\ f_k(x)\in V_k^8$ implies that $f_{k+1}(x)\in V_{k+1}^8$.
    \end{claim}
    \vspace{1ex}
    \begin{claimproof}
        The fact that $(x,r_k)\in\widehat{\sD}_p$ means that $(x,r_k)\in R_k\in\scW'_p$ and there exists $W\in T_p$ such that $R_k\simeq_{A_0} W$. This gives $\dist(R_k,W) \lesssim_{A_0}h(W)$ and $r_k \simeq h(R_k) \simeq_d \diam R_k \simeq_{A_0} h(W)$. If now $f_k(x)\in V_k^8$, then there exists $Q\in S_p\cap \scD_{s(k)}$ such that $|f_k(x) - x_Q| \leq 8r_k$, so that $b\beta_{\partial\Omega}(MB_Q)\leq\epsilon$ implies there is $Q_{k+1}\in\scD_{s(k+1)}$ with $|f_{k+1}(x) - x_{Q_{k+1}}| \leq 2r_{k+1}$. Applying item \ref{i:gp-dist} and \eqref{e:im-diam-height} gives
        \begin{align*}
            \dist(Q_{k+1}, g(W)) &\leq \dist(Q_{k+1}, g(R_k)) + \diam g(R_k) + \dist(g(R_k),g(W))\\
            &\leq 2\sqrt{d}h(R_k) + 2\sqrt{d}h(R_k) + A_0(\diam g(R_k) + \diam g(W))\\
            &\leq 5\sqrt{d}A_0 h(R_k) + 5\sqrt{d}A_0(h(R_k) + \diam g(W) \leq A_0^2\diam g(W)
        \end{align*}
        and
        $$\diam Q_{k+1} \leq r_{k+1} \leq h(R) \leq A_0h(W) \leq 2A_0 \diam g(W).$$
        $$\diam Q_{k+1} \geq c_0\ell(Q) \geq \frac{c_0\rho}{10}r_{k+1} \geq \frac{c_0\rho}{200}h(R_k) \geq \frac{c_0\rho}{200A_0}h(W) \geq \frac{c_0\rho}{1000\sqrt{d}A_0}\diam g(W).$$
        Therefore, $Q_{k+1} \simeq_{A_1} g(W)$. By the definition of $T_p$, we then have $Q_{k+1}\in S_p$ so that $x_{Q_{k+1}}\in Y_{k+1}^p$ and $f_{k+1}(x)\in V_{k+1}^8$ as in Lemma \ref{l:g-prop} \ref{i:nearby-net}. 
    \end{claimproof}\ 
\end{proof}
We will also need to construct Lipschitz graph domains around non-flat $q\in\sN_0$. Because $\partial\Omega$ admits a graph coronization, there are a controlled number of such $q$ so that we can cover the regions around them by ``trivial'' domains without adding too much total boundary.

\begin{definition}[Trivial domains around non-flat $q$]
    Fix once and for all an auxiliary Whitney decomposition $\widetilde{\scW}$ of $\Omega$. For any $q\in\sN_0$, there exists a Whitney cube $W_q\in\widetilde{\scW}$ such that $q\in W_q$ and
    \begin{equation*}
        \diam W_q \leq \dist(q,\partial\Omega) \leq 8\diam W_q.
    \end{equation*}
    We directly define
    \begin{align*}
        \sD_q &= \Omega_q = W_q,\\
        \widehat{\sD}_q &= \{W\in\widetilde{\scW} : W\simeq_{A_0} W_q\},\\
        \widehat{\Omega}_q &= \bigcup_{W\in\widehat{\sD}_q} W.
    \end{align*}    
\end{definition}

\subsection{The Lipschitz decomposition with bounded overlaps}
We will get our final collection of domains by choosing a well-spaced subset $\sF\subseteq\sF_0$ and $\sN\subseteq\sN_0$ and carving up the domains in $\{\widehat{\Omega}_p\}_{p\in\sF} \cup \{\widehat{\Omega}_q\}_{q\in\sN}$.

To choose our collections $\sF$ and $\sN$, we put an ordering on the points of $\sC_0$ by choosing some ordering on each finite set $C_n$ and then imposing $p_n <  p_m$ for any $p_n\in C_n,\ p_m\in C_m$ with $n < m$. The set $\sC_0$ has a least element that we call $c_0$, and we define an auxiliary collection $\sP_0 = \{p_0\}$ where $p_0 = c_0$. Given the definitions of $\sC_0$ and $\sP_0$, we define $\sC_{n+1}$ and $\sP_{n+1}$ inductively for any $n \geq 0$ by
\begin{align}\label{e:C-inductive}
    \sC_{n+1} &= \sC_n\setminus \left\{p\in\sC_n\ :\ \dist\left(p,\bigcup_{p'\in\sP_n}\Omega_{p'}\right) < \frac{A_0}{30}\dist(p,\partial\Omega).\right\},\\
    \sP_{n+1} &= \sP_n \cup \{p_{n+1}\}\label{e:P-inductive}
\end{align}
where $p_{n+1}$ is the least element of $\sC_{n+1}$ with respect to the ordering inherited from $\sC_0$. Finally, put
\begin{align*}
    \sC &= \bigcup_{n=0}^\infty \sP_n,\\
    \sF &= \sF_0 \cap \sC,\\
    \sN &= \sN_0 \cap \sC.
\end{align*}
We can now give the definition of our desired Lipschitz decomposition with bounded overlaps

\begin{definition}[Lipschitz decomposition with bounded overlap]\label{def:gen-L}
    For any $p\in\sF$, Proposition \ref{p:DT-lip-graph} implies there exists an Ahlfors regular $d$-rectifiable set $\Sigma_{T_p}$ such that
    \begin{equation*}
        \sD_p \setminus \Sigma_{T_p} = \bigcup_{j\in J_p} \sD_p^j
    \end{equation*}
    where $\sD_p^j$ is an $L_0(d)$-Lipschitz graph domain. We set $\Omega_p^j = g_p(\sD_p^j)$ and define our \textit{Lipschitz decomposition with bounded overlap}
    \begin{equation}\label{e:gen-L}
        \scL_b = \{g_p(W)\}_{p\in\sF,\ W\in\widehat{\sD}_p\setminus\sD_p}\cup\{\Omega_p^j\}_{p\in\sF,\ j\in J_p} \cup \{R\}_{q\in\sN,\ R\in\widehat{\sD}_q}.
    \end{equation}
\end{definition}

In analogy to Propositions \ref{p:RF-Lipschitz} and \ref{p:RF-surface-msr}, we will finish the proof of Theorem \ref{t:gen-stop} if we can prove the following two propositions.

\begin{proposition}\label{p:gen-Lipschitz}
    Let $\Omega$ be as in Theorem \ref{t:gen-stop} and $\mathscr{L}_b = \{\Omega_j\}_{j\in J_{\scL_b}}$ be as in \eqref{e:gen-L}. There exists $L_1(d) > 0$ such that for any $j\in J_{\scL_b},\ \Omega_j$ is an $L_1$-Lipschitz graph domain. In addition, we have
    \begin{enumerate}[label=(\roman*)]
        \item $\Omega_j\subseteq\Omega$, \label{i:prop-om-subset}
        \item $\Omega\subseteq \bigcup_{j\in J_{\scL_b}}\overline{\Omega_j}$, \label{i:prop-om-cover}
        \item $\exists C(d) > 0$ such that $\forall x\in\Omega$, $x\in\Omega_j$ for at most $C$ values of $j$. \label{i:prop-bounded-overlap}
    \end{enumerate}
\end{proposition}

\begin{proposition}\label{p:gen-surface-msr}
     Let $\Omega$ be as in Theorem \ref{t:gen-stop} and $\mathscr{L} = \{\Omega_j\}_{j\in J_{\scL_b}}$ be as in \eqref{e:gen-L}. For any $y\in\partial\Omega\cap B(0,1)$ and $0 < r < 1$, we have
    \begin{equation}
        \sum_{j\in J_{\scL_b}}\sH^d(\partial\Omega_j\cap B(y,r)) \lesssim_{\epsilon,d}\sH^d(\partial\Omega\cap B(y,A_1r)).
    \end{equation}
\end{proposition}

\subsection{Lipschitz bounds and covering / overlap properties for Theorems \ref{t:thmB} and \ref{t:thmC}}

In order to prove Propositions \ref{p:gen-Lipschitz} and \ref{p:gen-surface-msr}, we must show that the mapping $g_p$ behaves on $\widehat{\sD}_p$ as our single Reifenberg parameterization $g$ did on each $\sD_T$ in the setting of Theorem \ref{t:thmA}.

The following analogue of Lemma \ref{l:Dg-change} allows us to control the change in $Dg_p$ on any extended stopping time domain $\widehat{\sD}_p$.

\begin{lemma}[Variation of $Dg_p$]\label{l:Dgp-change}
    For any $p\in\sF$ and $z,w\in \widehat{\sD}_p$, we have
    \begin{equation}\label{e:Dgp-full-change}
        |Dg_p(z)\circ Dg_p(w)^{-1} - I| \leq C\delta.
    \end{equation}
    In particular,
    \begin{equation}\label{e:Dgp-change}
        |Dg_p(z) - I| \leq C\delta,
    \end{equation}
    and $g_p|_{\widehat{\sD}_p}$ is $(1+C\delta)$-bi-Lipschitz.
\end{lemma}
\begin{proof}
    Equation \eqref{e:Dgp-full-change} has a very similar proof to that of \eqref{e:Dg-change} in Lemma \ref{l:Dg-change}. The only significant difference is that we must use Lemma \ref{l:gp-prop} in place of Lemma \ref{l:g-prop}. Equation \eqref{e:Dgp-change} follows from the added observation that $p\in\sD_p$ and $\dist(p,\partial\Omega) \geq 2$ (after normalizing) implies $Dg_p(p) = I$ so that the claim follows from \eqref{e:Dgp-full-change} by taking $w = p$.
\end{proof}
We now have enough to show that each domain in $\scL$ as in \eqref{e:gen-L} is Lipschitz graphical.
\begin{lemma}\label{l:gen-Lipschitz}
    There exists a constant $L_1(d) > 0$ such that $\Omega_j$ is an $L_1$-Lipschitz graph domain for all $j\in J_\scL$.
\end{lemma}
\begin{proof}
    Each domain in the set $\{R\}_{q\in\sN,\ R\in\widehat{\sD}_q}$ is a cube, which is an $L_0(d)$-Lipschitz graph domain trivially. Each domain $\Omega_j$ in the set $\{g_p(W)\}_{p\in\sF,\ W\in\widehat{\sD}_p\setminus\sD_p}\cup\{\Omega_p^j\}_{p\in\sF,\ j\in J_p}$ is the image under $g_p$ of an $L_0$-Lipschitz graph domain. Therefore, by Lemma \ref{l:Dgp-change} and Proposition \ref{p:slow-vary-lip-graph} there exists $L_1(d) > L_0$ such that each such $\Omega_j$ is an $L_1$-Lipschitz graph domain.
\end{proof}

In order to prove the remaining statements of Proposition \ref{p:gen-Lipschitz}, we first show that the buffer region $\widehat{\Omega}_p\setminus\Omega_p$ contains a cone around $\Omega_T$ with respect to the distance to $\partial\Omega$ for any $p\in\sC$:
\begin{lemma}\label{l:omega-cone}
For any $p\in\sC$, $\widehat{\Omega}_p$ contains an $\frac{A_0}{10}$-cone around $\Omega_p$ with respect to distance from $\partial\Omega$. That is, 
\begin{equation}\label{e:omega-cone}
    F = \left\{w\in\Omega : \dist(w,\Omega_p) <  \frac{A_0}{10}\min\left\{\dist(w,\partial\Omega),\dist(g_p(W(T_p)),\partial\Omega)\right\} \right\} \subseteq \widehat{\Omega}_p
\end{equation}
\end{lemma}
\begin{proof}
First, suppose that $p\in\sF$ and let $z\in F$. Since $\widehat{\Omega}_p = g_p(\widehat{\sD}_p)$ where $g_p$ is $(1+C\delta)$-bi-Lipschitz by Lemma \ref{l:Dgp-change} and translates distance in the domain to $\R^d$ to distance to $\partial\Omega$ in the image by Lemma \ref{l:gp-prop} \ref{i:gp-dist}, it suffices to show 
\begin{equation}\label{e:Dc-cone}
    \left\{z\in \Omega : \dist(z,\sD_p) <  \frac{A_0}{4}\min\left\{\dist(z,\R^d),\dist(W(T_p),\R^d)\right\} \right\} \subseteq \widehat{\sD}_p
\end{equation}
because the desired containment then follows by mapping \eqref{e:Dc-cone} forward. Now, there exists $W\in T_p$ such that $\dist(z,W) = \dist(z,\sD_p)$ and there exists a cube $W_z\in\mathscr{W}'_p$ such that $z\in W_z$. By the definition of $\widehat{\sD}_p$, it suffices to show that $W\simeq_{A_0} W_z$. We estimate
\begin{align}\label{e:cone-cubes-close}
    \dist(W,W_z) &\leq \dist(z,\sD_p) < \frac{A_0}{4} \min\{\dist(z,\R^d),\dist(W(T_p),\R^d)\} \\\nonumber
    &\leq \frac{A_0}{2} \min\{h(W_z),h(W(T_p))\} = \frac{A_0}{2}\min\{\ell(W_z),\ell(W(T_p))\}.
\end{align}
Using this we get
\begin{align*}
    \ell(W) &= h(W) \leq \dist(W,W_z) + \diam W_z + h(W_z) \leq \left(\frac{A_0}{2} + \sqrt{d+1} + 1\right)\ell(W_z) \leq A_0\ell(W_z)
\end{align*}
given that $A_0 \geq 4\sqrt{d}$. A similar calculation shows that $\ell(W_z) \leq A_0 \ell(W)$ which completes the proof in the case when $p\in\sF$. 
If $q\in\sN$, then $\Omega_q = W_q$. Let $w\in F$ and let $W_w\in\widetilde{\scW}$ with $w\in W_w$. By a similar computation to the above, one can show that $W_q\simeq_{A_0} W_w$ from which the result follows.
\end{proof}

With the help of Lemma \ref{l:omega-cone}, we can prove the bounded overlap and covering properties of $\scL$.

\begin{lemma}\label{l:omegas-are-nice}\ 
Let $p,p'\in\sC,\ p\not=p'$. The following hold:
\begin{enumerate}[label=(\roman*)]
    \item $\Omega_p \cap \Omega_{p'} =\varnothing$, \label{i:omegas1}
    \item $\Omega\cap B(0,1)\subseteq \bigcup_{p\in\sC}\overline{\widehat{\Omega}}_p$, \label{i:omegas2}
    \item $\exists C(d) > 0$ such that $\forall x\in\Omega,\ x\in\widehat{\Omega}_p$ for at most $C$ values of $j$,\label{i:omegas3}
    \item $\widehat{\Omega}_p \subseteq\Omega.$\label{i:omegas4}
\end{enumerate}
\end{lemma}
\begin{proof}
We begin with proving \ref{i:omegas1}. Using the partial order on $\sC$, assume without loss of generality that $p' < p$. By the definition of $\sC$ (Recall \eqref{e:C-inductive}.), we have $\dist(p,\Omega_{p'}) \geq \frac{A_0}{30}\dist(p,\partial\Omega)$. We claim that
\begin{equation*}
    \Omega_{p}\subseteq B(p,3\sqrt{d}\dist(p,\partial\Omega)) \subseteq B\left(p,\frac{A_0}{30}\dist(p,\partial\Omega)\right)
\end{equation*}
where the final inclusion follows because $A_0 \geq 120\sqrt{d}$. Indeed, if $p\in\sN$, then $\Omega_{p} = W_{p}\ni p$ with $\diam W_{p} \leq \dist(p,\partial\Omega)$. If instead $p\in\sF$, then $\Omega_{p} = g_{p}(\sD_{p})$ where $\sD_{p}$ is composed of a union of cubes in the descendants $D(W(T_{p}))$ where $\dist(p,\partial\Omega) \geq \ell(W(T_{p}))$ so that the fact that $g_p$ is $(1+C\delta)$-bi-Lipschitz means $\sqrt{d+1}\dist(p,\partial\Omega) \geq \diam(g_p(W))$ and $\dist(p,\partial\Omega) \geq \dist(g_p(W),p)$ for any $W\in T_p$. The claim follows.

We now prove \ref{i:omegas2}. Let $z\in \Omega\cap B(0,1)$ and let $k\geq0$ be such that $s_{k+1}\leq \dist(z,\partial\Omega)\leq s_k$. By the definition of $C_k$, there exists $p_k\in C_k$ such that
\begin{equation*}
    |z-p_k| \leq 3s_k = 6s_{k+1} \leq 6\dist(z,\partial\Omega).
\end{equation*}
Now, if $p_k\in\sC$, then by Lemma \ref{l:omega-cone}, $z\in\widehat{\Omega}_{p_k}$. Otherwise, $p_k\not\in\sC$ so that by \eqref{e:C-inductive} there exists $p\in\sC$ such that $p < p_k$ and $\dist(p_k,\Omega_p) < \frac{A_0}{30}\dist(p_k,\partial\Omega) = \frac{A_0}{30}s_k = \frac{A_0}{15}s_{k+1}$. But then
\begin{equation*}
    \dist(z,\Omega_p) \leq |z-p_k| + \dist(p_k,\Omega_p) \leq 6s_{k+1} + \frac{A_0}{15}s_{k+1} \leq \frac{A_0}{10}\dist(z,\partial\Omega)
\end{equation*}
so that $z\in\widehat{\Omega}_p$ by Lemma \ref{l:omega-cone} as long as $\dist(z,\Omega_p) \leq \frac{A_0}{10}\dist(g(W_p),\partial\Omega)$, which follows from the fact that $p < p_k$ ($p$ is not a net point of smaller scale).

We now prove \ref{i:omegas3}. Let $z\in \Omega\cap B(0,1)$ and define $\sC_z = \{p\in\sC : z\in\widehat{\Omega}_p\}$. It suffices to prove 
$$\#(\sC_z) \lesssim_{d} 1.$$ 
First, suppose $p\in\sC_z\cap \sF$. There exists $W\in\widehat{T}_p$ such that $z\in g_p(W)$ and the definition of $\widehat{\Omega}_p$ then implies that there exists $R_p\in T_p$ such that $R_p\simeq_{A_0} W$. Let $r_z = \dist(z,\partial\Omega)$. Lemmas \ref{l:gp-prop} and \ref{l:Dgp-change} imply that $\diam g_p(R_p) \simeq_{d} \dist(g_p(R_p),\partial\Omega) \simeq_{A_0} r_z$ and there exists $C_0(d,A_0) > 0$ such that
\begin{equation}\label{e:omega-contain-balls}
    B(g_p(c_{R_p}), C_0^{-1}r_z) \subseteq g_p(R_p) \subseteq B(z,C_0r_z).
\end{equation}
Since $\Omega_p\cap\Omega_{p'} = \varnothing$ for $p\not=p'$, we have
\begin{equation}\label{e:omega-disjoint-balls}
    g_p(R_p) \cap g_{p'}(R_{p'}) = \varnothing.
\end{equation}
It follows from \eqref{e:omega-contain-balls} and \eqref{e:omega-disjoint-balls} that $\#(\sC_z\cap\sF) \lesssim_{A_0,d} 1$. A similar argument shows that $\#(\sC_z\cap\sN)\lesssim_{d,A_0} 1$ from which the claim follows.

Item \ref{i:omegas4} follows from Lemma \ref{l:gp-prop} \ref{i:gp-dist}.
\end{proof}

\begin{remark}[Whitney family]\label{rem:gen-whit-fam}
    In fact, \eqref{e:omega-contain-balls} and \eqref{e:omega-disjoint-balls} in combination with Lemma \ref{l:gp-prop} show that there exists a constant $\Lambda_1(d)$ such that the family
    \begin{equation}\label{e:gen-whit-fam}
        \mathcal{G}_1 = \bigcup_{\substack{p\in\sF \\ W\in T_p}} g_p(W) \cup \bigcup_{q\in\sN} W_q.
    \end{equation}
    is a $\Lambda_1$-Whitney family in the sense of Definition \ref{def:whit-family} (compare with Lemma \ref{l:boxes-whitney}).
\end{remark}

We can now finish the proof of Proposition \ref{p:gen-Lipschitz}.

\begin{proof}[Proof of Proposition \ref{p:gen-Lipschitz}]
    We showed the existence of $L_1$ such that $\Omega_j$ is $L_1$-Lipschitz graphical for any $j\in J_{\scL_b}$ in Lemma \ref{l:gen-Lipschitz}. The fact that $\Omega_j \subseteq \Omega$ follows from Lemma \ref{l:omegas-are-nice} \ref{i:omegas4} while $\Omega\subseteq \bigcup_{j\in J_{\scL_b}}\Omega_j$ follows from Lemma \ref{l:omegas-are-nice} \ref{i:omegas2}. Finally, item \ref{i:prop-bounded-overlap} of Proposition \ref{p:gen-Lipschitz} follows from Lemma \ref{l:omegas-are-nice} \ref{i:omegas3} because for each $p\in\sC$, there is by definition at most one index $j_p$ such that $x\in\Omega_{j_p}\subseteq\widehat{\Omega}_p$.
\end{proof}

\subsection{Surface area bounds for Theorems \ref{t:thmB} and \ref{t:thmC}}
In this section, we prove Proposition \ref{p:gen-surface-msr}. The proof is similar to that of Proposition \ref{p:RF-surface-msr} given Remark \ref{rem:gen-whit-fam}. Fix $y\in\partial\Omega\cap B(0,1)$ and $0 < r \leq 1$ and let $A_2 = 100\sqrt{d}A_0^2,\ A_3 = 50\sqrt{d}A_1A_2$. If $p\in\sF$ is such that $\widehat{\Omega}_p\cap B(y,r)\not=\varnothing$, then there exists a cube $R$ with $\ell(R) \leq 2r$ such that $g_p(R)\cap B(y,r)\not=\varnothing$ and $g_p(R)\simeq_{A_0} W$ with $W\in T_p$. Then
$$\dist(g_p(W),y) \leq \dist(g_p(W),g_p(R)) + \diam g_p(R) \leq  A_0(1+A_0)\diam g_p(R) \leq 3\sqrt{d}A_0^2\ell(R) < 10\sqrt{d} A_0^2 r.$$ 
Therefore, since $A_2 > 50\sqrt{d}A_0^2$, we get that $\Omega_p\cap B(y,A_2r)\not=\varnothing$. We set
\begin{equation*}
    \scT_{y,A_2r}' = \{T_p : p\in \sF,\ \Omega_p\cap B(y,A_2r)\not=\varnothing\}.
\end{equation*}
The above discussion gives that $\widehat{\Omega}_p \cap B(y,r)\not=\varnothing \implies \Omega_p\cap B(y,A_2r)\not=\varnothing$, so it suffices to consider stopping time domains in the family $\scT_{y,A_2r}'$. Break up $\scT_{y,A_2r}'$ into regions with large and small top cubes:
\begin{align*}
    \scT_{L,A_2r} &= \{T_p\in\scT_{y,A_2r}' : h(W(T_p)) > 10A_2r\},\\
    \scT_{y,A_2r} &= \scT_{y,A_2r}'\setminus \scT_{L,A_2r}.
\end{align*}
We also collect all of the boundaries of domains in our decomposition $\scL$ (See \ref{e:def-of-Lambda}.) associated with a given flat point $p\in\sF$ into the set
\begin{equation}\label{e:flat-boundaries}
    \mathcal{B}_p = \bigcup_{j\in J_p}\partial\Omega^j_p \cup \bigcup_{W\in\widehat{\sD}_p\setminus\sD_p}g(\partial W).
\end{equation}
We note that $\sB_p$ is Ahlfors $d$-regular with constant depending on $d$ and $A_0$ by Proposition \ref{p:DT-lip-graph} and the fact that each cube $W\subseteq\widehat{\sD}_p\setminus\sD_p$ is $A_0$-close to a cube $W'\in T_p$ with at least one face inside $\partial\sD_p$.

We can then use the arguments of the previous section to get the following analogues of Lemmas \ref{l:large-top-cubes} and \ref{l:small-top-cubes}.

\begin{lemma}\label{l:gen-large-top-cubes}
We have
\begin{equation*}
    \sum_{\substack{p\in\sF \\ T_p\in\scT_{L,A_2r}}} \sH^d(\mathcal{B}_p\cap B(y,r)) \lesssim_d r^d \leq \sH^d(\partial\Omega\cap B(y,r)).
\end{equation*}
\end{lemma}
\begin{proof}
It follows from the proof of Lemma \ref{l:large-top-cubes} and the fact that $\mathcal{G}_1$ is a Whitney family (see Remark \ref{rem:gen-whit-fam}) that $\#(\scT_{L,A_2r})\lesssim_{A_2,d} 1$. Since $\mathcal{B}_p$ is Ahlfors $d$-regular, we have
\begin{equation*}
     \sum_{\substack{p\in\sF \\ T_p\in\scT_{L,A_2r}}}\sH^d(\mathcal{B}_p\cap B(y,r)) \lesssim_{A_0,d} \#(\scT_{L,A_2r})r^d \lesssim_{A_2,d} r^d.
\qedhere\end{equation*}
\end{proof}

We now handle the regions with small top boxes:
\begin{lemma}\label{l:gen-small-top-cubes}
We have
\begin{equation}\label{e:gen-small-top-cubes}
    \sum_{\substack{p\in\sF \\ T_p\in\scT_{y,A_2r}}} \sH^d(\mathcal{B}_p\cap B(y,r)) \lesssim_{d,\epsilon} \sH^d(\partial\Omega\cap B(y,A_3r)).
\end{equation}
\end{lemma}
\begin{proof}
We modify the proof of Lemma \ref{l:small-top-cubes}. We first observe that since $\sH^d(\mathcal{B}_p) \lesssim_{A_0,d}\sH^d(\partial\Omega_p)$, we have
$$\sum_{\substack{p\in\sF \\ T_p\in\scT_{y,A_2r}}} \sH^d(\mathcal{B}_p\cap B(y,r)) \leq \sum_{\substack{p\in\sF \\ T_p\in\scT_{y,A_2r}}} \sH^d(\mathcal{B}_p) \lesssim_{A_0,d} \sum_{\substack{p\in\sF \\ T_p\in\scT_{y,A_2r}}} \sH^d(\partial\Omega_p).$$
Therefore, it suffices to prove \eqref{e:gen-small-top-cubes} with $\mathcal{B}_p\cap B(y,r)$ replaced by $\partial\Omega_p$. For any $p\in\sF$ and $T_p\in\scT_{y,A_2r}$, we get
\begin{equation}\label{e:gen-partial-omega-bound}
    \sH^d(\partial\Omega_p) \lesssim_d \sH^d(\partial\Omega_p\cap\partial\Omega) + \sum_{W\in m(T_p)} h(W)^d.
\end{equation}
Now, $W\in m(T_p)$ implies that there exists a child $W'\in\Stop(T_p)$ for which we have $Q\in \scB_e$ of \eqref{e:bad-scD} with $g(W')\simeq_{A_1} Q$ by Lemma \ref{l:scB-packing}. By replacing $A_0$ with $A_1$ and $r$ with $A_2r$ in \eqref{e:final-inflated-ball}, we get $Q\subseteq B(y,50\sqrt{d}A_1A_2r) \subseteq B(y,A_3r)$. Hence, applying Lemma \ref{l:whit-card} with $\scV = \{g(W): W\in m(T_p),\ T_p\in\scT_{y,A_2r}\}$ and $\scU = \{Q\in\scB_e: Q\subseteq B(y,A_3r)\}$, we get
\begin{align}
    \sum_{\substack{p\in\sF \\ T\in\scT_{y,A_2r}}}\sum_{W\in m(T_p)} h(W)^d \lesssim_{d} \sum_{\substack{Q\in\scB_e \\ Q\subseteq B(y,A_3r)}}\ell(Q)^d \lesssim_{d,\epsilon}\sH^d(\partial\Omega\cap B(y,A_3r))
\end{align}
where the last inequality follows from the Carleson packing condition for $\scB_e$. By observing that $\partial\Omega_p\cap\partial\Omega\subseteq B(y,50\sqrt{d}A_2r)$ and $\sH^d(\partial\Omega_p\cap\partial\Omega_{p'}\cap\partial\Omega) = 0$ for any $p\not=p'$, \eqref{e:gen-partial-omega-bound} implies 
\begin{equation*}
    \sum_{T_p\in\scT_{y,A_2r}}\sH^d(\partial\Omega_p) \lesssim_{d,\epsilon} \sH^d(\partial\Omega \cap B(y,A_3r)).
\qedhere\end{equation*}
\end{proof}

We also need to bound the surface measure associated to trivial domains around non-flat $q\in\sN$. For any $q\in\sN$, we define the set of boundaries
\begin{equation*}
    \sB_q = \partial W_q \cup \bigcup_{\substack{W\in\widetilde{\scW} \\ W\subseteq \widehat{\sD}_q\setminus\sD_q}}\partial W.
\end{equation*}
We note that $\sH^d(\sB_q) \lesssim_{d,A_0} \ell(W_q)^d$.
\begin{lemma}\label{l:gen-non-flat}
We have
    \begin{equation*}
        \sum_{\substack{q\in\sN}}\sH^d(\sB_q\cap B(y,r)) \lesssim_{d,\epsilon} \sH^d(\partial\Omega\cap B(y,A_3r)).
    \end{equation*}
\end{lemma}
\begin{proof}
    Observe that $\sB_q\cap B(y,r)\not=\varnothing$ implies there exists $Q\in\scB_e$ such that $W_q \simeq_{10A_1} Q$ and $Q\subseteq B(y,A_3r)$ so that we have
    \begin{equation*}
        \sum_{\substack{q\in\sN}}\sH^d(\sB_q\cap B(y,r)) \leq \sum_{\substack{q\in\sN \\ \sB_q\cap B(y,r)\not=\varnothing}} \ell(W_q)^d \lesssim_{A_1,d} \sum_{\substack{Q\in\scB_e \\ Q\subseteq B(y,A_3r)}}\ell(Q)^d \lesssim_{d,\epsilon} \sH^d(\partial\Omega \cap B(y,A_3r)).
    \qedhere\end{equation*}
\end{proof}

\begin{proof}[Proof of Proposition \ref{p:gen-surface-msr}]
    If $\Omega_j\in\scL_b$, then either there exists $p\in\sF,\ j_0\in J_p$ such that $\Omega_j = \Omega_{p}^{j_0}$ or there exists $q\in\sN$ such that $\Omega_j = R\in\widetilde{W}$ where $R\simeq_{A_0} W_q$. This means that
    \begin{align*}
        \sum_{j\in j_{\scL_b}}&\sH^d(\partial\Omega_j\cap B(y,r)) \\
        &\leq \sum_{p\in\sF}\left[\sum_{T_p\in\scT_{L,A_2r}}\sum_{j\in J_p}\sH^d(\partial\Omega^j_p\cap B(y,r)) + \sum_{T_p\in\scT_{y,A_2r}}\sum_{j\in J_p}\sH^d(\partial\Omega_p^j\cap B(y,r))\right]\\
        &\quad\quad+\sum_{q\in\sN}\sH^d(\sB_q\cap B(y,r))\\
        &\lesssim \sum_{p\in\sF}\left[\sum_{T_p\in\scT_{L,A_2r}} \sH^d(\sB_p\cap B(y,r)) + \sum_{T_p\in\scT_{y,A_2r}} \sH^d(\sB_p\cap B(y,r))\right] + \sum_{q\in\sN}\sH^d(\sB_q\cap B(y,r))\\
        &\lesssim_{d,\epsilon} \sH^d(\partial\Omega\cap B(y,A_3r))
    \end{align*}
    by Lemmas \ref{l:gen-large-top-cubes}, \ref{l:gen-small-top-cubes}, and \ref{l:gen-non-flat}.
\end{proof}
\section{Further questions}\label{sec:problems}
\begin{problem}
Can one find a disjoint decomposition rather than one of bounded overlap in Theorems \ref{t:thmB} and \ref{t:thmC}?
\end{problem}
In the construction given in this paper for Theorems \ref{t:thmB} and \ref{t:thmC}, the only overlap between domains occurs in intersections between $(1+C\delta)$-bi-Lipschitz images of Whitney cubes of comparable side length inside the ``buffer zones''. It seems plausible that one could devise a different scheme to divide the space between the disjoint ``core'' stopping time domains into Lipschitz graph domains.

Similarly, it seems plausible that one could obtain a similar result by modifying the methods of the proof to prove a version of Theorem \ref{t:thmA} with assumption \ref{item:beta-squared} removed (see Remark \ref{rem:better-refine}).

\begin{problem}
Can Theorem \ref{t:thmB} be extended to general lower-content $d$ regular sets?
\end{problem}
It seems possible that the necessary tools to handle this extension to non-Reifenberg flat sets are present in the $\beta$-number estimates in \cite{AS18} and \cite{Hyd22}. The technical disconnect between this paper and those ones is caused by the ``smoothing'' procedure needed there in defining the stopping time regions makes this generalization non-obvious to the author. It seems that a proof would require new ideas.

\appendix

\section{Graph coronizations for Reifenberg flat sets}\label{sec:RF-appendix}
The goal of this section is to provide a proof of Proposition \ref{p:RF-graph-coronization} which states that there exist (sufficiently small in terms of $d$) constants $\epsilon, \delta > 0$ such that Reifenberg flat sets admit $(M,\epsilon,\delta)$-graph coronizations. 

Reifenberg flat sets are a subset of a more general class of sets called \textit{lower content $d$-regular sets} studied by Azzam and Schul \cite{AS18} and later Hyde \cite{Hyd21} as a class of objects for $d$-dimensional traveling salesman results.
\begin{definition}[lower content $d$-regularity]
    A set $E\subseteq \R^{d+1}$ is said to be \textit{lower content $d$-regular} in a ball $B(x,r)$ if there exists a constant $c > 0$ and $r_B > 0$ such that
\begin{equation*}
    \sH^d_\infty(E\cap B(x,r)) \geq cr^d\ \text{for all $x\in E\cap B$ and $r\in(0,r_B)$}.
\end{equation*}
A set $E$ is \textit{lower content $d$-regular} if there exists a constant $c$ such that $E$ is lower content regular with constant $c$ in every ball centered on $E$.
\end{definition}
Since a Reifenberg flat set $\Sigma$ satisfies $b\beta_\Sigma(B) \leq \epsilon$ in every ball by definition, the only remaining requirements for the existence of a graph coronization are control over $\beta^{d,1}_\Sigma$-squared sums and control over the frequency of angle turning of well-approximating planes. The necessary control over $\beta$-sums is contained in the following traveling salesman theorems formulated for general lower content regular sets:
\begin{theorem}[\cite{Hyd21} Theorem 1.6]\label{t:hyde-lcr-beta-bound}
Let $H$ be a Hilbert space and $1 \leq d < \dim(H),\ 1 \leq p < p(d),\ C_0 > 1,\ $and $A > 10^5$. Let $E\subseteq H$ be a lower content $d$-regular set with regularity constant $c$ and Christ-David cubes $\mathscr{D}$. There exists $\epsilon > 0$ small enough so that the following holds. Let $Q_0\in\mathscr{D}$ and
\begin{equation*}
    \beta_{E,C_0,d,p}(Q_0) = \ell(Q_0)^d + \sum_{Q\subseteq Q_0}\beta_E^{d,p}(C_0B_Q)^2\ell(Q)^d.
\end{equation*}
Then
\begin{equation}\label{e:lcr-beta-bound}
    \beta_{E,C_0,d,p}(Q_0) \lesssim_{A,d,c,p,C_0,\epsilon} \sH^d(Q_0) + \BWGL(Q_0,A,\epsilon).
\end{equation}
\end{theorem}

\begin{theorem}[\cite{Hyd21} Theorem 1.7]
Let $H$ be a Hilbert space, $1 \leq d < \dim(H),\ 1 \leq p < \infty,\ A>1,\ \epsilon > 0,$ and $C_0 > 2\rho^{-1}$ where $\rho$ is as in the construction of the Christ-David lattice $\mathscr{D}$. Let $E\subseteq H$ be lower content $d$-regular with regularity constant $c$ and Christ-David cubes $\mathscr{D}$. For $Q_0\in\mathscr{D}$, we have
\begin{equation*}
    \sH^d(Q_0) + \BWGL(Q_0,A,\epsilon) \lesssim_{A,d,c,C_0,\epsilon}\beta_{E,C_0,d,p}(Q_0).
\end{equation*}
\end{theorem}

If $E$ is $(\epsilon,d)$-Reifenberg flat, then the $\BWGL$ terms above vanish and \eqref{e:lcr-beta-bound} gives a Carleson packing condition for the content beta number sum reminiscent of the strong geometric lemma for uniformly rectifiable sets from which we will conclude the desired $\beta^2$ sum control.

We will require small technical tweaks of the stopping time machinery of Azzam and Schul on Reifenberg flat sets. We review the necessary definitions here, but refer to \cite{AS18} sections 5-8 for a full treatment of the construction. 
\begin{definition}[$d$-dimensional traveling salesman stopping time]
    We fix constants $0 < \epsilon \ll \alpha^4$ with $\alpha(d),\epsilon(d)$ to be chosen sufficiently small in terms of $\delta$ as required in \cite{AS18} . For any cube $Q\in\scD$, we define a stopping time region $S_Q^\alpha$ by adding cubes $R\subseteq Q$ to $S_Q$ if
\begin{enumerate}[label=(\roman*)]
    \item $R^{(1)}\in S_Q^\alpha$,
    \item $\ang(P_U,P_Q) < \alpha$ for any sibling $U$ of $R$ (including $R$ itself).
\end{enumerate}
For any collection of cubes $\scQ$, define a distance function
\begin{equation*}
    d_{\scQ}(x) = \inf\{\ell(Q) + \dist(x,Q) : Q\in\scQ\}.
\end{equation*}
For any $Q\in\scD$, define
\begin{equation*}
    d_\scQ(Q) = \inf_{x\in Q}d_\scQ(x) = \inf\{\ell(R) + \dist(Q,R) : R\in\scQ\}.
\end{equation*}
We let $m(S)$ be the set of minimal cubes of $S$, those which have no children contained in $S$ and define
\begin{equation*}
    z(S) = Q(S) \setminus \bigcup_{Q\in m(S)}Q.
\end{equation*}
Let
\begin{equation*}
    \Stop(-1) = \scD_0
\end{equation*}
and fix a small constant $\tau\in(0,1)$. Suppose we have defined $\Stop(N-1)$ for some integer $N\geq 0$ and define
\begin{align*}
    \Layer(N) &= \bigcup\{S_Q^\alpha : Q\in\Stop(N-1)\}.
\end{align*}
We then set $\Up(-1) = \varnothing$ and put
\begin{align*}
    \Stop(N) &= \{Q\in\scD : \text{$Q$ maximal such that $Q$ has a sibling $Q'$ with $\ell(Q') < \tau d_{\Layer(N)}(Q')$}\},\\
    \Up(N) &= \Up(N-1) \cup \{Q\in\scD : Q\supset R \text{ for some $R\in\Stop(N)\cup\Layer(N)$}\}
\end{align*}
\cite{AS18} Lemma 5.5 says that, in fact
\begin{equation*}
    \Up(N) = \{Q\in\scD : Q\not\subset R \text{ for any $R\in\Stop(N)$}\}.
\end{equation*}
\end{definition}
Essentially, $\Layer(N)$ is a layer of stopping time regions $S_Q^\alpha$ beginning at the stopped cubes of the previous generation and continuing until reaching a cube $R$ with a child $R'$ such that $\ang(P_Q,P_{R'}) > \alpha$. $\Stop(N)$ is formed by taking a ``smoothing'' of $\Layer(N)$ that ensures that nearby minimal cubes in $\Stop(N)$ are of similar size. One forms a CCBP from the centers and $b\beta$-minimizing planes of cubes in $\Up(N)$ which gives a surface $\Sigma_N$ for any $N\geq 0$ which converges to $\Sigma$ as $N\rightarrow\infty$. Azzam and Schul give tools for proving bounds on the degree of stopping in this construction in the following lemma
\begin{lemma}[\cite{Hyd21} Lemma 4.4 (5)]\label{l:reif-flat-angle-turn}
Let $\Sigma$ be $(\epsilon,d)$-Reifenberg flat and $\scD$ a Christ-David lattice for $\Sigma$. Let $N \geq 0$. For any $Q_0\in\scD$,
\begin{equation*}
    \sum_{N\geq 0}\sum_{\substack{Q\in\Stop(N) \\ Q\subseteq Q_0}}\ell(Q)^d \lesssim_{d,\alpha,\epsilon}\sH^d(Q_0).
\end{equation*}
\end{lemma}
\begin{proof}[Proof of Proposition \ref{p:RF-graph-coronization}]
    Fix $M \geq 1,$ and $ \epsilon,\alpha > 0$ sufficiently small in terms of $M,d,n$ determined by Lemma \ref{l:reif-flat-angle-turn} and Theorem \ref{t:hyde-lcr-beta-bound} and let $\delta = 100\alpha$. Let $\scD$ be a Christ-David lattice for $\Sigma$ and let $\{P_Q\}_{Q\in\scD}$ be a family of $d$-planes such that $x_Q\in P_Q$ and $\beta^{d,1}_\Sigma(MB_Q,P_Q) \leq 2\beta^{d,1}_\Sigma(MB_Q)$. Fix $Q_0\in \scD$ and form a collection of stopping time regions $\scF = \{S_Q\}$ contained within $Q_0$ satisfying the stopping conditions Items \ref{i:gc-beta} and \ref{i:gc-angle} of Definition \ref{def:graph-coronization}. We set $\scG = \scD$, $\scB = \varnothing$. To prove that $\scC = (\scG,\scB,\scF)$ is an $(M,\epsilon,\delta)$-graph coronization, we only need to show that $\scC$ is a coronization, i.e., there exists a constant $C(M,\epsilon, \delta, d)$ such that 
    \begin{equation*}
        \sum_{S\in\scF}\ell(Q(S))^d \leq C\sH^d(Q_0).
    \end{equation*}
    Define
    \begin{align*}
        S_\delta &= \{Q\in\scD : \exists S\in\scF,\ Q\in\Stop(S),\ \ang(P_Q,P_{Q(S)}) > \delta\},\\
        S_\beta &= \left\{Q\in\scD : \exists S\in\scF,\ Q\in\Stop(S),\ \sum_{Q\subseteq R\subseteq Q(S)}\beta^{d,1}_\Sigma(MB_R)^2 > \eta\right\}.
    \end{align*}
    It suffices to show that $\sum_{Q\in S_\delta\cup S_\beta}\ell(Q)^d \leq C\sH^d(Q_0)$. We define
\begin{equation*}
    \Stop(-1,\delta) = \{Q_0\},
\end{equation*}
and, given $\Stop(N-1,\delta)$ for some integer $N\geq 0$, we define
\begin{equation*}
    \Stop(N,\delta) = \{R\in S_\delta : R \text{ maximal such that } R\subseteq Q\in\Stop(N-1,\delta)\}.
\end{equation*}
With this, we have
\begin{equation*}
    S_\delta = \bigcup_{N\geq 0}\Stop(N,\delta).
\end{equation*}
We will use this to show that $\sum_{Q\in S_\delta}\ell(Q)^d \leq C\sH^d(Q_0).$

Fix $Q\in\Stop(N,\delta)$ and let $x\in Q\setminus z(S)$. Then there exists $R\in S_\delta,\ R\subset Q$ such that $x\in R$ and, since $\delta \geq 100\alpha$, there exists a cube $R'\in\Stop(K)$ for some $K \geq 0$ such that $R\subseteq R' \subseteq Q$. Set
\begin{equation*}
    \Stop(Q) = \left\{R\in\scD : \text{$R$ maximal such that $R\in\bigcup_{N\geq0}\Stop(N)$ and $R\subseteq Q$}\right\}.
\end{equation*}
The above argument has shown that $Q\setminus z(S_Q)\subseteq \cup_{R\in\Stop(Q)}R$. We see
\begin{equation*}
    \ell(Q)^d \lesssim_d \sum_{R\in\Stop(Q)}\ell(R)^d + \sH^d(z(S_Q)).
\end{equation*}
This means
\begin{align*}
    \sum_{Q\in S_\delta}\ell(Q)^d &= \sum_{N\geq 0}\sum_{Q\in\Stop(N,\delta)}\ell(Q)^d \\
    &\lesssim_d \sum_{N\geq 0}\sum_{Q\in\Stop(N,\delta)}\left( \sum_{R\in\Stop(Q)}\ell(R)^d + \sH^d(z(S_Q)) \right)\\
    &\lesssim \sH^d(Q_0) + \sum_{K\geq 0}\sum_{R\in\Stop(K)}\ell(R)^d\\
    &\lesssim_{d,\delta,\epsilon} \sH^d(Q_0)
\end{align*}
where the penultimate line follows from the fact that $\Stop(Q)\cap\Stop(Q')=\varnothing$ for $Q,Q'\in S_\delta,\ Q\not=Q'$, and the final line follows from Lemma \ref{l:reif-flat-angle-turn}.

We now show that $\sum_{Q\in S_\beta} \ell(Q)^d \leq C\sH^d(Q_0)$. We have
\begin{align*}
    \sum_{Q\in S_\beta}\ell(Q)^d &\leq \sum_{Q\in S_\beta}\ell(Q)^d\bigg[\epsilon^{-2}\sum_{\substack{Q\subseteq R \subseteq Q(S) \\ Q\in S\in\scF}}\beta^{d,1}_\Sigma(MB_R)^2\bigg] = \epsilon^{-2}\sum_{R\in\scD}\beta^{d,1}_\Sigma(MB_R)^2\sum_{\substack{Q\text{ maximal }\subseteq R \\ Q\in S_\beta}}\ell(Q)^d,\\
    &\lesssim \epsilon^{-2}\sum_{R\in\scD}\beta^{d,1}_\Sigma(MB_R)^2\ell(R)^d \lesssim_{d,\eta} \sH^d(Q_0)
\qedhere\end{align*}
using Theorem \ref{t:hyde-lcr-beta-bound} in the last line.
\end{proof}

\noindent\textbf{Acknowledgements}

\vspace{1ex}

\noindent The author thanks Raanan Schul for many helpful discussions and suggestions. This work would not have been possible without his original inquiry into the problem and his general guidance on attacking it.

The author thanks the referee for many useful suggestions and for catching numerous typos and small inaccuracies in the initial draft. Their work significantly improved the readability of this paper. The author also thanks a reviewer for mentioning additional background literature on harmonic measure and for explanations that improved the presentation of the results in \cite{MT24} and \cite{MPT22}.

\bibliographystyle{alpha}
\bibliography{bib-file-RST}

\end{document}